\documentclass{jams-l}

\pdfoutput1
\usepackage{tikz}
\usetikzlibrary{matrix,arrows}
\usepackage{tikz-cd}

\makeatletter
\def\@seccntformat#1{%
  \protect\textup{%
    \protect\@secnumfont
    \expandafter\protect\csname format#1\endcsname % <--- added
    \csname the#1\endcsname
    \protect\@secnumpunct
  }%
}

\makeatother

\usepackage{longtable}
\usepackage{amssymb}
\usepackage{mathtools}
\usepackage{hyperref}
\usepackage{mathrsfs}
\usepackage{color}

\usepackage{enumerate}

\usepackage{thmtools}
\usepackage{thm-restate}

\usepackage{cleveref}

\usepackage{stackengine,amsmath}

\newcommand{\xrightcong}{\xrightarrow{\raisebox{-0.4ex}[0ex][0ex]{$\,
     \scriptstyle \cong\, $}}}

\newcommand{\xrightcongdbl}{\xrightarrow{\raisebox{-0.4ex}[0ex][0ex]{$\,
      \scriptstyle \cong\, $}}\mathrel{\mkern-14mu}\rightarrow}

\newcommand{\xrightsimeq}{\xrightarrow{\raisebox{-0.2ex}[0ex][0ex]{$\,
      \scriptstyle\simeq\, $}}}

\newcommand{\xleftcong}{\xleftarrow{\raisebox{-0.4ex}[0ex][0ex]{$\,
      \scriptstyle \cong\, $}}}
\newcommand{\xleftsimeq}{\xleftarrow{\raisebox{-0.2ex}[0ex][0ex]{$\,
      \scriptstyle\simeq\, $}}}

\newcommand{\xrightncong}{\xrightarrow{\raisebox{-0.4ex}[0ex][0ex]{$\,
      \scriptstyle {\not\cong}\, $}}}

\newcounter{thecounter}
\numberwithin{thecounter}{section}
\newtheorem{lemma}[thecounter]{Lemma}
\newtheorem{prop}[thecounter]{Proposition}
\newtheorem{thrm}[thecounter]{Theorem}
\newtheorem{cor}[thecounter]{Corollary}

\newtheorem{Th}{Theorem}

\newtheorem{Co}[Th]{Corollary}

\theoremstyle{definition}

\newtheorem{example}[thecounter]{Example}

\newtheorem{rem}[thecounter]{Remark}

\theoremstyle{remark}

\numberwithin{equation}{section}

\newcommand{\sd}{\operatorname{sd}}

\newcommand{\sta}{\operatorname{\bf star}}
\newcommand{\link}{ \operatorname{\bf link}}

\newcommand{\colim}{\operatorname{colim}}

\newcommand{\Aut}{\mathrm{Aut}}
\newcommand{\Rep}{\operatorname{Rep}}
\newcommand{\Inn}{\operatorname{Inn}}
\newcommand{\Out}{\operatorname{Out}}
\newcommand{\End}{\operatorname{End}}
\newcommand{\rk}{\operatorname{rk}}

\newcommand{\Hom}{\operatorname{Hom}}
\newcommand{\sHom}{{\operatorname{\underbar{Hom}}}}
\newcommand{\im}{\operatorname{im}}
\newcommand{\mor}{\operatorname{Mor}}

\newcommand{\id}{\operatorname{id}}
\newcommand{\res}{\operatorname{res}}

\newcommand{\diag}{\operatorname{diag}}
\newcommand{\Rmod}{R\mbox{--}\mathrm{mod}}

\newcommand{\hocolim}{\operatorname{hocolim}}

\newcommand{\fF}{{\mathfrak F}}
\newcommand{\calF}{{\mathcal F}}

\newcommand{\scrF}{{\mathscr F}}

\newcommand{\calC}{{\mathcal{C}}}
\newcommand{\calX}{{\mathcal{X}}}
\newcommand{\calY}{{\mathcal{Y}}}

\newcommand{\calB}{{\mathcal{B}}}
\newcommand{\calA}{{\mathcal{A}}}

\newcommand{\bO}{{\mathscr O}}

\newcommand{\coker}{\operatorname{coker}}
\newcommand{\op}{\operatorname{op}}

\newcommand{\GL}{\operatorname{GL}}

\newcommand{\Sp}{\operatorname{Sp}}

\newcommand{\SL}{\operatorname{SL}}
\newcommand{\PSL}{\operatorname{PSL}}

\newcommand{\Tr}{\operatorname{Tr}}

\newcommand{\luft}{\medskip\par\noindent}
\newcommand{\QED}{\hfill {\bf $\square$}\luft}

\newcommand{\Q}{{\mathbb {Q}}}

\newcommand{\F}{{\mathbb {F}}}
\newcommand{\Z}{{\mathbb {Z}}}

\newcommand{\D}{{\mathbf D}}

\newcommand{\A}{{\mathbf A}}

\newcommand{\beq}{\begin{eqnarray*}}
\newcommand{\eeq}{\end{eqnarray*}}

\newcommand{\tuborg}{\left\{\begin{array}{ll}}
\newcommand{\sluttuborg}{\end{array}\right.}

\newfont{\bm}{msbm10}
\newcommand{\semi}{{\mbox{{\bm \symbol{111}}}}}

\newcommand{\Ob}{\operatorname{Ob}}
\newcommand{\Mor}{\operatorname{Mor}}

\input xypic
\xyoption{all}
\CompileMatrices

\newcommand{\cA}{{\mathcal A} }

\renewcommand{\phi}{\varphi}

\newcommand{\calS}{{\mathcal{S}}}
\newcommand{\calT}{{\mathscr{T}}}

\newcommand{\calG}{{\mathcal{E}}}

\newcommand{\stmod}{\operatorname{stmod}}
\newcommand{\Stmod}{\operatorname{StMod}}

\newcommand{\Oep}{{{\mathscr O}_p^*}}
\newcommand{\Op}{{\mathscr O}_p}

\newcommand{\Tep}{{{\mathscr T}_p^*}}

\newcommand{\G}{{\mathbf G}}

\newcommand{\sym}{{\mathfrak S}}
\newcommand{\alt}{{\mathfrak A}}

\newcommand{\TS}{T_k(G,S)}

\newcommand{\Pic}{\operatorname{Pic}}

\newcommand{\Tors}{\operatorname{Tors}}

\newcommand{\atlas}{\mathbb A \mathbb T  \mathbb L  \mathbb A  \mathbb
  S}

\newcommand{\bcalF}{{\bar{\mathcal F}}}

\newcommand{\Dperf}{D^{{\mathrm{perf}}}}

\def\co{\colon\thinspace}

\newcommand{\thmunderline}[1]{%
  \sc{{#1}}%
}

\allowdisplaybreaks

\begin{document}
\title{Endotrivial modules for finite groups via homotopy theory}

%\date{\today}

\author{Jesper Grodal}

\thanks{Supported by the Danish National Research Foundation through
  the Centre for Symmetry and Deformation (DNRF92). % \today
}


\address{Department of Mathematical Sciences, University of
Copenhagen, Denmark}
\email{jg@math.ku.dk}

\begin{abstract} Classifying endotrivial
  $kG$--modules, i.e., elements of the Picard group of the stable module
  category for an
  arbitrary finite group $G$, has been a long-running quest. By
  deep work of Dade, Alperin, Carlson, Th\'evenaz, and others, it has
been reduced to understanding the subgroup consisting of modular
representations that split as the trivial module $k$ direct sum a
projective module when restricted
to a Sylow $p$--subgroup.
In this paper we
identify this subgroup as the first cohomology group of the orbit
category on non-trivial $p$--subgroups with values in the units
$k^\times$, viewed as a constant
coefficient system. 
We then use homotopical techniques to give a
number of formulas for this group in terms of 
the abelianization of normalizers and centralizers in $G$, in particular
verifying the Carlson--Th\'evenaz conjecture---this reduces the calculation
of this group to algorithmic calculations in local group theory rather than
representation theory.
We also provide strong
restrictions on when such representations of dimension greater than one can occur,
in terms of the $p$--subgroup complex and $p$--fusion systems. 
We immediately recover and extend a large
number of computational results in the literature,
and further illustrate the computational potential by calculating the group in other sample new
cases, e.g., for the Monster at all primes. 
\end{abstract}

\maketitle

\setcounter{tocdepth}{1}
\tableofcontents

\section{Introduction}
In modular representation theory of finite groups, the
indecomposable $kG$--modules $M$ whose restriction
to a Sylow $p$--subgroup $S$ split 
as the trivial module $k$ plus a
free $kS$--module are basic yet somewhat mysterious objects. Such modules
form a group $\TS$  under tensor product, discarding projective
$kG$--summands, with neutral element $k$ and $M^*$ the inverse of $M$
(see \S\ref{conventions} for details). It is also denoted $K(G)$ in the literature.
It contains the one-dimensional
characters $\Hom(G,k^\times)$ as a subgroup, but has been observed to
sometimes also contain {\em exotic} elements, i.e., modules of dimension greater
than one. The group $\TS$ is an important subgroup 
of the larger group of all so-called endotrivial modules $T_k(G)$,
i.e., $kG$--modules $M$ where $M^* \otimes M \cong k \oplus P$,
for $P$ a projective $kG$--module. Namely $\TS =
\ker\left(T_k(G) \to T_k(S)\right)$, i.e., the kernel of the restriction to
$S$. The group of endotrivial modules has a categorical interpretation as
$T_k(G) \cong \Pic(\Stmod_{kG})$, the Picard group of the stable
module category. Such modules
occur in many parts of representation theory, e.g., as source modules
(see the surveys
\cite{thevenaz07, carlson12, carlson17} and the papers quoted below). 

Classifying endotrivial modules has
been a long-running quest, which has been reduced to calculating
$\TS$, through a series of fundamental papers: The group
$T_k(S)$ was determined in celebrated works of Dade
\cite{dade78i,dade78ii}, Alperin \cite{alperin01}, and
Carlson--Th\'evenaz \cite{CT04,CT05}. From this, Carlson--Mazza--Nakano--Th\'evenaz
\cite{CMN06,MT07,CMT13} worked out the image of the restriction
$T_k(G)\mkern-1mu \to T_k(S)$, at least as
an abstract abelian group, and showed that the restriction is split onto its
image, except in well-understood cases with cyclic Sylow
$p$--subgroup (see \eqref{Lsequence} later in the introduction).
Subsequently there has been an intense
interest in calculating $\TS$,
with contributors Balmer, Carlson, Lassueur, Malle, Mazza,
Nakano, Navarro, Robinson, Th\'evenaz, and others.

\smallskip

In this paper we give an elementary and computable homological description
of the group $\TS$, 
as the first cohomology group of the orbit
category on non-trivial $p$--subgroups of $G$, with constant coefficients in
$k^\times$, for any finite group $G$ and any field $k$ (not assumed algebraically closed) of characteristic
$p$ dividing the order of $G$. 
Using homological methods, adapted from mod $p$ homology decompositions
(but now with
$k^\times$--coefficients, hence ``prime to $p$''), we deduce a range of structural and computational results on
$\TS$, with answers expressed in terms of standard $p$--local group theory: We write $\TS$ as an inverse limit of homomorphisms from normalizers of
chains of $p$--subgroups  to $k^\times$, answering the main
conjecture of
Carlson--Th\'evenaz \cite[Ques.~5.5]{CT15} in the positive (see also
\cite[Ques.~1]{carlson17}).
A related ``centralizer decomposition'' expresses it in
terms of the $p$--fusion system of $G$ and centralizers of elementary
abelian $p$--subgroups  
(see \S\S\ref{homologydecomp-subsection}--\ref{ctconj-sec}). 
We also get a formula for $\TS$ in terms of $\pi_0$ and $\pi_1$ of the
$p$--subgroup complex $|\calS_p(G)|$ of $G$. It implies that $\TS \cong \Hom(G,k^\times)$ when the $p$--subgroup complex
of $G$ is simply connected. The formula can be seen as a topological correction to the
old hope (too na\"ive, see \cite[p.~106]{carlson12}), that exotic
modules could only occur in the presence of a strongly
$p$--embedded subgroup, meaning $|\calS_p(G)|$ disconnected (see
\S\ref{subgroupcpx-subsection}).
We get bounds on $\TS$ in terms of the fundamental group of the $p$--fusion
system of $G$, and see the contribution of specific $p$--subgroups in $G$
 (see \S\ref{fusion-subsec}). 
Lastly we provide consequences of these results for specific
classes of groups, e.g., finite groups of Lie type and sporadic
groups, obtaining new computations as well as recovering and simplifying
many old ones in the vast literature. As an example we try out
one of our formulas on the Monster sporadic simple group, and easily
calculate $\TS$ for $p$ any of the harder primes $3,5,7,11,$ and $13$, which
had been left open in the literature \cite{LM15sporadic} (see
\S\ref{computations-subsection}).

Our proof of the identification of $\TS$ is direct and self-contained, and provides  ``geometric'' models for the module
generator in $\TS$ corresponding to a $1$--cocycle: it is the class in the stable module
category represented by the unreduced Steinberg complex
of $G$ twisted by the $1$-cocycle. It has the further conceptual
interpretation as the homotopy left Kan extension of the
$1$--cocycle from the orbit category on non-trivial $p$--subgroups, to
all $p$--subgroups. 
Our identification was originally inspired by a
characterization due to Balmer of $\TS$ in terms of what he dubs  ``weak
homomorphisms'' \cite{balmer13}, and we also indicate another
argument of how to deduce the identification using these.

\smallskip

Let us now describe our work in detail. Call a $kG$--module {\em Sylow-trivial}
if it upon restriction to $kS$ splits as the
trivial module $k$ plus a projective $kS$--module. $\TS$
identifies with the group of equivalence classes of
Sylow-trivial modules, identifying two if they become isomorphic after discarding projective
$kG$--summands. Each equivalence class
contains a unique indecomposable representative, up to isomorphism
(see Proposition~\ref{onedim}).
Let $\Oep(G)$ denote the orbit category of $G$ with objects $G/P$, for $P$
a non-trivial $p$--subgroup, and morphisms $G$--maps.
The following is our main identification of $\TS$.

\begin{Th}
\label{main} Fix a finite
  group $G$ and $k$ a field of characteristic $p$ dividing the order
  of $G$.
The group $\TS$ 
is described via the following
  isomorphism of abelian groups
$$\Phi\co \TS \xrightcong H^1(\Oep(G);k^\times)
$$
The map $\Phi$ sends  $[M] \in \TS$ to the functor $\phi\co \Oep(G) \to
k^\times$ 
defined as follows: First consider
$${\textstyle k_\phi\co G/P \mapsto \hat H^0(P;M) =
M^P/(\sum_{g \in P} g)M,}$$
given by zeroth Tate cohomology, a functor from $\Oep(G)$ to the
connected groupoid of one-dimensional
$k$--modules and isomorphisms.
Then identify the target with the group $k^\times$, regarded as a
category with one object, via an equivalence of categories, to obtain a functor $\phi\co \Oep(G) \to
k^\times$, well-defined up to natural isomorphism of functors. 

The inverse $\Phi^{-1}(\phi)$ is the
class in $\TS$ of the unaugmented ``twisted Steinberg
complex''
$$C_*(|\calS_p(G)|;k_\phi) \in D^b(kG)/\Dperf(kG) \xleftsimeq \stmod_{kG} $$ 
where $k_\phi$ is the $G$--twisted coefficient system
induced by  $\phi\co \Oep(G) \to k^\times$, i.e., assigning
$k_\phi(G/P_0)$ ($\cong k$) to the $n$--simplex $(P_0 \leq \cdots
\leq P_n)$ and endowing the chain complex with the canonical
$G$--action (see \S\ref{coeff-subsub}).
\end{Th}

In fact we establish in Theorem~\ref{sylowss} a more general correspondence between $kG$--modules
that split as a sum of trivial and projective modules upon
restriction to $S$, and representations of the fundamental group
$\pi_1(\Oep(G))$, refining parts of Green correspondence.

The one-dimensionality of $\hat H^0(P;M)$ is by definition of
Sylow-trivial, as $P$ is subconjugate to $S$.
Also, we used the equivalence of homotopy categories $\stmod_{kG} \xrightsimeq D^b(kG)/\Dperf(kG)$
between $\stmod_{kG}$, the full subcategory of the stable module
category $\Stmod_{kG}$ with objects finitely generated $kG$--modules, and the bounded derived category of
finitely generated $kG$--modules, modulo perfect
complexes, recalled in \S\ref{rickard-subsec}. We recall
$G$--twisted coefficient systems on the $p$--subgroup complex
$|\calS_p(G)|$, the nerve of the poset of non-trivial $p$--subgroups
of $G$, in \S\ref{coeff-subsub}. 

The chain complex
$C_*(|\calS_p(G)|;k_\phi)$ can be interpreted as the value on $G/e$ of the homotopy
left Kan extension of
$\phi$ along $\Oep(G)^{\op} \to \Op(G)^{\op}$, i.e., to the opposite orbit
category on {\em all} $p$--subgroups, which has model the homotopy
colimit in chain complexes $\hocolim_{P \in
  \calS_p(G)^{\op}} k_\phi$ (see
Proposition~\ref{kanext-prop}). When $\phi=1$, the complex is the Steinberg complex without the $k$--augmentation in degree $-1$, and the fact that
$\Phi^{-1}(k) \cong k$ is equivalent to projectivity of the augmented Steinberg
complex, proved by Quillen \cite[4.5]{quillen78} and Webb \cite{webb91}. 

\smallskip

Let us briefly describe how to obtain the whole group of endotrivial
modules $T_k(G)$ from
that of $T_k(G,S)$ together with results in the
literature. Using the definitions (see \S\ref{subsec:sylowtriv}) have an
exact sequence
\begin{equation}\label{Lsequence}
0 \to \TS \to T_k(G) \xrightarrow{\res_S} L \to 0
\end{equation}
with $L = \im(T_k(G) \to
T_k(S))$.
The
torsion-free rank of $L$ has been determined in
\cite[\S3]{CMN06}, extending the work of Alperin \cite{alperin01}. By
the classification of endotrivial modules for finite $p$--groups
\cite{CT04,CT05}, $L$ is torsion-free except when $S$ is cyclic or a
semi-dihedral or quaternionic $2$--group, and in particular the above sequence
is split outside those cases. The exceptions can be described explicitly by a
case-by-case analysis carried out in \cite{MT07,CMT13} and it turns
out that in all those cases the torsion part of $L$  equals that of $T_k(S)$. When $S$ is a semi-dihedral or
quaternionic $2$--group the restriction is furthermore split by \cite[Thms.~6.4 and 4.5]{CMT13}. When
$S$ is cyclic, $T_k(\Z/2) =0$, and for $|S|>2$, $T_k(S) \cong
\Z/2$. The $4$--fold periodic resolution of the trivial
$\F_3\sym_3$--module shows that the restriction is not always split,
but the structure of the extension above can in all cases be described
explicitly (see \cite[Thm.~3.2 and
Lem.~3.5]{MT07}).  We remark that while $L$ hence is known
as an abstract abelian group for any finite group
$G$,  explicit generators for the torsion-free part have in fact hitherto
been elusive
(see \cite{CMT14}). Combining the methods of this
paper with methods from higher algebra, we have recently, together
with Tobias Barthel and Joshua Hunt \cite{BGHprep}, also been
able to describe these, giving the precise image $L \subseteq
\lim_{G/P \in \Oep(G)}T_k(P) \subseteq T_k(S)$, and obtaining answers to conjectures in
\cite{CMT14}.

\smallskip

We now embark in putting the model for $\TS$ of Theorem~\ref{main} to use, expressing
$H^1(\Oep(G);k^\times)$ in terms of $p$--local information about the group.
We start with some
elementary observations:
By standard algebraic topology (recalled in \S\ref{pi1-subsec})
\begin{equation}\label{algtop}
 H^1(\Oep(G);k^\times)\mkern-1mu \cong \mkern-1mu  \Rep(\Oep(G),
 k^\times) \mkern-1mu \cong \mkern-1mu\Hom(\pi_1(\Oep(G)),k^\times)
 \mkern-1mu \cong \mkern-1mu
\Hom(H_1(\Oep(G)),k^\times)
\end{equation}
where $\Rep$ means isomorphism classes of
functors, viewing $k^\times$ as a category with one object.
Let
\begin{equation}\label{pembeddeddef}
G_0 = \langle N_G(Q) | 1< Q \leq S\rangle
\end{equation}
also called the $1$--generated core $\Gamma_{S,1}(G)$ by group
theorists, which, if proper
in $G$, is
the {\em smallest strongly
$p$--embedded subgroup containing $S$} (see Remark~\ref{stronglyembeddedclassification}). Recall that a Frattini
argument implies that we have surjections $N_G(S)/S \twoheadrightarrow
(G_0)_{p'}  \twoheadrightarrow G_{p'}$,
where we throughout the paper adopt the {\em convention} that
\begin{equation}G_{p'} = G/\langle\, g \in G\, |\, g \mbox { \em is of finite $p$--power order}\,\rangle.\end{equation}
(Hence, $G_{p'} = G/O^{p'}(G)$, when $G$ is
finite, and $M_{p'} = M/\Tors_{p}(M)$ when $M$ is abelian.)
 An application of Alperin's
fusion theorem \cite[\S3]{alperin67}
shows that $\Oep(G)$ and $\Oep(G_0)$ are equivalent categories and
the Frattini surjections above refine to
\begin{equation}\label{naivebounds2}
 N_G(S)/S \twoheadrightarrow \pi_1(\Oep(G)) \twoheadrightarrow
(G_0)_{p'} \twoheadrightarrow G_{p'}
\end{equation}
displaying $\pi_1(\Oep(G))$ as a finite $p'$--group (see Proposition~\ref{boundslemma}).
Via Theorem~\ref{main}, this encodes the classical bounds on $\TS$
(cf.~\cite[Prop.~2.6]{CMN06}, \cite[Lem.~2.7]{MT07}, and Proposition~\ref{injectivity}):
\begin{equation}\label{classicalbounds}
\Hom(G,k^\times) \leq  \Hom(G_0,k^\times) \leq \TS \leq
\Hom(N_G(S)/S,k^\times) 
\end{equation}
using that $k^\times$ does not contain $p$--torsion.

We describe the precise
kernel of $N_G(S)/S \twoheadrightarrow \pi_1(\Oep(G))$ in terms of $p$--local group
theory in Theorem~\ref{grouptheory-pi1}---it already directly implies a list of structural properties of
$\TS$ via Theorem~\ref{main} and \eqref{algtop} (see e.g.,
Corollary~\ref{standardprop}).

In the rest of the introduction we present our further
descriptions of
$\pi_1(\Oep(G))$ and its abelianization
$H_1(\Oep(G))$, each highlighting different structural properties.
We divide this into $5$ subsections: \S\ref{subgroupcpx-subsection}  Subgroup
categories, \S\ref{homologydecomp-subsection} Decompositions,
 \S\ref{ctconj-sec} The Carlson--Th\'evenaz conjecture, \S\ref{fusion-subsec} Fusion systems, and \S\ref{computations-subsection} Computations.

\subsection{Descriptions in terms of subgroup complexes} \label{subgroupcpx-subsection}
Let $\Tep(G)$ denote the transport category of $G$ with objects
the non-trivial $p$--subgroups of $G$, and $$\Hom_{\Tep(G)}(P,Q) = \{ g
  \in G| {}^gP \leq Q\}.$$
We have a functor $\Tep(G) \to \Oep(G)$ sending $g$ to the $G$--map
$G/P \to G/Q$ specified by $eP \mapsto g^{-1}Q$.
Since $\Oep(G)$ is a quotient of $\Tep(G)$
by morphisms in $p$--groups $Q \leq \Aut_{\Tep(G)}(Q)$,
\begin{equation} \label{pprimequotient}
\pi_1(\Tep(G))_{p'} \xrightcong
\pi_1(\Oep(G))
\end{equation}
(see Proposition~\ref{pprime}). As $k^\times$ contains no elements of
finite $p$--power order, \eqref{pprimequotient} and Theorem~\ref{main}
imply that $1$--dimensional
characters of $\pi_1(\Tep(G))$ also
parametrize Sylow-trivial modules for $G$, i.e.,
\begin{equation} \label{transportfundgrp}
\TS \cong \Hom(\pi_1(\Tep(G)),k^\times)
\end{equation}
By definition $\Tep(G)$ equals the transport category (or
Grothendieck construction) of the $G$--action on the poset of nontrivial
$p$--subgroups $\calS_p(G)$, under inclusion. Hence
\begin{equation}\label{grothendieck}
|\Tep(G)| \simeq |\calS_p(G)|_{hG}
\end{equation}
by Thomason's theorem \cite[Thm.~1.2]{thomason79}, where 
$X_{hG} = EG \times_G X$ denotes the Borel
construction (see also Lemma~\ref{thomasonlemma}). Hence we can also
conclude the following.
\begin{Co} \label{brown} For any finite group $G$ and $k$ any field of
  characteristic $p$,
$$\TS \cong H^1(|\calS_p(G)|_{hG};k^\times)$$

In particular if $H_1(|\calS_p(G)|_{hG})_{p'}
\xrightcong H_1(G)_{p'}$ then $\TS \cong \Hom(G,k^\times)$.
\end{Co}
From this perspective, exotic
Sylow-trivial modules parametrize the failure of the
  collection of non-trivial $p$--subgroups to be
  `$H_1(-;\Z)$--ample' generalizing Dwyer's definition
  \cite[1.2]{dwyer97} for mod $p$ cohomology (see Remark~\ref{amplerem} and Theorem~\ref{homotopy-finiteness}).
It also allows us to deduce a very recent result of Balmer
\cite{balmer15}, as we explain in Remark~\ref{reltobalmerremark}.

To further describe the group $\pi_1(\Tep(G))$, recall that
\begin{equation}\label{redtoG_0}|\calS_p(G)| \cong G \times_{G_0}
|\calS_p(G_0)|
\end{equation}
 with $|\calS_p(G_0)|$ connected, as observed by
 Quillen \cite[\S5]{quillen78} (see
 Proposition~\ref{conn-comp-C}).
 Hence
 $|\calS_p(G)|_{hG} \cong EG \times_G(G \times_{G_0} |\calS_p(G_0)|) \cong
 EG \times_{G_0}  |\calS_p(G_0)| \simeq |\calS_p(G_0)|_{hG_0}$, and we have a fibration sequence
\begin{equation}\label{fibseq}
|\calS_p(G_0)| \to |\calS_p(G)|_{hG} \to BG_0
\end{equation}
On fundamental groups it induces an exact sequence
\begin{equation}\label{fundseq}
1 \to \pi_1(\calS_p(G_0)) \to \pi_1(\Tep(G)) \to G_0 \to 1
\end{equation}
displaying $\pi_1(\Tep(G))$ as an extension of  $G_0$ by another group,
possibly infinite.

Using the identification \eqref{transportfundgrp}, the low-degree cohomology sequence of the 
group extension \eqref{fundseq}  (see \cite[{VI.8}]{HS71}) furthermore
induces an
exact sequence as follows.

\begin{Th}[Subgroup complex sequence]\label{sequence} Let $G$ be a finite
  group, $k$ a field of
  characteristic $p$, and $G_0 \leq G$ as in \eqref{pembeddeddef}.
We have an exact sequence
$$ 0 \to \Hom(G_0,
k^\times) \to \TS \to
H^1(\calS_p(G_0);k^\times)^{G_0} \xrightarrow{\partial} H^2(G_0;k^\times)$$
where superscript $G_0$ means invariants.
If $(H_{1}(\calS_p(G))_{G})_{p'}=0$ then $\TS \cong
\Hom(G_0,k^\times)$ and if $|\calS_p(G)|$ is simply connected,  then
$\TS \cong \Hom(G,k^\times)$.
\end{Th}
In words, $\TS$ is an extension of
$\Hom(G_0,k^\times)$ (producing $kG$--modules via induction
and discarding projective summands, cf.\ Lemma~\ref{stronglyembeddedinduction-lemma}) 
by a ``truly exotic'' part, not induced from $1$--dimensional $kG_0$--modules,
described above as the kernel of the
boundary map $\partial\co H^1(\calS_p(G_0);k^\times)^{G_0} \to H^2(
G_0;k^\times)$ (see also Remark~\ref{boundarymap}).
The group $H^2(G_0;k^\times)$ identifies with the $p'$--part of the
Schur multiplier of $G_0$ if $k$ is algebraically closed.

There is already an extensive literature on when $|\calS_p(G)|$ is simply
connected, see e.g., \cite[\S9]{smith11}. It is known to hold for
``sufficiently large'' symmetric
groups and finite groups of Lie type at the characteristic, as well as
some finite group of Lie type away from the characteristic and 
certain sporadic groups. It is conjectured to hold for many more (see
also \S\ref{computations-subsection}).

The description of
$\TS$ as $1$-dimensional representations of $\pi_1(\Tep(G))$ in
\eqref{transportfundgrp}
combined with manipulations with subgroup complexes also enables us to
see precisely how $\TS$ behaves under for instance passage to $p'$--index subgroups or
$p'$--central extensions (see Corollaries \ref{pprimesubgroups}
and \ref{pprimecentral}). 
Furthermore the groups $\pi_1(\Oep(G))$ and $\pi_1(\Tep(G))$ only
depend on very few of the $p$--subgroups of $G$, as we analyze in detail in
Appendix~\ref{propagating-sec}---see in particular
Theorems~\ref{remove-non-essential} and \ref{remove-big-subgroups}.

\subsection{Homology decomposition descriptions} \label{homologydecomp-subsection}
We now use homology decomposition techniques to get formulas for $\TS$. These techniques
have a long history for providing results about mod $p$ group
cohomology (see e.g., \cite{dwyer97, grodal02, GS06} and their
references), but here we are interested in coefficients in $k^\times$,
an abelian group with no $p$--torsion. We can however still 
describe the
low-degree $p'$--homology of $|\calS_p(G)|_{hG}$, by examining the
bottom corner of spectral sequences, even if they do not collapse.
More precisely, given a arbitrary
{\em collection} $\calC$ of subgroups (i.e., a set of subgroups closed under
conjugation), there are $3$ homology decompositions
one usually considers associated to the $G$--action on $|\calC|$: the
subgroup decomposition, the normalizer
decomposition, and the centralizer decomposition. The subgroup
decomposition does not provide new information, if $\calC$ is a
collection of $p$--subgroups, but the two others do.
Let us start with the normalizer decomposition.
\begin{Th}[Normalizer decomposition]\label{limitformula}

 Let $G$ be a finite group, $k$ a
 field of characteristic $p$, and
  $\calC \subseteq \calS_p(G)$ a subcollection such that the inclusion
  is
  a  $G$--homotopy equivalence, e.g., $\calC$ the collection of non-trivial
  $p$--radical subgroups or of non-trivial elementary
  abelian $p$--subgroups (see \S\ref{propagating-subcat}).
Then
$$\TS \cong \lim_{[P_0 < \cdots < P_n]}\Hom(N_G(P_0 < \cdots
<P_n),k^\times) $$
inside $\Hom(N_G(S)/S,k^\times)$, with the limit taken over conjugacy classes of chains in $\calC$
ordered by refinement. Explicitly:
\begin{equation*}\begin{split} \TS & \cong \ker\left( \oplus_{[P]} \Hom(N_G(P),k^\times) \to
\oplus_{[P<Q]} \Hom(N_G(P)\cap N_G(Q),k^\times)\right) \\ &\leq
\Hom(N_G(S),k^\times)
\end{split}
\end{equation*}
\end{Th}
Different possibilities for $\calC$
are given in
Theorem~\ref{propagating} (see also
\cite[Thm.~1.1]{GS06}).
Appendix~\ref{propagating-sec}, along with
Proposition~\ref{homformula} and Theorem~\ref{homformula-nr},
provide a precise analysis of which subgroups are needed to make
the conclusion of Theorem~\ref{limitformula} hold. 
We note that the data going into calculating the righthand
side,  normalizers of chains of say $p$--radical subgroups, or elementary
abelian $p$--subgroups, has been tabulated for a large number of
groups, and this relates to a host of problems in local group theory and
representation theory, such as
the classification of finite simple groups and conjectures of Alperin,
McKay, Dade, etc. To obtain Theorem~\ref{limitformula} from Theorem~\ref{main}, the only extra input,
apart from the isotropy spectral sequence,
is a result of Symonds \cite{symonds98}, formerly known as Webb's
conjecture, which we provide a short proof of in
Proposition~\ref{webbsconj}, that appears to be new.

\smallskip

We now explain the centralizer decomposition, which
ties into fusion systems.
Let $\calF_\calC(G)$ denote the restricted $p$--fusion system of $G$ with objects
$P \in \calC$ and 
$\Hom_{\calF_{\calC}(G)}(P,Q) =\Hom_{\Tep(G)}(P,Q)/C_G(P)$
i.e., monomorphisms induced by $G$--conjugation, writing
$\calF_p^*(G)$ when $\calC = \calS_p(G)$.

\begin{Th}[Centralizer decomposition] \label{centralizerthm} For $G$ a
  finite group and $k$ a field of characteristic $p$, we have an exact sequence
$$ 
0 \to  H^1(\calF^*_p(G);k^\times) \to \TS \to  \lim_{V \in
  \calF_{\!\!\calA_p^2}(G)}\Hom(C_G(V),k^\times) \to H^2(\calF^*_p(G);k^\times)$$
where $\calA^2_p$ denotes the collection of elementary abelian
$p$--subgroups of rank one or two.

 In particular, if $H_1(C_G(x))_{p'}=0$ for all elements
  $x$ of order $p$, and $H_1(N_G(S)/S)$ is generated by
  elements in $N_G(S)$ that commute with some non-trivial element in $S$, then
  $\TS =0$.
\end{Th}
This breaks $\TS$ up into parts depending on the underlying fusion system
and parts calculated from centralizers. It may be illuminating to note
that it specializes to the sequence
in cohomology with $k^\times$ coefficients induced by $1 \to C_G(V) \to N_G(V) \to N_G(V)/C_G(V)
\to 1$ in the very special case where $G$ has $p$--rank $1$, and hence
 a unique non-trivial elementary abelian $p$--subgroup
 (see Corollary~\ref{standardprop}\eqref{cyclic}).
In particular $H^1(\calF_p^*(G);k^\times)$ is a subgroup of $\TS$ 
depending {\em only} on the fusion system. We describe how to
calculate cohomology of $\calF_p^*(G)$
in Section~\ref{fund-fusion-subsec} and Appendix~\ref{propagating-sec}; the first
cohomology group is in
fact zero in many, but not all, cases. The assumptions of the `in
particular' are often satisfied for the sporadic groups, e.g., for the
Monster at the more difficult primes up to $13$.

\subsection{The Carlson--Th\'evenaz conjecture}\label{ctconj-sec}
Theorem~\ref{limitformula} implies the {Carlson--}\linebreak Th\'evenaz
conjecture, which predicts an algorithm for calculating $\TS$ from
$p$--local information, essentially by a change of language, taking $\calC
= \calS_p(G)$. Set $A^{p'}(G) = O^{p'}(G)[G,G]$, the smallest
normal subgroup of $G$ such
that the quotient is an abelian $p'$--group.

\begin{Th}[{The Carlson--Th\'evenaz conjecture
    \cite[Ques.~5.5]{CT15}}] \label{carlsonthevenazconj} Let $G$ be a
  finite group with non-trivial Sylow $p$--subgroup $S$, and define
  $\rho^r(S) \leq N_G(S)$ (depending on $G$ and $S$) via the following definition (cf.~\cite[Prop.~5.7]{CMN14}
\cite[\S4]{CT15}):
$\rho^1(Q) = A^{p'}(N_G(Q))$, 
$\rho^i(Q) = \langle N_G(Q) \cap \rho^{i-1}(R) | 1 < R
\leq S\rangle \supseteq \rho^{i-1}(Q)$.
Then 
$$H_1(\Oep(G)) \cong N_G(S)/\rho^r(S)$$ 
for any $r$ at least either
the nilpotency class of $S$ plus $1$,
or the number of groups in the longest proper
chain of
non-trivial $p$--radical subgroups.
Hence by Theorem~\ref{main}, for any field $k$ of
characteristic $p$,
$$\TS \cong \Hom(N_G(S)/\rho^r(S),k^\times)$$
\end{Th}
This in fact strengthens the Carlson--Th\'evenaz conjecture,
by providing a rather manageable bound on $r$ (see also Theorem~\ref{carlsonthevenaz-general2}). (The original
conjecture was only a prediction about the union $\rho^\infty(S)$, and
also had an algebraically closed assumption on $k$.) 
Theorem~\ref{carlsonthevenazconj} is well adapted to implementation
on a computer, and indeed Carlson has already made one such
implementation calculating $\rho^i(S)$, and to use this for proofs a
theoretical bound on when the $\rho^i(S)$ stabilize is obviously also necessary.

As already noted, the inverse limit in Theorem~\ref{limitformula}
identifies with a subset of $\Hom(N_G(S),k^\times)$. 
One may na\"ively ask if the limit could simply be described as
the elements in $\Hom(N_G(S),k^\times)$ whose restriction to
$\Hom(N_G(P)\cap N_G(S),k^\times)$ is zero on $A^{p'}(N_G(P)) \cap N_G(S)$ for
all $[P] \in \calC/G$, an obvious necessary condition to lie in the
limit. In other words, one may ask if one, in the language of
Theorem~\ref{carlsonthevenazconj}, could always take $r=2$. Computer calculations announced in \cite{CT15},  say this is
not the case for $G_2(5)$ when $p=3$.
The main theorem of that paper \cite[Thm.~5.1]{CT15} shows
that this  na\"ive guess {\em is} true when $S$ is abelian, as also follows
from Theorem~\ref{limitformula}. Our proof allows for a stronger statement:
\begin{Co}\label{radicalsnormal}
If all non-trivial $p$--radical subgroups in $G$ with $P < S$ are
normal in $S$,
then 
$$\TS \cong \ker \left(\Hom(N_G(S),k^\times) \to \oplus_{[P]}\Hom(N_G(S)\cap A^{p'}(N_G(P)),k^\times)\right)$$
where $[P]$ runs through $G$--conjugacy classes of non-trivial
$p$--radical subgroups with
$P < S$.
In particular in the notation of Theorem~\ref{carlsonthevenazconj}
$$\TS \cong \Hom(N_G(S)/\rho^2(S),k^\times).$$
More generally if we for each $[P]$ can pick $P \leq S$ with $Q = N_S(P)$ Sylow in $N_G(P)$
and  $N_G(P \leq Q \leq  S)A^{p'}(N_G(P \leq Q)) = N_G(P \leq Q)$, then
the same formulas hold, choosing such $P$.
\end{Co} 
See also Corollaries~\ref{CTcor} and \ref{CTcor2}. The last part of Corollary~\ref{radicalsnormal} provides a strengthening of
Carlson--Th\'evenaz's more technical \cite[Thm.~7.1]{CT15}, which
instead of abelian assumes that $N_G(S)$ controls $p$--fusion along with extra
conditions (see Remark~\ref{CT7.1rem}). Corollary~\ref{radicalsnormal}
however moves beyond these cases with limited fusion, and e.g., also holds for finite groups of Lie type in
characteristic $p$. To illustrate the failure in general we
calculate $\TS$ for $G= G_2(5)$ and
$p=3$ in
Proposition~\ref{G_25p3}, using
Theorem~\ref{limitformula}. Remark~\ref{CT-bound-rem} contains a more
detailed discussion of bounds on $r$.

\subsection{Further relations to fusion systems} \label{fusion-subsec}

Recall that a
$p$--subgroup $P$ is said to be centric if $Z(P) = C_G(P)$ and $p$--centric if
$Z(P)$ is a Sylow $p$--subgroup in $C_G(P)$. We denote by a
superscript $c$ full subcategories with objects the $p$--centric subgroups.
Subgroups of $\pi_1(\calF^c)$ parametrize
sub-fusion systems of $p'$--index of a fusion system $\calF$ by
\cite[\S5.1]{bcglo2}, and calculating $\pi_1(\calF^c)$ is of
current interest within $p$--local group theory---see e.g.,
\cite[\S4]{AOV12}, \cite{ruiz07}, and \cite[Ch.~16]{aschbacher11}. The condition that $\pi_1(\calF^c)=1$ is one of the conditions for a fusion system to be
reduced \cite[Def.~2.1]{AOV12}.
We here state a rough, yet useful, relation between
Sylow-trivial modules and $\pi_1(\calF^c)$---much more precise information is given in the paper
proper.

\begin{Th} \label{fusionthm}
We have the following commutative diagram of monomorphisms
$$\xymatrix@C-7pt{ \TS \ar@{^{(}->}[r] &
  \Hom(\pi_1(\bO_p^c(G)),k^\times)  \ar@{^{(}->}[r] &
  \Hom(N_G(S)/S,k^\times) \\
\Hom(\pi_1( \calF^*_p(G)),k^\times) \ar@{^{(}->}[r]   \ar@{^{(}->}[u]  &
 \Hom(\pi_1( \calF^c_p(G)),k^\times)  \ar@{^{(}->}[r] \ar@{^{(}->}[u]&
 \Hom(N_G(S)/SC_G(S),k^\times) \ar@{^{(}->}[u]}$$

If all $p$--centric $p$--radical subgroups are 
 centric then $\pi_1(\bO_p^c(G)) \cong \pi_1(\calF^c_p(G))$ and 
$$\Hom(\pi_1(\calF_p^*(G)),k^\times) \leq \TS \leq \Hom(\pi_1(\calF_p^c(G)),k^\times)$$
\end{Th}
The underlying maps are elaborated in
\eqref{fusionfundgrp-diagram}. The condition that $p$--radical $p$--centric subgroups are centric
is satisfied for finite groups of Lie type at the characteristic,
but also holds e.g., for many sporadic groups.
In general the inclusions in the diagram of Theorem~\ref{fusionthm}
may all be strict: for $\GL_n(q)$ and $p$ not dividing $q$, the main theorem in Ruiz \cite{ruiz07} states that
$\pi_1(\calF^c_p(\GL_n(q)))\cong \Z/e$ for $e$ the multiplicative order
of $q$ mod $p$ and $n \geq ep$, whereas
$\pi_1(\Oep(\GL_n(q))) =1$ when the $p$--rank of $\GL_n(q))$ is
at least $3$ by \S\ref{subgroupcpx-subsection} combined with
\cite[Thm.~12.4]{quillen78} \cite[Thm.~A]{das95}.
Section~\ref{fundgrp-sec} and Appendix~\ref{propagating-sec} analyses $\pi_1(\bO_\calC(G))$ and $\pi_1(\calF_\calC(G))$, for $\calC$ an arbitrary
collection of $p$--subgroups.

\subsection{Computational results}\label{computations-subsection}
We have already described how the results of this paper can be used to obtain a
range of structural and computational results about $\TS$.
To further illustrate the computational potential we will in Section~\ref{computations-sec} go
through different classes of groups: symmetric, groups of Lie type, sporadic,
$p$--solvable, and others, obtaining new results, and reproving a range
of old results. We briefly summarize this:

For sporadic groups $|\calS_p(G)|$ is sometimes known to agree with a
building, where simple connectivity has been studied extensively (see \cite[\S9]{smith11}). However
it is often easier to apply Theorems~\ref{limitformula},
\ref{centralizerthm}, \ref{fusionthm}, or \ref{grouptheory-pi1} directly, in particular since
the necessary $p$--local data has already been tabulated, due
to interest arising from counting conjectures in modular
representation theory. We demonstrate this in \S\ref{sporadic-subsec} by showing that
$\TS =0$ for  $G$ the Monster and $k$ a field of
characteristic $p=3,5,7,11$, or $13$, the primes left open in the recent paper
\cite{LM15sporadic} (see Theorem~\ref{monsterthm}).
It  should be possible to fill in the remaining gaps in the existing sporadic
group computations using similar arguments, though this is
outside the scope of the present paper. (This has subsequently been carried out by
David Craven \cite{craven21}.)

For finite groups of Lie type in characteristic $p$ the $p$--subgroup
complex is just the Tits building  \cite{quillen78}, which is simply connected
when the rank is at least $3$, again recovering results of
Carlson--Mazza--Nakano \cite{CMN06}.
For finite groups of Lie type in arbitrary characteristic,
 the $p$--subgroup complex is also
believed to generically be a wedge of high dimensional spheres, which
would imply that generically there were no exotic Sylow-trivial modules by
Theorem~\ref{sequence}. It has been verified in
a number of cases  \cite{das95,das98,das00}, and this way, we recover very recent results of Carlson--Mazza--Nakano for
the general linear group for any characteristic \cite{CMN14,CMN16}, again using
Theorem~\ref{sequence}, and for symplectic groups we get the
following new result (see \S\ref{ftgrpsoflietype}).

\begin{restatable}{Th}{symplectic}\label{Spn} 
Let
  $G = \Sp_{2n}(q)$, and $k$ a field of characteristic $p$. If the multiplicative
  order of $q$ mod $p$ is odd, and $G$ has an elementary abelian
  $p$--subgroup of rank $3$, then $\TS = 0$.
\end{restatable}
In joint work in progress with Carlson,
Mazza, and Nakano, we determine the group of endotrivial modules for all finite groups of
Lie type using the methods of this paper
combined with the ``$\Phi_d$-local'' approach to finite groups of Lie
type.

For symmetric and alternating groups, the $p$--subgroup complex is known to be
simply connected (except known small exceptions) \cite{ksontini03,ksontini04}, so we recover results of
Carlson--Hemmer--Mazza--Nakano \cite{CMN09,CHM10}, using
Theorem~\ref{sequence} (see \S\ref{symmetricgroups}, where we also correct a
small mistake in the literature about alternating groups).

For $p$--solvable groups above $p$--rank one, it is has been shown that there are no exotic
Sylow-trivial modules by works of Carlson--Mazza--Th\'evenaz
\cite{CMT11} and
Robinson--Navarro \cite{NR12}, using an inductive argument, that at one
inductive step indirectly relies on the classification of finite simple groups. We
link this inductive step to stronger conjectures by Quillen and Aschbacher
about the connectivity of the $p$--subgroup complex for $p$--solvable
groups (see \S\ref{psolvable}).

As explained above, a calculation of
$\pi_1(\Oep(G))$ and $H_1(\Oep(G))$  for all finite simple groups should be within
reach. To pass to arbitrary finite groups on may then hope to use the structural properties of
$\pi_1(\Oep(G))$ given in this paper 
(in particular Section~\ref{fundgrp-sec} and Appendix~\ref{propagating-sec}) to reduce the question of
the existence of exotic Sylow-trivial modules to the simple case,
using the generalized Fitting subgroup $F^*(G)$ from finite group
theory. We note in this
connection that Aschbacher \cite{aschbacher93} has, in a
certain sense, reduced the question of simple connectivity of $|\calS_p(G)|$ to simple
groups, modulo his conjecture for $p$--solvable groups alluded to above. 
Based on available data, it
may be that $\pi_1(\Oep(G)) \xrightcong (G_0)_{p'}$ when the
$p$--rank is at least three? This would imply $T_k(G,S) \cong
\Hom(G_0,k^\times)$ under that assumption. See
Remarks~\ref{CT-bound-rem}, \ref{stronglyembeddedclassification} and
Section~\ref{computations-sec} for more information.

\subsubsection*{Further vistas} In addition to the structural and computational
consequences
described so far, it is natural to wonder about further
representation theoretic significance of the finite $p'$--groups
$\pi_1(\bO_p^*)$, $\pi_1(\bO_p^c)$, $\pi_1(\calF^c_p)$, and
$\pi_1(\calF_p^*)$, and the higher
homotopy and homology groups,
yet to be found? In a similar vein one may wonder if the method for
constructing representations in Theorems~\ref{main} and \ref{sylowss}
of this paper,
when applied to
more general coefficient systems on $|\calS_p(G)|$ (see
\S\ref{coeff-subsec}), could enable one to describe a larger slice of
the stable module category, potentially shedding light on standard
counting conjectures and their derived generalizations?
Remarks
\ref{local-to-global},
\ref{partial-groups-rem}, \ref{transportersystem-rem}, and \ref{gluing-rem} give some hints in these directions.
Further afield, one may look for a contribution of the homotopy type
of the orbit or fusion category on collections of $p$--subgroups for
problems involving any
other `$p$--local' symmetric
monoidal (infinity) category that depend on $G$, whether in algebra
or topology,
similar to the one discovered here for $\Stmod_{kG}$ (see also
\cite{mathew16} for infinity categorical considerations).

\subsubsection*{Organization of the paper} In Section~\ref{prelim}
we state conventions and introduce the categories and
constructions needed for the main results, providing a fair amount of
detail in the hope of making the paper accessible to both group
representation theorists and topologists. In Section~\ref{mainproofs}
we prove Theorem~\ref{main}, only relying on the recollections in
Section~\ref{prelim}. In Section~\ref{fundgrp-sec} we prove a number
of results on $\pi_1(\Oep(G))$ and $\pi_1(\calF_p^*(G))$, and among
other things deduce the
consequences described in
\S\S\ref{subgroupcpx-subsection},\ref{fusion-subsec}, including
Corollary~\ref{brown}, and
Theorems~\ref{sequence} and \ref{fusionthm}. In Section~\ref{homologydec-sec} the
decompositions and the Carlson--Th\'evenaz conjecture are established as stated in
\S\S\ref{homologydecomp-subsection},\ref{ctconj-sec}, proving
Theorems~\ref{limitformula}, \ref{centralizerthm}, and \ref{carlsonthevenazconj},
and Corollary~\ref{radicalsnormal}. In Section~\ref{computations-sec} we
go through the computational consequences and results, including Theorem~\ref{Spn}. Finally Appendix~\ref{propagating-sec} contains a number of results about changing the
collection of subgroups, which are used throughout
Sections~\ref{fundgrp-sec}--\ref{computations-sec}, and should also be of independent interest.

\subsubsection*{Acknowledgements} 
The idea for this paper arose from hearing a talk by Jon Carlson at Eric
Friedlander's 70th birthday conference at USC in 2014, explaining his conjecture with
Th\'evenaz \cite{CT15} and Balmer's work on weak homomorphisms
\cite{balmer13}.
I'm grateful to him and Paul Balmer for inspiring
conversations on that occasion.
I would also like to thank Ellen Henke, Anne Henke, Bob Oliver,
Geoff Robinson, and Ron Solomon for patiently answering my group theoretic
questions, and Serge Bouc, Jon Carlson, David Craven, Joshua Hunt, Radha Kessar, Nadia Mazza, Bob Oliver, Jacques
Th\'evenaz, and especially Steve Smith for helpful comments on
preliminary versions of this paper. Last but not least a big thanks to
two anonymous referees whose detailed reading caught several
inaccuracies and led to many improvements of the exposition.

 The author enjoyed the hospitality of MSRI Berkeley, Spring 2018
 (NSF grant DMS-1440140) and the Isaac Newton Institute, Cambridge,
 Fall 2018
  (EPSRC grants EP/K032208/1 and EP/R014604/1) where the
  manuscript was revised.

\section{Notation and preliminaries}\label{prelim}
This short section collects some conventions and definitions, giving some detail, in the
hope to make the paper accessible to both group theorists and
topologists. We defer certain parts of the discussion of coefficient
systems and derived categories to   \S\ref{coeff-subsub} and
\S\ref{rickard-subsec} respectively.

\subsection{Conventions}\label{conventions}
In this paper $G$ will always be an arbitrary finite group and $p$ an
arbitrary prime dividing the order of $G$ (to avoid having to make
special statements in the trivial case where this is not so). We use
the notation $S$ for its Sylow $p$--subgroup and set $G_0 = \langle
N_G(Q) | 1< Q \leq S\rangle$, which when $G_0 < G$ is the smallest
strongly $p$--embedded subgroup of $G$ containing $S$ (see also the
introduction
\eqref{pembeddeddef} and \S\ref{components-subsec}). By $k$ we will
always mean
a field of characteristic $p$, where $p$ divides $|G|$, but subject to
no further restrictions. Thus $k$ is not assumed to be algebraically closed. Note that the units $k^\times$ cannot have
$p$--torsion, as the Frobenius map is injective, and it will be
uniquely $p$--divisible if $k$ is perfect (see also \S\ref{higher-homotopy}). 
Our $kG$--modules will not be assumed
finitely generated, though everything could also be phrased inside
the smaller category of finitely generated modules with the same
result. Tensor products are over $k$.

By a collection of $p$--subgroups $\calC$, we mean a set of
$p$--subgroups of $G$, closed under conjugation, which we view as a
poset under inclusion, hence as a category. We use standard
notation for various specific collections of $p$--subgroups, like
$\calA_p(G)$  for the non-trivial elementary abelian $p$--subgroups,
$\calB_p(G)$  for the non-trivial $p$--radical subgroups, etc, which we
also recall in Appendix~\ref{propagating-sec}.

We use
standard group theoretic notation, plus that
$$G_{p'} = G/\langle\, g \in G\, |\, g \mbox { is of finite $p$--power order}\,\rangle$$
and $A^{p'}(G) = O^{p'}(G)[G,G]$, as also mentioned in the
introduction.

By a space we will for convenience mean a simplicial set,
$|\cdot|$ denotes the nerve functor from categories to simplicial
sets, and homotopy equivalence means homotopy equivalence after
geometric realization. Group theorists not familiar with
simplicial sets are largely free to think of them as simplical
complexes, or topological spaces, and can find a quick introduction in
\cite[Ch.~1.8]{benson91vol2}, and more information e.g., in
\cite{DH01}. A space is simply connected if it is connected
with trivial fundamental group, and connected spaces are assumed to be
non-empty.

We use $[\cdot]$ for conjugacy and equivalence classes, $\cong$ for isomorphism, and $\simeq$ for
equivalence.

\subsection{Sylow-trivial modules}\label{subsec:sylowtriv}
As stated in the introduction we use the term {\em Sylow-trivial} for
our basic objects: $kG$--modules that when restricted to a Sylow
$p$--subgroup $S$ split as the trivial module $k$
direct sum a projective module.
Two Sylow-trivial modules $M$ and $N$ are called {\em equivalent} if
there exist projective $kG$--modules $P,Q$ such that $M \oplus P
\cong N \oplus Q$.
We denote by $T_k(G,S)$ the set of equivalence
classes of Sylow-trivial modules. By Proposition~\ref{onedim} each
Sylow-trivial module has an, up to isomorphism unique, indecomposable representative, and this is isomorphic to a summand of $k[G/S]$. We claim that
tensor product over $k$ endows $T_k(G,S)$ with an abelian group structure with
neutral element $k$.  First it is clear that the tensor product of two
Sylow-trivial modules is again Sylow-trivial, as the tensor product of
a projective $kS$--module with any $kS$--module is again
projective. The same fact (now over $kG$) implies that the multiplication descends to equivalence classes.
Also if $M\downarrow_S \cong k \oplus \mathrm{(proj)}$, then the cokernel of the
unit $kG$--map $ k \to  M^* \otimes M$ is a projective $kG$--module, since
projectivity is detected on $kS$, and as $kS$--modules, $M^* \otimes M  \cong k^*
\otimes k \oplus \mathrm{(proj)}$, with the map $k \to k^* \otimes k
\cong k$
the identity.
Thus $M^*$ is  an inverse to $M$ in $\TS$, and $\TS \leq T_k(G)$, the
group of endo-trivial modules.

\subsection{The stable module category $\Stmod_{kG}$}\label{stmod-subsec}
Recall that the stable module category
$\Stmod_{kG}$ is the category with objects $kG$--modules and morphisms
from $M$ to $N$ the quotient of $\Hom_{kG}(M,N)$ where we identify two
maps if their difference fractors through a projective
$kG$--module. We denote by $\stmod_{kG}$ the full subcategory with
objects {\em finitely generated} $kG$--modules. The relevance of the
stable module category for endo-trivial modules stems from the following well-known fact.

\begin{lemma}
Two $kG$--modules $M$, $N$ are isomorphic in $\Stmod_{kG}$ if and only
if there exist projective $kG$--modules $P$, $Q$ such that $M \oplus
P \cong N \oplus Q$. In particular $T_k(G) \xrightcong \Pic(\Stmod_{kG})$.
  \end{lemma}
  \begin{proof}
    It is clear that equivalent modules are isomorphic in
    $\Stmod_{kG}$. Conversely (following an online argument by Rickard), assume we have maps $f\co M
\to N$ and $g\co N \to M$ such that $fg-1$ and
$gf-1$ factor through projectives, and let $M \xrightarrow{\phi} Q
\xrightarrow{\psi} M$ be a factorization of $gf-1$. Let $\tilde f = (f,\phi)
\co M \to N \oplus Q$ and $\tilde g = (g - \psi) \co N
\oplus Q \to M$. Then $\tilde g \tilde f = 1$ and so $M \oplus
\ker(\tilde g) \cong N
\oplus Q$. As $\tilde f \tilde g-1\co N \oplus Q \to N \oplus Q$ also factors through a
projective (using that $fg-1$ does), and is the identity on
$\ker(\tilde g)$, we conclude that
$P = \ker(\tilde g)$ is a retract of a projective and hence projective as
wanted.

By the first part $T_k(G) \hookrightarrow \Pic(\Stmod_{kG})$,
where $\Pic(\Stmod_{kG})$ is defined as
isomorphism classes of invertible objects under tensor product in
$\Stmod_{kG}$. It is also surjective as any invertible object $M$ in $\Stmod_{kG}$
has inverse $M^*$ by a small calculation, true in any closed
symmetric monoidal category (see e.g.,
\cite[Prop.~A.2.8]{HPS97}). \end{proof}
As for Sylow-trivial modules, modules in $\Stmod_{kG}$ in fact have a representative
without projective summands, unique up to isomorphism of $kG$--modules
(see
Proposition~\ref{onedim} and \cite[Lem.~3.1]{rickard97}).

\subsection{Categorical constructions} \label{cat-subsec}
Define the transport category $\calT(G)$ as the category with
objects all subgroups of $G$ and morphisms $\mor(P,Q) = \{ g \in G |
{}^gP \leq Q\}$, i.e., the transport category, or Grothendieck construction,
of the left conjugation action of $G$ on the poset of
all subgroups (see also Lemma~\ref{thomasonlemma}). We have a quotient functor $\calT(G) \to \bO(G)$ which on objects
sends $H$ to $G/H$ and assigns to $(g, {}^g H \leq K)$ the $G$--map
$G/H \xrightarrow{[g^{-1}]} G/K$
that sends the trivial coset $eH$ to $g^{-1}K$. This induces
$K\backslash \Hom_{\calT(G)}(H,K) \xrightcong
\Hom_{\bO(G)}(G/H,G/K)$ (see also e.g., \cite[I.10]{tomdieck87book}). Denote by $\bO_\calC(G)$ and $\calT_\calC(G)$ the full
subcategories with objects
$G/H$ and $H$ respectively, for $H \in \calC$, and we continue to use 
the notation $\Oep(G)$ and $\Tep(G)$ for these categories when
$\calC = \calS_p(G)$ is the collection of all non-trivial $p$--subgroups.
We also introduce the fusion category $\calF_\calC(G)$ and
fusion-orbit category $\bcalF_\calC(G)$ both with objects $P \in
\calC$ and morphisms
\begin{align*}&\Hom_{\calF_\calC(G)}(P,Q) =\Hom_{\calT_\calC(G)}(P,Q)/C_G(P) \mbox{\,\,\,\,\,\,\,\,\,\,\,\,
  and }\\
&\Hom_{\bcalF_\calC(G)}(P,Q) =Q\backslash\Hom_{\calT_\calC(G)}(P,Q)/C_G(P)
            \end{align*}
respectively, i.e., monomorphisms induced by conjugation in $G$ and
ditto modulo conjugation in the target.
(The fusion-orbit category is called the exterior quotient $\tilde
\calF$ of $\calF_\calC(G)$ by Puig \cite{puig06}, and is also
sometimes denoted $\bO(\calF)$.)
All four categories hence have object-sets identifiable with $\calC$,
and morphisms related via quotients
$$\xymatrix@R-7pt{ & {\calT_\calC(G)} \ar@{->>}[dr] \ar@{->>}[dl] & \\
 \bO_\calC(G) \ar@{->>}[r] & \bcalF_\calC(G) & \calF_\calC(G) \ar@{->>}[l]}$$

\begin{rem}[On $\op$'s and inverses]\label{ops-inverses}
Since op's and inverses are a common source of light confusion, we make
a few remarks about their presence in the formulas:
A group $G$ viewed as a category with one object is isomorphic as a category to $G^{\op}$ via the map $g
\mapsto g^{-1}$.  In particular $\calT(G)$ is isomorphic to the category with morphism set 
$\mor'(P,Q) = \{ g \in G |
P^g \leq Q\}$, which is the Grothendieck construction
$\calS_p(G)_{G^{\op}}$. Redefining the transport category this
way would get rid of the inverse appearing in the formula
for the projection map $\calT(G) \to \bO(G)$; alternatively one could reparametrize
the orbit category, the choice made e.g., in \cite[III.5.1]{AKO11}. 
Notice also that when we are considering functors to an abelian group
such as $k^\times$, viewed as a category with one object, covariant
functors naturally {\em equals} contravariant functors. (The
identification using the isomorphism between $G$ and $G^{\op}$
produces the automorphism given by ``pointwise inverse''.)
\end{rem}

\subsection{Low dimensional cohomology and homotopy of categories}\label{pi1-subsec}
The cohomology of a small category $\D$ with constant coefficients in
an abelian group $A$
is defined as the cohomology of the simplicial
set $|\D|$ with constant coefficients $A$. In particular
$H^1(\D;A)$ identifies with functors $\D \to A$, up to natural isomorphism of
functors, where $A$ is viewed as a
category with one object. Indeed, a $1$--cocycle 
is a function $F\co \Mor(\D) \to A$ such that $F(\beta \circ \alpha) =
F(\beta) + F(\alpha)$ (hence
$F(\id) =0$), in other words a functor $F\co \D \to A$. And a $1$--coboundary is
a $1$--cocycle of the form $F(\alpha) = g({\operatorname{cod}}(\alpha))-g({\operatorname{dom}}(\alpha))$,
for a function $g\co \Ob(\D) \to A$, i.e., a functor that admits a natural
transformation (hence isomorphism) to the zero functor. Furthermore
$H^1(\D;A) \xrightcong \Hom(H_1(\D),A)$, where $H_1(\D)$
is the abelian group of cycles of morphisms in $\D$, modulo the
equivalence relation coming from composition, a special case of the universal
coefficient theorem \cite[Thm.~3.2]{hatcher02}.
(See also
e.g., \cite{webb07}.)

Homotopy groups are defined as $\pi_i(\D,d) = \pi_i(|\D|,d)$, and also
have
a categorical description for $i=1$: the fundamental groupoid
$\pi(|\D|)$ identifies as $\D[\Mor(\D)^{-1}]$ the category obtained by
formally inverting all morphisms in $\D$ (left
adjoint to the inclusion functor from groupoids to categories), providing a canonical
isomorphism $\pi_1(|\D|,d) \cong \Aut_{\D[\Mor(\D)^{-1}]}(d)$.  (See
also e.g., \cite[\S1]{quillen73}.)

Assume that $\D$ is connected, i.e., that all objects $x, y$ can be connected
by a finite zig-zag $x = x_0 \to x_1 \leftarrow \cdots \to x_n = y$ of
morphisms. The universal properties
gives us isomorphisms
$$\Hom(\pi_1(\D,d),A) \xrightcong H^1(\pi_1(\D,d);A) \xleftcong H^1(\D[\Mor(\D)^{-1}];A) \xrightcong H^1(\D;A)$$
and one also sees directly that the canonical
map $\pi_1(\D,d) \to H_1(\D)$, viewing loops as $1$--cycles,
is abelianization, a special case of the Hurewicz theorem
\cite[Thm.~2A.1]{hatcher02}. 

It is often convenient to have a concrete way of writing elements in
$\pi_1(\D,d)$. For this, pick for each object $x \in \D$ a
path $i_x$ from $x$ to $d$, producing a functor $\D \to
\pi_1(\D,d)$ via $(\phi\co x \to y) \mapsto i_y \circ \phi \circ
i_x^{-1}$. This induces a functor $\omega\co \D[\Mor(\D)^{-1}] \to
\pi_1(\D,d)$, which is manifestly an inverse equivalence of categories to the
inclusion, and we can think of elements of $\pi_1(\D,d)$ as finite
zig-zags of morphisms in $\D$ in this way. Recall also that we may
replace $\D$ by an equivalent category without changing the result.
In particular for $\D = \Oep(G)$, we can, up
to equivalence of categories, replace it with the full subcategory
$\bO_S^*(G)$ with objects $G/Q$ for $1<Q \leq S$, and
take basepoint $G/S$, so
that we have 
canonical maps $i_{G/Q}\co G/Q \to G/S$. 
Hence $\omega(G/P \to G/Q) = 1$
for $P \leq Q$,
allowing us to effectively ignore
morphisms induced by inclusions;
similarly for the other standard categories from
\S\ref{cat-subsec}.
We will often suppress the basepoint from the notation, and the above shows why we can do this
without ambiguity. Note also that any basepoint-dependence disappears
after abelianization.

We used at various points, e.g., in \eqref{grothendieck}, that the Borel
construction on the nerve of a small category can be expressed as the
nerve of the transport category (or Grothendieck construction). For
convenience of the reader, let us prove this special case of Thomason's theorem \cite[Thm.~1.2]{thomason79}.
\begin{lemma}[Thomason's theorem for Borel
  constructions]\label{thomasonlemma}
    For a small category $\D$ with an action of $G$, let
  $\D_G$ denote the transport category with objects the objects of $\D$ and morphisms from $x$ to $y$
given by a pair $(g,f\co gx
\to y)$, where $g \in G$ and $f \in \Hom_\D(gx,y)$. Then
  $$|\D_G| \cong |\D|_{hG}$$
  where $|\D|_{hG}$ denotes the Borel construction.
   In particular $|\calT_\calC| = |\calC_G| \cong |\calC|_{hG}$  for
   any collection $\calC$.
   \end{lemma}
\begin{proof}
  Define the category $E\G$ to be the
category with objects the elements of $G$ and a unique morphism
between all elements, so that $|E\G| = EG$, the universal free
contractible $G$--space. Our group $G$ acts freely on the product
category $E\G \times \D$, on objects given by $g \cdot (h,x) =
(hg^{-1},gx)$. The quotient $E\G \times_G \D$ identifies with
$\D_G$. Since the nerve functor commutes with products and free
$G$--actions we have identifications 
$ |\D|_{hG} = EG \times_G |\D| \cong |E\G \times_G \D| = |\D_G|$
as wanted.
\end{proof}

\subsection{Coefficient systems} \label{coeff-subsec}
Finally we recall the notion of coefficient system, to be
specialized and elaborated in later sections. A general homological $G$--coefficient system on a
$G$--space $X$ is just a
functor $A \co(\Delta X)_G \to \Rmod$, to $R$--modules, for $R$ a
ground ring.
Here $\Delta X$ is the {\em category of simplices} with
objects the simplices of $X$, and morphisms given by iterated face and
degeneracy maps \cite[I.2]{GJ99}, and $(\Delta X)_G$ is the associated transport category
of the left $G$--action on $\Delta X$.
The chain complex $C_*(X;A)$, with $C_n(X;A) = \oplus_{\sigma} A(\sigma)$,
and the standard simplicial differential, is a chain complex of
$RG$--modules via the induced $G$--action, $A((g,\id_{g\sigma}))\co
A(\sigma) \to A(g\sigma)$  (see also \cite[\S2]{grodal02}).
In this paper two more restrictive types of coefficient systems play a special
role, namely $G$--twisted coefficient systems, used in
Section~\ref{mainproofs}, and Bredon  $G$--isotropy coefficient systems, used in
Section~\ref{homologydec-sec}. These special systems enjoy homotopy
invariance properties, not enjoyed by general $G$--coefficient systems,
as we explain in those sections.

\section{Proof of Theorem~\ref{main}}\label{mainproofs}
In this section we prove Theorem~\ref{main}, just using the preparations from the preceding
section.

\subsection{Injectivity of the map $\Phi$}
We start with some elementary facts about Sylow-trivial modules,
including dealing with the finite-dimensionality issue once and for all.
\begin{prop}\label{onedim} Any Sylow-trivial $kG$--module $M$ is of the
  form $N \oplus P$,  where $N$ is an indecomposable direct summand of
  $k[G/S]$ and $P$ is projective, and $N$ is uniquely
  determined, up to
  isomorphism.
If $O_p(G) \neq 1$, then any indecomposable Sylow-trivial $kG$--module
is one-dimensional.
\end{prop}

\begin{proof} 
This is well known, and follows by
\cite[Thm.~2.1]{BBC09} and \cite[Lem.~2.6]{MT07}, but since it is
among the few ``classical'' representation theory facts used,
we give a direct proof.
Let $M$ be our Sylow-trivial module, and
recall that $M$ is a direct summand of $M\downarrow_S^G\uparrow^G_S$,
since the composite of the $kG$--map $M \to
M\downarrow_S\uparrow^G = kG \otimes_{kS} M$, $m \mapsto \sum_{g_i \in G/S} g_i \otimes
g_i^{-1}m$, and the $kG$--map $M\downarrow_S\uparrow^G \to M$, $g\otimes m
\mapsto gm$, is multiplication by $|G:S|$, which is a unit in $k$. (See
also e.g., \cite[Cor.~3.6.10]{benson91}, where the standing finitely generated assumption is not being used.)
  By assumption $M\downarrow_S \cong k \oplus (\mbox{proj})$, so $M$
  is a summand of  $k\uparrow_S^G
\oplus (\mbox{proj})$. As explained in \cite[Lem.~3.1]{rickard97}, any
$kG$--module, also infinite dimensional, can be written as a direct sum of a
projective module and a module without projective summands, and the
non-projective part is unique up to isomorphism. Furthermore, by
\cite[Lem.~3.2]{rickard97}, this decomposition respects direct
sums. This shows that $M \cong N \oplus P$, with $N$ a non-projective direct
summand of $k\uparrow_S^G$ which is uniquely determined, up to isomorphism. Furthermore,  $N$ has to be
indecomposable, since otherwise it cannot be Sylow-trivial.

Now assume $O_p(G) \neq 1$. Since $M$ is a direct summand of $k\uparrow_S^G$ by the first part, then $M\downarrow^G_S$
is a direct summand of $k\uparrow_S^G\downarrow^G_S \cong \oplus_{g \in S\backslash
  G/S} k\uparrow_{S\cap {}^gS}^S$, which does not
contain any projective summands, since $S\cap {}^gS
\geq O_p(G) \neq 1$. Hence $M\downarrow^G_S \cong k$ as wanted.
\end{proof}

Let us also give the following well known special case of Green correspondence
\cite[Thm.~3.12.2]{benson91} (see also
\cite[{Prop.~2.6(a)}]{CMN06}), used for injectivity of the
map $\Phi$ of Theorem~\ref{main}.
\begin{prop}\label{injectivity}
Let $M$ be a $kG$--module such that $M\downarrow_{N_G(S)}
\cong k \oplus \mathrm{(proj)}$. Then $M \cong k \oplus  \mathrm{(proj)}$.
In particular restriction provides an inclusion
$$\TS \hookrightarrow T_k(N_G(S),S) \cong
\Hom(N_G(S)/S,k^\times).$$
\end{prop}
\begin{proof}
As mentioned this is a special case of Green
correspondence, but let us extract a direct argument: By Proposition~\ref{onedim} it is enough to see that
if $M$ is indecomposable, then $M \cong k$. So, set $N = N_G(S)$, and
assume that $M$ is an indecomposable $kG$--module such that $M\downarrow^G_N \cong k \oplus {\mathrm{(proj)}}$. As in Proposition~\ref{onedim}, $M$ will be a summand of
$M\downarrow_N^G\uparrow_N^G \cong k\uparrow_N^G \oplus
{\mathrm{(proj)}}$, and hence a summand of $k\uparrow_N^G \cong k
\oplus L$, for $L$ a complement of $k(\sum_{g\in G/N}gN)$. But 
$L\downarrow^G_S \cong \oplus_{g \in S \backslash G/N, g \not \in
  N}k\uparrow_{S\cap {}^gN}^S$, and in particular it does not contain
$k$ as a direct summand, so $M$ has to be a direct summand of
$k$. Thus restriction provides an inclusion $\TS \hookrightarrow
T_k(N_G(S),S)$, and  furthermore $T_k(N_G(S),S) \cong
\Hom(N_G(S)/S,k^\times)$ by Proposition~\ref{onedim}.
\end{proof}
We can already now prove a part of Theorem~\ref{main}.
\begin{prop}\label{phiishomo}
  For any Sylow-trivial module $M$,
  $$G/P \mapsto \hat H^0(P;M) = M^P/(\sum\nolimits_{g \in P} g)M$$
defines a functor from $\Oep(G)^{\op}$ to one-dimensional $k$--modules and
isomorphisms, which we can identify with a functor $\Oep(G) \to
k^\times$. 
The assignment that sends a Sylow-trivial module $M$ to the above functor defines an injective group homomorphism $\Phi\co \TS \to
H^1(\Oep(G);k^\times)$.
\end{prop}

\begin{proof}
It is clear that $\hat H^0(P;M) = M^P/(\sum_{g \in P}g)M$ is
one-dimensional, since $M\downarrow_P \cong k \oplus \mathrm{(proj)}$
by assumption, and the construction kills the projective part.
It furthermore defines a functor on $\Oep(G)^{\op}$ that sends $G/P \xrightarrow{g}
G/P'$ to the morphism $M^{P'}/(\sum_{g \in P'}g)M \to M^{P}/(\sum_{g
  \in P}g)M$ given by multiplication by $g$. Picking a fixed
one-dimensional $k$--vector space, say
$M^S/(\sum_{g \in S}g)M$, and for each one-dimensional $k$--vector space a fixed isomorphisms to $M^S/(\sum_{g \in S}g)M$, identifies this with a functor $\Oep(G) \to
k^\times$ (concretely we may model $\Oep(G)^{\op}$, up to equivalence of
categories, by the full subcategory $\bO_S^*(G)^{\op}$ with objects $G/P$ for $1<P\leq
S$, and use the identifications induced by
canonical morphisms $G/P \to G/S$). (See also Remark~\ref{ops-inverses} for a discussion of variance.)
  The resulting functor is uniquely
defined, up to isomorphism of functors, and hence defines a a unique
element in $H^1(\Oep(G);k^\times)$.

It is also clear that $\TS \to
H^1(\Oep(G);k^\times)$ is a group homomorphism, where the group
structure on the right is pointwise multiplication, as
$$(M \otimes N)^P/(\textstyle \sum_{g \in P}g)(M\otimes N)
\xleftcong   (M^P/(\textstyle \sum_{g
  \in P}g)M) \otimes  (N^P/(\textstyle \sum_{g \in P}g)N)$$ 
and tensoring two $1$--dimensional $kN_G(P)$--modules amount to multiplying
their characters. 

Finally we check injectivity: By Proposition~\ref{onedim},
$M\downarrow_{N_G(S)} \cong \hat H^0(S;M) \oplus \mathrm{(proj)}$. If
$\Phi([M])$ is the identity, then the action of $N_G(S)$ on $\hat
H^0(S;M)$ is trivial, i.e.,  $M\downarrow_{N_G(S)} \cong  k \oplus
\mathrm{(proj)}$.
Hence $M \cong k \oplus \mathrm{(proj)}$ by
Proposition~\ref{injectivity} as wanted.
\end{proof}

\begin{rem} For the modules we consider in Theorem~\ref{main},
  $$\hat H^0(P;M) \cong M^P/(\oplus_{Q<P}(\sum_{[g]\in P/Q} g) M^Q),$$ the Brauer quotient, and this may be another way to view this
  construction.
\end{rem}

\subsection{$G$--twisted coefficient systems and the
  Buchweitz--Rickard equivalence}
Before continuing with the rest of the proof of Theorem~\ref{main} in
the next subsection, we now recall the
Buchweitz--Rickard equivalence, used in Theorem~\ref{main}, and also explain
$G$--twisted coefficient systems in some detail.

\subsubsection{The Buchweitz--Rickard equivalence}\label{rickard-subsec}
We recall the equivalence of homotopy categories
\begin{equation}\label{BR-equivalence}\stmod_{kG} \xrightsimeq
  D^b(kG)/\Dperf(kG).
  \end{equation}
given by viewing a module as a chain complex concentrated in degree zero,
 going back to
 \cite[Thm.~2.1]{rickard89} and \cite[Thm.~4.4.1]{buchweitz86}. Here
 $\stmod_{kG}$ is the ``small'' stable module category with objects
 finitely generated $kG$-modules and morphisms homomorphisms of
 $kG$--modules modulo the relation that two morphisms are
 equivalent if their difference factors through a projective
 $kG$--module. The category $D^b(kG)$ is the ``small'' bounded derived category, with
 objects unbounded chain complexes of finitely generated $kG$--modules,
 with homology concentrated in a bounded range of degrees. A
 morphism in the underlying category is a $kG$--linear chain map. It
 it induces an isomorphism
 in the homotopy category if it induces isomorphisms on homology
 (i.e., is a quasi-isomorphism). (We
 shall not describe precisely the set of morphisms in the homotopy category as
 it shall not use it here, but see e.g., \cite[\S2]{huybrechts06}
 for an elementary treatment, and \cite[\S1]{lurieHA} for an $\infty$--categorical
 perspective.)
 The category $\Dperf(kG)$ is the full subcategory of complexes
 quasi-isomorphic to a finite complex of finitely generated projective
 $kG$--modules. The quotient $D^b(kG)/\Dperf(kG)$ is the Verdier
 quotient, inverting morphisms in  $D^b(kG)$ with cofiber in
 $\Dperf(kG)$. (Recall that the cofiber can be obtained as the cokernel
 of an underlying monomorphism of chain complexes.)

Let us construct the inverse used
in Theorem~\ref{main}, displaying how an object in
$D^b(kG)/\Dperf(kG)$  is isomorphic to one in the image under
\eqref{BR-equivalence}: By choosing a projective resolution, represent an isomorphism class by a
bounded below complex $P_*$ of finitely generated projective
$kG$--modules. In $D^b(kG)/\Dperf(kG)$ this complex is canonically equivalent to its truncation $\bar
P_* = (\cdots \to  P_{r+1} \to P_{r} \to 0 \to \cdots)$, for any $r$. Taking $r$ to be
the degree of the top non-trivial homology class of $P_*$, the complex $\bar
P_*$ has
homology only in degree $r$, and is hence equivalent in
$D^b(kG)/\Dperf(kG)$  to $\Omega^{-r}(P_r/\im(d_{r+1}))$, the $-r$th Heller shift
of the $r$th homology group,
viewed as a chain complex in degree $0$. (Recall that the inverse Heller shift $\Omega^{-1}(M)$ is the cokernel of the
map from $M$ to its injective hull, the ``suspension'' in the
triangulated structure.)  Hence it is in the image of
$\Omega^{-r}(P_r/\im(d_{r+1})) \in \stmod_{kG}$ under \eqref{BR-equivalence}.

\subsubsection{$G$--twisted coefficient systems on subgroup complexes}\label{coeff-subsub}
A {\em $G$--twisted coefficient system} over $k$ on a space $X$ is a
$G$--coefficient system $A\co (\Delta X)_G \to
k{\operatorname{--mod}}$, as in Section~\ref{coeff-subsec},
with the added feature that it sends all morphisms
to isomorphisms. It hence factor through
fundamental groupoid of the category $(\Delta X)_G$, or equivalently
the fundamental groupoid of $X_{hG}$ (see Section~\ref{pi1-subsec}).
For $X_{hG}$ connected, a $G$--twisted
coefficient system over $k$ is thus equivalent to a
$k\pi_1(X_{hG},x)$--module $M$, for a choice of basepoint $x \in X_{hG}$.
Such coefficient systems are $hG$--homotopy
invariants in the sense that if $Y \to X$ is a $G$--equivariant map
and a homotopy equivalence, and $Y$ is given the coefficient system
induced by a $G$--twisted coefficient system $A$ on $X$, then $C_*(Y;A)
\to C_*(X;A)$ is an $kG$--homomorphism and a homology equivalence. See e.g.,
\cite[\S7]{quillen78},\cite[6.2.3]{benson91vol2}, or
\cite[Ch.~3.H]{hatcher02} for more information.

In particular for $X = |\calC|$ a $G$--twisted coefficient system is the same as a functor from 
$\calT_\calC(G)$ to $k$--vector spaces sending all morphisms to
isomorphisms. If furthermore $\calT_\calC(G)$ is connected (e.g., $\calC$ is
a collection of $p$--subgroups containing $S$), then a $G$--twisted coefficient system can be identified with a
$k\pi_1(\calT_\calC(G),P)$--module, for $P \in \calC$. A one-dimensional $k\pi_1(\calT_\calC(G),P)$-module is just a homomorphism
$\pi_1(\calT_\calC(G),P) \to k^\times$. 
So, for $\phi\co \Oep(G) \to k^\times$ a functor (where $k^\times$ is
viewed as a category with one object) we consider the
corresponding functor $k_\phi$ from $\Oep(G)$ to one dimensional $k$--vector
spaces and isomorphisms, and view this as a $G$--twisted coefficient
system on $|\calS_p(G)|$, still denoted $k_\phi$, in the canonical way via
\begin{equation}\label{kphi}
  (\Delta|\calS_p(G)|)_G \to \calS_p(G)_G = \Tep(G) \to \Oep(G)
  \end{equation}
Here $\Delta|\calS_p(G)| \to
\calS_p(G)$ sends $(P_0 \leq \cdots \leq P_n) \mapsto P_0$
and $\calS_p(G)_G \to \Oep(G)$ sends $({}^gP \leq P',g) \mapsto (G/P
\xrightarrow{[g^{-1}]} G/P')$.
Note that, if one ignores the group action, this is a twisted coefficient system in
the ordinary (non-equivariant) sense, depending on the fundamental
groupoid of $|\calS_p(G)|$.

\subsection{Surjectivity of  $\Phi$}
With the above preliminaries in place, we are ready for 
surjectivity of $\Phi$.
\begin{prop} \label{rightinverse} Let $\phi\co \pi_1(\Oep(G)) \to
  k^\times$ be a homomorphism and equip $|\calS_p(G)|$ with the
  corresponding $G$--twisted coefficient system $k_\phi$ via
  \eqref{kphi}.
Then for any non-tri\-vial $p$--subgroup $P$ we have the equivalences
$$k_\phi(P) \xrightsimeq C_*(|\{P\}|;k_\phi) \xrightsimeq  C_*(|\calS_p(G)|^P;k_\phi) \xrightsimeq C_*(|\calS_p(G)|;k_\phi)$$
 in $D^b(kN_G(P))/\Dperf(kN_G(P))$.

Consequently $C_*(|\calS_p(G)|;k_\phi)$ gives a Sylow-trivial module $M$
via
$\stmod_{kG} \xrightsimeq D^b(kG)/\Dperf(kG)$
(cf.~\S\ref{rickard-subsec}) and  $\Phi([M]) =\phi$.
\end{prop}
\begin{proof}
The first map is a chain homotopy equivalence, indeed an isomorphism
onto the normalized chain complex concentrated in degree $0$.

For the second map, recall that for a non-trivial $p$--subgroup $P$,
  $|\calS_p(G)|^P$ is $N_G(P)$--equivariantly contractible by Quillen's argument: $|\calS_p(G)|^P = |\calS_p(G)^P|$,
  which is contractible via a contracting homotopy induced by $Q \leq PQ \geq P$
  (see \cite[1.3]{quillen78} \cite[Rec.~2.1]{GS06}).
  Thus the second map, induced by the
$N_G(P)$--homotopy equivalence $|\{P\}|
\to |\calS_p(G)|^P$, is an equivalence in $D^b(kN_G(P))$, as $G$--twisted coefficient
systems are $hG$--homotopy invariant, recalled above in \S\ref{coeff-subsub}.

We show that the right map is an
equivalence in  $D^b(kN_G(P))/\Dperf(kN_G(P))$ by generalizing Quillen's proof that the
(reduced) Steinberg complex is isomorphic to a finite complex of
projectives \cite[Prop 4.1 and 4.5]{quillen78} (see also
\cite[Thm.~2.7.4]{webb91} and \cite{symonds08}): Choose a Sylow
$p$--subgroup $S'$ of $N_G(P)$, and note that it is enough to prove
that the map is an equivalence in $D^b(kS')/\Dperf(kS')$ (since an element in the bounded derived category is perfect if, in the notation of
Section~\ref{rickard-subsec}, $P_r/\im(d_{r+1})$ is projective, which
is detected on a Sylow $p$--subgroup).  Also since
$|\calS_p(G)|^{S'} \to |\calS_p(G)|^P$ is an $S'$--homotopy
equivalence, it is enough to prove the statement with $P$ replaced by
$S'$. Now, set $\Delta = |\calS_p(G)|$ and let $\Delta_s$ denote the singular set under the
$S'$--action, i.e., the union of all simplicies where $S'$ does not
act freely. By definition we have an exact sequence of chain complexes
$$ 0 \to C_*(\Delta_s,\Delta^{S'};k_\phi) \to
C_*(\Delta,\Delta^{S'};k_\phi) \xrightarrow{f}
C_*(\Delta,\Delta_s;k_\phi) \to 0$$

As observed in \cite[Prop.~4.1]{quillen78} (though only stated for $S'$
the Sylow $p$--subgroup), the singular set $\Delta_s$
is contractible. (To see this, one can also note that $\Delta_s$ is covered by the contractible subcomplexes
$\Delta^Q$, for $1 \neq Q \leq S'$, all of whose intersections are also
contractible, see \cite[\S4]{segal68} or \cite[Thm.~6.7.11]{tomdieck08}.)
Hence $\Delta^{S'} \to \Delta_s$ is a homotopy equivalence, so  $C_*(\Delta_s,\Delta^{S'};k_\phi)$ is
acyclic, using the homotopy invariance of non-equivariant homology
with twisted coefficient coefficient systems (see
\cite[Sec.~3.H]{hatcher02} and  \S\ref{coeff-subsub}). Therefore $f$ is an
equivalence in $D^b(kS')$. By definition
$C_*(\Delta,\Delta_s;k_\phi)$ is a complex of free $kS'$--modules, so
$C_*(\Delta,\Delta^{S'};k_\phi)$ is in $\Dperf(kS')$ as wanted,
establishing the first part of the proposition.

To see that $C_*(|\calS_p(G)|;k_\phi)$ is Sylow-trivial in $\stmod_{kG}$, it is enough to prove that $C_*(|\calS_p(G)|;k_\phi)$ is
equivalent to the trivial module $k$ in $D^b(kS)/\Dperf(kS)$, as the equivalence 
$\stmod_{kG} \to D^b(kG)/\Dperf(kG)$
is compatible with restriction. However this follows from the first
part upon taking $P=S$.

We finally observe that $\Phi([C_*(|\calS_p(G)|;k_\phi)]) =
\phi$ in $H^1(\Oep(G);k^\times)$. Namely, by Proposition~\ref{boundslemma} it is enough
to see that the two functors agree as $kN_G(S)$--modules when evaluated
on $G/S$, which follows by the first part, finishing the
proof. (In fact the above argument shows  that for any $G/P$ the identification of $C_*(|\calS_p(G)|;k_\phi)$ with $k_\phi$ in
$\stmod_{kN_G(P)}$ is compatible with restriction and
conjugation, and hence defines an isomorphism of functors on $\Oep(G)$,
avoiding the reference to Proposition~\ref{boundslemma}.)
\end{proof}

\begin{proof}[\thmunderline{Proof of Theorem~\ref{main}}]
By Proposition~\ref{phiishomo}, $\Phi$ is a group monomorphism.
Proposition~\ref{rightinverse} shows that $C_*(|\calS_p(G)|;k_\phi)$ does define
a Sylow-trivial module via the equivalence of categories
$\stmod_{kG} \xrightsimeq D^b(kG)/\Dperf(kG)$, and this assignment is a right inverse to
$\Phi$. So $\Phi$ is surjective as well.
\end{proof}

We note that Theorem~\ref{main} includes the bijection
between Sylow-trivial $kG_0$--modules and Sylow-trivial $kG$--modules
is given by induction modulo projectives  \cite[Lem.~2.7(2)]{MT07}, using \eqref{redtoG_0}.
\begin{lemma}[Sylow-trivial modules for groups with a strongly $p$--embedded subgroup]\label{stronglyembeddedinduction-lemma}
  \begin{equation*} C_*(|\calS_p(G)|;k_\phi) \cong
    C_*(|\calS_p(G_0)|;k_\phi)\uparrow_{G_0}^G
 \end{equation*}

\vspace{-15pt}

\qed \end{lemma}

As mentioned in the introduction the chain complex
$C_*(|\calS_p(G)|;k_\phi)$ may be interpreted as a homotopy colimit of $k_\phi$
over $ \calS_p(G)^{\op}$.  We detail this as a proposition for the
interested reader.
\begin{prop}[Homotopy Kan extensions]\label{kanext-prop}
We have equivalences of $kG$--chain complexes
 \begin{equation*}
C_*(|\calS_p(G)|;k_\phi) \cong \hocolim_{P \in
  \calS_p(G)^{\op}} k_\phi \xleftsimeq  \hocolim_{\eta \downarrow
  G/e}k_\phi \cong ({\mathbb L} Kan_{\eta}\, k_\phi)(G/e)
\end{equation*}
where the homotopy colimits are taken in chain complexes over $k$,
using the standard model from e.g.,
\cite[Ch.~18.1.1]{hirschhorn03}, ${\mathbb L} Kan_{\eta}\,
k_\phi\co \Op(G)^{\op} \to k\mbox{--(chain complexes)}$ is the homotopy
left Kan extension of $k_\phi$ along $\eta\co
\Oep(G)^{\op} \to \Op(G)^{\op}$, and $\eta \downarrow
  G/e$ is the overcategory of $G/e$.
\end{prop}
\begin{proof}
  The left isomorphism is by the model
  for $\hocolim$, and the right isomorphism is also by the
  definitions.
As the overcategory $\eta \downarrow G/e$ admits
a canonical $G$--equivariant functor to $\calS_p(G)^{\op}$, which is
an equivalence of categories (see \eqref{EOCEA}), this produces the
middle equivalence in  $D^b(kG)$.
  \end{proof}

The na\"ive guess for the inverse in Theorem~\ref{main} might have been
 the non-derived  $\colim_{P \in \calS_p(G)^{\op}}k_\phi$. This is however zero, unless $\phi$
corresponds to a Sylow-trivial module induced from a
$1$-dimensional $kG_0$--module, as we see next---it was this
viewpoint that  led us to the formula in Theorem~\ref{main}.
\begin{prop}[Kan extensions] \label{kanext-prop2} We have isomorphisms of $kG$--modules
\begin{multline*}
H_0(|\calS_p(G_0))|;k_\phi)\!\! \uparrow_{G_0}^G \cong
H_0(|\calS_p(G)|;k_\phi)
 \cong  \\ \colim_{P \in
  \calS_p(G)^{\op}} k_\phi \xleftcong  \colim_{\eta \downarrow
  G/e}k_\phi \cong (LKan_{\eta}\, k_\phi)(G/e)
\end{multline*}
with $LKan$ the left Kan extension. The $kG_0$--module
$H_0(|\calS_p(G_0))|;k_\phi)$ is 0 unless the action of
$kN_G(S)$ on $k_\phi(G/S)$ extends to $kG_0$, where it is
the unique $1$--dimensional $kG_0$--module with this property.
\end{prop}
\begin{proof} The first isomorphism is by
  Lemma~\ref{stronglyembeddedinduction-lemma}, using that induction is
  exact. The second and fourth are by definition, while the third
  follows as in Proposition~\ref{kanext-prop}.
  
As  $\calS_p(G_0)$ is connected by \eqref{redtoG_0}, 
$k_\phi(G_0/S) \twoheadrightarrow \colim_{P \in
  \calS_p(G_0)^{\op}}k_\phi(G_0/P)$ as $kN_G(S)$--modules. Hence if
the quotient is non-zero this means that the $kN_G(S)$--action extends
to $kG_0$. Likewise if $k_\phi(G/S)$ extends to a $kG_0$--module, then
the colimit is obviously $1$--dimensional. Finally note furthermore that
any extension is necessarily unique by the Frattini argument
\cite[Thm.~I.3.7]{gorenstein68}.
 \end{proof}

Using the above remarks, 
Theorem~\ref{main} gives a more
explicit model for the Sylow-trivial module when $|\calS_p(G)|$ is
$G$--homotopy equivalent to one-dimensional complex. This in fact
appears to cover all
currently known examples of exotic Sylow-trivial modules! Recall the Heller shift $\Omega$ from \S\ref{rickard-subsec}.

\begin{cor} \label{onedimcor} Suppose that $G$ is a finite group such that $|\calS_p(G)|$ is
  $G$--homotopy equivalent to a one-dimensional complex (e.g., $G$ has
  $p$--rank at most $2$, or at most one proper inclusion between
  $p$--radicals), and suppose $\phi \in \Hom(\pi_1(\Oep(G)),k^\times)$
  is not in the subgroup $\Hom(G_0,k^\times)$, cf.\ \eqref{naivebounds2}.
Then the corresponding Sylow-trivial module is given as
$$\Omega^{-1}(H_1(|\calS_p(G_0)|;k_\phi)) \uparrow_{G_0}^G$$
\end{cor}
\begin{proof}
By Proposition~\ref{kanext-prop2}, $H_0(|\calS_p(G_0)|;k_\phi)$ is
trivial if $\phi$ is not in 
$\Hom(G_0,k^\times)$. Hence  $C_*(|\calS_p(G_0)|;k_\phi)$ is
isomorphic in $D^b(kG_0)$ to $H_1(|\calS_p(G_0)|;k_\phi)$, viewed
as a chain complex concentrated in degree $1$. The corollary now
follows from Theorem~\ref{main} and
Lemma~\ref{stronglyembeddedinduction-lemma} (see also \S\ref{rickard-subsec}).
\end{proof}

\begin{rem}
If $\calC$ a $1$--dimensional collection in $G_0$,
$G_0$--homotopy equivalent to $|\calS_p(G_0)|$, then 
\begin{multline} 
H_1(|\calS_p(G_0)|;k_\phi) \cong  \\ \ker \left(
\oplus_{[P<Q]\in |\calC|_1/G_0}k_\phi(P)\uparrow_{N_{G_0}(P<Q)}^{G_0}
\xrightarrow{d_0-d_1}
\oplus_{[P] \in |\calC|_0/G_0}k_\phi(P)
\uparrow_{N_{G_0}(P)}^{G_0}\right)
\end{multline}
as $kG_0$--modules, by definition, where $k_\phi(P)$ is the $1$--dimensional $N_G(P)$--module given by
$N_G(P) \to \pi_1(\Oep(G)) \to k^\times$. Often $\calC$ can be chosen
so that $|\calC|_1/G$ is small, maybe even a single element (see
\S\ref{symmetricgroups} for an example). (The relationship between
collections in $G$ and $G_0$ is explained in 
\S\ref{components-subsec}.)

Lastly we remark that $\Omega^{-1} M$ is isomorphic to $\Omega^{-1}k \otimes M$,
modulo projectives, so in terms of finding generators for the group of endotrivial
modules, $\Omega^{-1}M$ works as well as $M$. 
\end{rem}

One may wonder what happens if one applies the map
$\Phi^{-1}(-)= [C_*(|\calS_p(G)|;-)]$ of Theorem~\ref{main} to an
arbitrary $k\pi_1(\Oep(G))$--module (semi-simple since $\pi_1(\Oep(G))$ is a
$p'$--group by \eqref{naivebounds2}).
The proof of Theorem~\ref{main} in fact shows that $\Phi$ will still
be a left inverse, and one can identify the image. The following more precise theorem generalizes
Theorem~\ref{main}, and arose as response to
questions by Radha Kessar (Remark~\ref{general-inverse-remark}
below) and David Craven. Call a $kG$--module $M$
Sylow-semi-simple if $M\downarrow_{S} \cong k^r \oplus (kS)^s$ for
    non-negative integers $r,s$.

\begin{thrm}[Classification of ``Sylow-semi-simple''
  modules] \label{sylowss} We have a bijection
$$  \left\{\begin{array}{c}
\text { Sylow-semi-simple } kG\text{--modules } \\
\text { without projective summands, } \\
\text { up to isomorphism }
\end{array}\right\} \,\,\, \xrightcong \,\,\, \left\{\begin{array}{c}
\text { finitely generated } \\ k\pi_{1}(\mathscr{O}_{p}^{*}(G))\text{--modules, } \\
\text { up to isomorphism }
                                \end{array}\right\}
$$
given by restriction to $N_G(S)$,  discarding projective summands, a
viewing the rest as a  $k\pi_1(\Oep(G)))$--module, using that $\ker(N_G(S)
\twoheadrightarrow \pi_1(\Oep(G)))$ acts trivially.

Under this bijection indecomposable modules correspond to simple
 modules, giving the bijection between
 Sylow-trivial modules and $1$--dimensional
 $k\pi_1(\Oep(G))$--modules of Theorem~\ref{main}, a
 restriction of Green correspondence.
Furthermore
  \begin{enumerate}
\item\label{forward}
The forward map in the bijection identifies with the functor $\Phi$ sending a $kG$--module $M$ to the
  $k\pi_1(\Oep(G))$--module corresponding to the functor from
  $\Oep(G)^{\op}$ to the connected groupoid of $k$--vector
  spaces and isomorphisms, given by $G/P \mapsto \hat H^0(P;M)$,  as in Theorem~\ref{main}.

\item \label{backward} The inverse map is described as assigning to a
$k\pi_1(\Oep(G))$--module $N$, the element $C_*(|\calS_p(G)|;N) \in
D^b(kG)/\Dperf(kG)$, with $G$--twisted coefficient system on
$|\calS_p(G)|$ via $\Tep(G) \to
\Oep(G) \to \pi_1(\Oep(G))$ as in Theorem~\ref{main},
and identifying this with a unique $kG$--module $M$ without
projective summands, via the Buchweitz--Rickard equivalence of
\S\ref{rickard-subsec}.
\end{enumerate}
\end{thrm}
\begin{proof} This is shown by modifying the argument of the proof of
  Theorem~\ref{main} slightly, and we follow the setup there: Suppose that $M$ is a $kG$--module as above.
Then it is clear that
  $\Phi$, via the functor $G/P \mapsto \hat H^0(P;M)$, produces a
  $k\pi_1(\Oep(G))$--module, which when inflated along $N_G(S)
  \twoheadrightarrow \pi_1(\Oep(G))$ agrees with the module $N$ of
  the decomposition $M\downarrow_{N_G(S)} \cong N \oplus P$, as stated
  in \eqref{forward}.
  
By Green correspondence
\cite[Thm.~3.12.2]{benson91} there is a
bijection between indecomposable trivial-source $kG$--modules with vertex $S$ and
simple $N_G(S)/S$--modules, given by restriction and disposing
summands not with vertex $S$. In particular the map in the theorem is
injective. As restriction preserves direct sum it is also clear that
indecomposable modules correspond to simple modules under the
bijection, once we have seen surjectivity.

For surjectivity, with inverse as described in \eqref{backward},
suppose that $N$ is a $k\pi_1(\Oep(G))$--module, and let $M$ be the
$kG$--module without projective summands corresponding to
$C_*(|\calS_p(G)|;N)$. Observe that the argument given in the first half of
Proposition~\ref{rightinverse} still gives equivalences in $D^b(kN_G(S))/\Dperf(kN_G(S))$
$$N \to  C_*(|\calS_p(G)|^S;N) \to C_*(|\calS_p(G)|;N)$$
which again implies that $N$ and $M\downarrow_{N_G(S)}$ are isomorphic
after throwing away projective summands, by the Buchweitz--Rickard
equivalence \S\ref{rickard-subsec}.
\end{proof}

\begin{rem} Subsequent sections give many ways of computing
$\ker(N_G(S) \to \pi_1(\Oep(G))$. In particular
Theorem~\ref{grouptheory-pi1} gives a group theoretic description in
terms of generators and relations.  \end{rem}

\begin{rem}
  As  noted, the map in Theorem~\ref{sylowss} is a restriction of
Green correspondence \cite[Thm.~3.12.2]{benson91} which provides a
bijection $M \mapsto N$ between indecomposable $kG$--modules $M$ that
split $M\downarrow_{N_G(S)} \cong N \oplus N'$ where $N$ is simple and $N'$ is a sum of indecomposable
modules with vertex a proper subgroup of $S$, and all simple
$kN_G(S)/S$-modules.
\end{rem}
\begin{rem}
   Since the restriction map preserves direct sum and
tensor product, Theorem~\ref{sylowss} in fact gives an isomorphism of
semi-rings between Sylow-semi-simple modules without projective
summands and
finitely generated $\pi_1(\Oep(G))$--modules, where the tensor product
on Sylow-semi-simple modules means tensoring and discarding projective
summands. Sylow-trivial modules constitute the units in the semi-ring
of Sylow-semi-simple modules.
\end{rem}

\begin{cor} \label{stable-equiv} Suppose $M$ is a $kG$--module that arises from the
  correspondence of Theorem~\ref{sylowss} with $N$ an absolutely simple
  $k\pi_1(\Oep(G))$-module. Then
 $\End_k(M) \cong k \oplus M'$ as $kG$--modules, where $M'$ is a $kG$--module without $k$ in its socle.
\end{cor}
\begin{proof}
If $N$ is absolutely simple, its dimension divides
$|\pi_1(\Oep(G))|$ (see \cite[\S6.5 Cor.~2]{serre79}), and is in
particular prime to $p$, since $\pi_1(\Oep(G))$ is a $p'$--group by
\eqref{naivebounds2}.  Hence the
dimension of $M$ is prime to $p$, since
the dimension of $M$ is congruent to the dimension of $N$ modulo
$|S|$, by Theorem~\ref{sylowss}.  In particular $k \to \End_k(M)
\xrightarrow{\Tr} k$ is an isomorphism, so $k$ splits off
  $\End_k(M)$. To see that $k$ is not in the socle of $M'$ note that
$$\sHom_G(M,M)\mkern-1mu \subseteq \mkern-1mu \sHom_{N_G(S)}(M,M) \mkern-1mu  \cong \mkern-1mu \sHom_{N_G(S)}(N,N)
\mkern-1mu\cong \mkern-1mu \sHom_{\pi_1(\Oep(G))}(N,N) \mkern-1mu \cong \mkern-1mu k$$
since $N$ is absolutely simple.
 \end{proof}

\begin{rem} \label{general-inverse-remark} As pointed out to us by Radha Kessar, modules as in
  Corollary~\ref{stable-equiv} are interesting since they are candidates for the image of simple modules under
  self-equivalences of the stable module category. Carlson proved in
  \cite{carlson98} that for $p$--groups, modules $N$ satisfying
  $\sHom_G(N,N) \cong k$  are in fact endotrivial, but for general finite groups the class is
  bigger, e.g., it contains all simple modules. Its size in general, and the
 precise image given via Theorem~\ref{sylowss}, is at present
  unclear.
\end{rem}

\begin{rem}[More general modules from $p$--local
  information]\label{local-to-global} The process of constructing modules from $p$--local
  information via a homotopy left Kan extension as in
  Theorems~\ref{main} and \ref{sylowss} should be of interest also for
  more general
  families of modules on $p$--local subgroup $N_G(P)$, even when they
 are not translates of a fixed module on $N_G(S)$. Or, said
  differently, when one considers more general $G$--local
  coefficient systems on $|\calS_p(G)|$, as explained in \S\ref{coeff-subsec}.
    For instance it would be worthwhile to understand the
  work of Wheeler \cite{wheeler02} from this point of
  view (see also \cite{mathew16}). Fundamental counting conjectures in representation theory, such as Alperin's
  \cite{alperin87}, predict a
  relationship between $kG$--modules and modules on normalizers
  of $p$--subgroups, and it is not unnatural to expect that the glue
  between these subgroups
  provided by the transport and orbit categories should play a role in
  categorifiying, and potentially proving, these conjectures.
\end{rem}

\begin{rem}[Discrete valuation rings]\label{dvr} Our model
  for Sylow-trivial modules can also be lifted to characteristic
  0: Suppose that $(K,R,k)$ is a $p$--modular system, with $R$
  a complete rank one discrete valuation ring with residue field
  $k$ of characteristic $p$ \cite[\S1.9]{benson91}.
Then reduction modulo the maximal ideal $R \to k$ induces an
isomorphism on finite roots of unity in $R$ and $k$ by Hensel's
lemma. Hence $H^1(\bO_p^*(G);R^\times) \xrightcong
H^1(\bO_p^*(G);k^\times)$, so we can uniquely define the twisted
Steinberg complex
$C_*(|\calS_p(G)|;R_\phi)$ over $R$. By the $RG$--lattice version of the Buchweitz--Rickard
theorem \cite[Prop.~3.4]{poulton13} it defines an object in the
stable module category of $RG$--lattices, lifting the 
$kG$--module corresponding to $\phi \in H^1(\bO_p^*(G);k^\times)$.
This makes explicit the lift known to exist by virtue
of Sylow-trivial modules being trivial source
\cite[Cor.~3.11.4(i)]{benson91}.
\end{rem}

\subsection{Addendum: $A_k(G,S) \cong H^1(\Oep(G);k^\times)$.}
In this addendum we relate Bal\-mer's notion of a
weak homomorphism to $H^1(\Oep(G);k^\times)$, providing a different proof of Theorem~\ref{main}.
This version is less direct as it uses the main result of
\cite{balmer13}, where he identifies $\TS$ with a group he calls
$A_k(G,S)$ of weak $S$--homomorphisms from $G$ to $k^\times$,
but may nevertheless be instructive for readers familiar with that
work.
By \cite[Def.~2.2]{balmer13}, a weak $S$--homomorphism is a map from $G$ to $k^\times$ such
that
  \begin{itemize}
\item[(WH1)] $\phi(g)=1$ for $g \in S$,
\item[(WH2)]
$\phi(g) = 1$ when $S \cap S^g =1$, and 
\item[(WH3)] $\phi(g)\phi(h) = \phi(gh)$ when $S \cap S^{h} \cap S^{gh} \neq
1$, 
\end{itemize}
where as usual $H^g = g^{-1}Hg$.  It would be
interesting to play off the construction of endotrivial modules
in Theorem~\ref{main} with
\cite[Constr.~2.5 and Thm.~2.9]{balmer13}.
\begin{prop} For any finite group and field $k$,  $A_k(G,S) \cong
  H^1(\Oep(G);k^\times)$, where $A_k(G,S)$ is the group of weak
  $S$--homomorphisms from $G$ to $k^\times$.
\end{prop}
\begin{proof}

We prove that $A_k(G,S)$ identifies with
$\Hom(\pi_1(\Oep(G)),k^\times)$, by observing that there are
canonical group homomorphisms in both directions, that we check are well defined and inverses to each other.
Recall the bijection $$\Hom_{\Oep(G)}(G/P,G/Q) \cong \{ g \in G | P^{g} \leq Q\}/Q$$
described in \S\ref{cat-subsec},
and that, by \S\ref{pi1-subsec}, we have an isomorphism of sets
which is an
isomorphism of abelian groups, where the group structure on
the left is pointwise multiplication in the target, and $\Rep$ means isomorphism classes of functors.
Up to equivalence of categories (which does not change $\Rep$) we can furthermore replace $\Oep(G)$ by the
equivalent full subcategory  $\bO_S^*(G)$ with objects $G/P$ for $1< P \leq S$, for our
fixed Sylow $p$--subgroup $S$ (see \S\ref{pi1-subsec}).

We are thus just left with verifying that isomorphism classes of functors $\bO^*_S(G) \to
k^\times$ agree with the group $A_k(G,S)$ that Balmer introduced:
Given $\phi \in A_k(G,S)$ define 
$\Phi\co  \bO_S^*(G) \to k^\times$ by sending a morphism $G/Q
\xrightarrow{[g]} G/Q'$
to $\phi(g)$.
 This is well defined, since replacing $g$ by $gq$, for $q \in Q$, yields
 $\phi(gq) = \phi(g)\phi(q) =\phi(g)$ by (WH3) and (WH1).
 It is likewise a functor: 
By  (WH1), $\Phi( [\id_{G/Q}]) = \phi(1) = 1$ and given a composite $G/Q \xrightarrow{[g]}
 G/Q' \xrightarrow{[h]} G/Q''$ we have $Q^g \leq Q' \leq S$ and $(Q')^h \leq
 Q'' \leq S$, so 
$$  S^{gh} \cap S^h \cap S = (S^g \cap S)^h \cap S \geq (Q^g)^h \cap S
= (Q^g)^{h} > 1$$
Hence by (WH3), writing composition in categories from right to left,
$$\Phi( [h] \circ [g] ) = \Phi([gh]) = \phi(gh) = \phi(g)\phi(h)
= \phi(h)\phi(g) = \Phi([h])\Phi([g])$$ as wanted. 

Conversely given a functor $\Phi\co \bO_S^*(G) \to k^\times$, up to isomorphism, we construct a
weak homomorphism $\phi\co G \to k^\times$ as follows:
By \S\ref{pi1-subsec}, $\Phi$ is isomorphic to a unique functor
$\tilde \Phi \co \bO_S^*(G) \to k^\times$ that factors through
$\bO_S^*(G) \to \pi_1(\bO_S^*(G))$, with the model for $\pi_1(\bO_S^*(G))$ of
\S\ref{pi1-subsec}. This model in
particular sends morphisms induced by inclusions to the identity. We may without
restriction replace $\Phi$ by this functor $\tilde \Phi$. Set 
$$\phi(g) =\tuborg \Phi(G/({}^gS \cap S) \xrightarrow{[g]} G/(S\cap S^{g}))
&\mbox{if } S \cap S^{g} \neq 1\\
1 & \mbox{otherwise} \sluttuborg 
$$
It is clear that (WH1) and (WH2) are satisfied. For (WH3) suppose that $S\cap S^h \cap S^{gh} \neq 1$
and consider the diagram
$$\xymatrix@R-7pt{ G/({}^{gh}S\cap{}^gS \cap S) \ar[dr]_{[g]} \ar[rr]^{[gh]} & & G/(S \cap
  S^h \cap S^{gh}) \\ 
 & G/({}^hS\cap S \cap S^g) \ar[ur]_{[h]}& }$$
where the top map is the quotient of $G/({}^{gh}S\cap S)
\xrightarrow{[gh]} G/(S \cap S^{gh})$, and similarly for the two other
maps. Hence applying $\Phi(-)$ to this diagram, and using that
morphisms $G/Q \to G/Q'$ induced by inclusions $Q \leq Q'$ go to the identity we see that
$\phi(gh) = \Phi( [gh]) = \Phi(h) \Phi(g) = \phi(h) \phi(g) =\phi(g)\phi(h)$ as
wanted. 

As we have now given maps between $A_k(G,S)$ and
$\Rep(\bO_S^*(G),k^\times)$ that are group
homomorphisms under pointwise multiplication in $k^\times$, and
mutual inverses, we have finished the proof.
\end{proof}

\begin{rem}\label{partial-groups-rem}
Another perspective on ``weak homomorphisms'' can be given by
showing that they correspond to morphisms of partial groups in the
sense of Chermak \cite{chermak13} from a locality of $G$ based on
all non-trivial subgroups $p$--subgroups to $k^\times$.
\end{rem}

\section{Fundamental groups of orbit and fusion categories}\label{fundgrp-sec}
In this section we describe how to calculate and manipulate our basic invariants
$\pi_1(\Oep(G))$ and $\pi_1(\calF_p^*(G))$, and related groups. In \S\ref{fund-orbit-subsec} we establish basic
properties of $\pi_1(\Oep(G))$ and establish Corollary~\ref{brown} and Theorem~\ref{sequence}. In \S\ref{fund-orbit-subsecII} we expand on its
properties and in \S\ref{fund-fusion-subsec} we carry out a similar
analysis for $\pi_1(\calF_p^*(G))$. In \S\ref{higher-homotopy}
we look at higher homotopy groups---these occur naturally in the
analysis, even if one is ultimately only interested in $\pi_1$. Some results are stated for an arbitrary
collection of $p$--groups $\calC$, appealing to
Appendix~\ref{propagating-sec} to get minimal
hypothesis on $\calC$---the reader may take $\calC =\calS_p(G)$ at
first reading. We refer to Appendix~\ref{propagating-sec} for much additional
information.
The categories discussed were introduced in \S\ref{cat-subsec}.

Recall that $H_{p'}$ means the quotient of $H$ by the
subgroup generated by elements of $p$--power order (so $H_{p'} = H/O^{p'}(H)$ when $H$ is finite). 
For a fixed Sylow $p$--subgroup $S$ and a collection $\calC$, set
\begin{equation}\label{pembeddedC}
G_{0,\calC} =
  \langle N_G(Q) | Q \leq S\mbox{, }Q \in \calC \rangle
  \mbox{\,\,\,\,\,  and \,\,\,\,\,  }\calC_0 = \{ Q \in \calC  |Q \leq
  G_{0,\calC}\}.
\end{equation}
Hence $\calC_0$ is a collection in $G_{0,\calC}$, and when $\calC  =
  \calS_p(G)$,  $G_{0,\calC} = G_0$ of 
  \eqref{pembeddeddef}. A $p$--subgroup is called
 {\em $p$--essential} if
  $\calS_p(N_G(P)/P)$ is disconnected (hence non-empty). See
  \S\ref{components-subsec} for an elaboration. 

\subsection{Fundamental groups of orbit categories: proof of
  Corollary~\ref{brown} and Theorem~\ref{sequence}}\label{fund-orbit-subsec}

\begin{prop}[Bounds on $\pi_1(\bO_\calC(G))$] \label{boundslemma} Let
  $\calC$ be a collection of $p$--subgroups closed under passage to
$p$--essential and Sylow overgroups.
Then the categories $\bO_\calC(G)$ and $\bO_{\calC_0}(G_{0,\calC})$, as well as
$\calT_\calC(G)$ and $\calT_{\calC_0}(G_{0,\calC})$, are equivalent, all connected, and hence $\pi_1(\bO_\calC(G)) \xleftcong
\pi_1(\bO_{\calC_0}(G_{0,\calC}))$ and $\pi_1(\calT_\calC(G)) \xleftcong \pi_1(\calT_{\calC_0}(G_{0,\calC}))$.
Furthermore we have a sequence of surjections
$$ N_G(S)/S \twoheadrightarrow \pi_1(\bO_\calC(G)) \twoheadrightarrow
(G_{0,\calC})_{p'} \twoheadrightarrow G_{p'}$$
and in particular $\pi_1(\bO_\calC(G))$ is a finite $p'$-group. 
\end{prop}
\begin{proof}  
The categories are connected since $S \in \calC$.
That  $\bO_\calC(G)$ and $\bO_{\calC_0}(G_{0,\calC})$  are
  equivalent categories follows from Alperin's fusion theorem: By
  Sylow's theorem they are both equivalent to their full subcategories
$\bO_{\calC,S}(G)$ and $\bO_{\calC_0,S}(G_{0,\calC})$
  with objects
  $G/Q$ and $G_{0,\calC}/Q$ respectively, with $Q \leq S$ and $Q\in \calC$, for some fixed Sylow
  $p$--subgroup $S$. Now Alperin's fusion theorem
  \cite[\S3]{alperin67}, in the version of 
 Goldschmidt--Miyamoto--Puig
\cite[Cor.~1]{miyamoto77},
says that for any conjugation $G/P \xrightarrow{[g]} G/P^g$ with $P,P^g \leq S$ 
we can write $g = g_1 \cdots g_r n$ where $g_i \in N_G(P_i)$ with
$P\leq P_1$, $P_i\leq S$ $p$-essential, $P^{g_1\cdots g_{i}} \leq P_{i+1}$, and $n \in N_G(S)$. In particular $g \in
G_{0,\calC}$. Hence the two subcategories $\bO_{\calC,S}(G)$ and $\bO_{\calC_0,S}(G_{0,\calC})$ are isomorphic, and thus 
 $\bO_\calC(G)$ and $\bO_{\calC_0}(G_{0,\calC})$  are equivalent. The same
argument applies verbatim to $\calT$. (See also
\cite[\S10]{grodal02} for information on versions of the fusion
theorem.) As the categories are equivalent, their fundamental groups are isomorphic.

To see the stated surjections, recall that 
$G = N_G(S)O^{p'}(G)$, for any finite group
$G$, by the Frattini argument \cite[Thm.~I.3.7]{gorenstein68}. In particular we have surjections
$N_G(S)/S \twoheadrightarrow G_{p'}$ and $N_G(S)/S
\twoheadrightarrow (G_{0,\calC})_{p'}$. Thus we have
established the proposition if we show the surjection $N_G(S)/S \twoheadrightarrow
\pi_1(\bO_\calC(G))$. For this, first note that by
Lemma~\ref{lem-add-over} we can without
restriction assume that $\calC$ is closed under passage to all
$p$--subgroups, not just $p$--essential and Sylow subgroups.
Next, recall the model for $\pi_1(\bO_\calC(G))$ of
\S\ref{pi1-subsec}: Take $G/S$ as basepoint and consider
  the functor $\omega\co\bO_{\calC,S}(G) \to \pi_1(\bO_\calC(G))$ from \S\ref{pi1-subsec}
  given by sending $G/P
  \xrightarrow{[g]} G/Q$, for  $P,Q \leq S$, to the loop $G/S
\leftarrow G/P \xrightarrow{[g]} G/Q \to G/S$. We have $\omega(G/P \to
G/Q) = 1$ for $P\leq Q$, and the image of $\omega$ 
generates $\pi_1(\bO_\calC(G))$.  Again by the fusion theorem,
$\pi_1(\bO_\calC(G))$ is in fact generated by
$N_G(P)/P$ for $P \leq S$, $P \in \calC$ (compare also \cite[Pf.~of~Prop.~1.12]{BLO03jams}). Furthermore, the
fusion theorem in Alperin's version \cite[\S3]{alperin67},
says that any conjugation $G/P \xrightarrow{[g]} G/P^g$ 
can be obtained as
a sequence of conjugations by elements of {\em $p$--power} order in
$N_G(P_i)$ for $p$--subgroups for a sequence of $p$--subgroups $P_i
\leq S$ containing a conjugate of $P$, related as above (but now not
necessarily $p$-essential), and an element
in $N_G(S)$.  However, any element $[x] \in N_G(P)/P$ of $p$--power order is trivial in the fundamental
group, since it will be conjugate to an element in $S$, which is zero:
To see this explicitly, pick $g$ which conjugates $\langle x,P\rangle$ into
$S$, then we can consider the diagram
$$\xymatrix{ G/P \ar[d]^{[x]} \ar[r]^{[g^{-1}]} & G/{}^gP  \ar[d]^{[gxg^{-1}]}
  \ar[r] & G/S \ar[d]^{[1]}\\
G/P \ar[r]^{[g^{-1}]} & G/{}^gP \ar[r] &G/S}$$
which commutes since $gxg^{-1} \in S$, hence showing that $G/P \xrightarrow{[x]} G/P$ maps
to the identity in $\pi_1(\bO_\calC(G))$. This shows that $N_G(S)/S \to \pi_1(\bO_\calC(G))$ is surjective as wanted.
\end{proof}

\begin{rem} 
For $\calC = \calA_2(G)$ and $G = \Z/4$,  $|\bO_\calC(G)| \simeq B\Z/2$,
so $\pi_1(\bO_\calC(G))$ is not a $2'$--group.

\end{rem}

\begin{rem}\label{natural-rem}
The maps in Proposition~\ref{boundslemma} are natural in
the collection, so if $\calC' \leq \calC$ we have a
surjection $\pi_1(\bO_{\calC'}(G)) \twoheadrightarrow \pi_1(\bO_{\calC}(G))$,
introducing a natural filtration on $\pi_1(\Oep(G))$.
\end{rem}

\begin{rem}\label{restriction-to-subgroups} For $\calC$ a collection
  of $p$--subgroups of $G$, closed under passage to $p$--overgroups,
 and $N_G(S) \leq H \leq G$, then by Proposition~\ref{boundslemma} we have surjections
$$N_G(S)/S \twoheadrightarrow \pi_1(\bO_{\calC'}(H)) \twoheadrightarrow
\pi_1(\bO_\calC(G))$$
for $\calC'$ the elements of $\calC$ that are subgroups of $H$.
Via Theorem~\ref{main} this can be seen as a refinement
of the fact that restriction to $H$ is injective on Sylow-trivial
modules, as is usually seen via Green correspondence
\cite[Prop.~2.6(a)]{CMN06}.
\end{rem}

\begin{prop}[Quotienting out by
  $p$--torsion]  \label{pprime} For any collection $\calC$ of $p$--sub\-groups, and any
  basepoint $P \in \calC$, the natural surjections of categories
  induce isomorphisms
$$\pi_1(\calT_\calC(G))_{p'}
\xrightcong\pi_1(\bO_\calC(G))_{p'} \mbox{ and }\pi_1(\calF_\calC(G))_{p'}
\xrightcong\pi_1(\bcalF_\calC(G))_{p'}$$
In particular $H^1(\bO_\calC(G);A)  \xrightcong
H^1(\calT_\calC(G);A)$ if $A$ is a $p$--torsion-free abelian group
(e.g., $A = k^\times$) and
$H_1(\calT_\calC(G);A) \xrightcong H_1(\bO_\calC(G);A)$ if $A$ instead
is assumed to be $p$--divisible.
\end{prop}

\begin{proof}Note that we include the basepoint $P$ in the
  formulation, since without further assumptions on $\calC$ the
categories could be disconnected (though this is never the case for the $\calC$ we
are interested in) (see \S\ref{pi1-subsec} for more detail).
We first prove
$\pi_1(\calT_\calC(G))_{p'} \xrightcong\pi_1(\bO_\calC(G))_{p'}$.
Note that $\calT_\calC$ and
$\bO_\calC$ have the same path components, since
the quotient functor is a bijection on objects and a surjection on
morphisms. 
Choose for each
$Q \in \calC$ which lie in the same path component as $P$, a preferred
path in $\calT_\calC$ from $Q$ to $P$, as explained in \S\ref{pi1-subsec},
which induces a corresponding path in $\bO_\calC$.  
Now, the morphisms in $\calT_\calC(G)$ surject onto
the morphisms of $\bO_\calC(G)$, and if two morphisms in $\calT_\calC(G)$ are
mapped to the same morphism in $\bO_\calC(G)$, then they differ by an
automorphism of $p$--power order. Since the morphisms in the
category generate the fundamental group, we conclude that
$\pi_1(\calT_\calC(G))_{p'} \xrightcong
\pi_1(\bO_\calC(G))_{p'}$ as wanted. The case 
$\pi_1(\calF_\calC(G))_{p'}
\xrightcong\pi_1(\bcalF_\calC(G))_{p'}$ is identical.

The consequences for cohomology and homology now follow, using the
universal coefficient theorem and the Hurewicz theorem explained in Section~\ref{pi1-subsec}.
\end{proof}

We have now justified all the ingredients to prove Corollary~\ref{brown}
and Theorem~\ref{sequence} from \S\ref{subgroupcpx-subsection}:

\begin{proof}[\thmunderline{Proof of Corollary~\ref{brown}}] We have isomorphisms
$$\TS \xrightcong H^1(\Oep(G);k^\times) 
\xrightcong H^1(\Tep(G);k^\times) \cong
H^1(|\calS_p(G)|_{hG};k^\times)$$
by Theorem~\ref{main},
Proposition~\ref{pprime}, and Lemma~\ref{thomasonlemma} respectively.
\end{proof}
  
\begin{proof}[\thmunderline{Proof of Theorem~\ref{sequence}}]
As mentioned in \eqref{fibseq} and  \eqref{fundseq}, we have a fibration sequence 
$|\calS_p(G_0)| \to |\calS_p(G)|_{hG} \to BG_0$, whose long exact
sequence in homotopy groups produces the short-exact sequence 
$$1 \to \pi_1(\calS_p(G_0)) \to   \pi_1(|\calS_p(G)|_{hG})  \to G_0 \to
1,$$ as
$BG_0$ has no higher homotopy groups and $|\calS_p(G_0)|$ is
connected.  
The exact sequence of Theorem~\ref{sequence} 
identifies with first four terms of the five-term exact sequence in group cohomology with
$k^\times$-coefficients (see
\cite[{VI.8}]{HS71})  arising from the above group extension, and using $\TS \xrightcong
H^1(|\calS_p(G)|_{hG};k^\times)$ by
Corollary~\ref{brown}. (Alternatively apply the five-term exact sequence
of the fibration \eqref{fibseq} directly.)

The first of the two final statements follows from the exact sequence
in the first part,
together with the Universal Coefficient Theorem and Frobenius
reciprocity (see also Proposition~\ref{prop:homologyseq} below). The
second now also follows, as $|S_p(G)|$ simply connected implies
$G = G_0$ by \eqref{redtoG_0}.  
\end{proof}

Let us also spell out the homology version of Theorem~\ref{sequence}, as
this is often useful in practice.
\begin{prop}\label{prop:homologyseq}
We have an exact sequence
$$  H_2(\Oep(G)) \to H_2(G_0)_{p'} \xrightarrow{\partial}   (H_1(\calS_p(G_0))_{G_0})_{p'}  \to
H_1(\Oep(G)) \to H_1(G_0)_{p'} \to 0$$
(where also $H_1(\calS_p(G_0))_{G_0} \cong
H_1(\calS_p(G))_{G}$ by  \eqref{redtoG_0}  and Frobenius reciprocity).
\end{prop}
\begin{proof}
  The five-term exact sequence in homology with
  $\Z[\frac1p]$--coefficients for the extension \eqref{fundseq} is
  \begin{multline*}
    H_2(\Tep(G); \Z[\tfrac1p]) \to H_2(G_0; \Z[\tfrac1p]) \to
  H_1(\calS_p(G_0);\Z[\tfrac1p])_{G_0} \\ \to H_1(\Tep(G); \Z[\tfrac1p]) \to
  H_1(G_0; \Z[\tfrac1p]) \to 0
  \end{multline*}
  (see again \cite[{VI.8}]{HS71}). We have $H_i(\Tep(G);\Z[\frac1p])
  \xrightcong H_i(\Oep(G))$ for $i=1,2$ by
  Proposition~\ref{pprime-higher} and
  Theorem~\ref{homotopy-finiteness} so we can rewrite the sequence as 
  \begin{multline*}H_2(\Oep(G)) \to H_2(G_0) \otimes  \Z[\tfrac1p] \to
  (H_1(\calS_p(G_0))_{G_0}) \otimes \Z[\tfrac1p] \\ \to H_1(\Oep(G)) \to
  H_1(G_0) \otimes  \Z[\tfrac1p]  \to 0
  \end{multline*}
using also exactness of inverting $p$. 
All terms except the middle are known to be
finite, so the middle is as well, and we can hence replace $(-)\otimes
\Z[\frac1p]$ with $(-)_{p'}$ everywhere as wanted.
\end{proof}

\begin{rem}\label{boundarymap}
  The boundary map $\partial$ in Theorem~\ref{sequence} sends an
  element $$f \in H^1(\calS_p(G_0);k^\times)^{G_0}\hookrightarrow
H^1(\calS_p(G_0);k^\times) \cong \Hom(H_1(\calS_p(G_0)),k^\times)$$ to
$-f_*([\alpha]) \in H^2(G_0;k^\times)$, where $[\alpha] \in
H^2(G_0;H_1(\calS_p(G_0)))$ is the extension class of the
abelianization of the extension
\eqref{fundseq} (see e.g.,  \cite[Thm.~4]{HS53} or
\cite[Thm.~7.3.1]{evens91}).
Dually for $\partial$ in Proposition~\ref{prop:homologyseq}. This
extension class may deserve closer study.
\end{rem}

\begin{rem}\label{amplerem}
Propositions~\ref{boundslemma} and \ref{pprime} should be compared to the situation at the
prime $p$, where
\begin{equation}\label{ampleequation}
  H_*(\Tep(G))_{(p)} \xrightcong H_*(G)_{(p)}
  \end{equation}
by a classical result of Brown \cite[Thm.~X.7.3]{brown94} (translated
via Lemma~\ref{thomasonlemma}). The theory of `ample
collections' describe for which $\calC$ this continues to hold (see
\cite[1.3]{dwyer97}, \cite[\S9]{grodal02}). This combines to show
that $H_1(\Tep(G)) \xrightcong H_1(G)$ if and only if
$H_1(\bO_p^*(G)) \xrightcong H_1(G)_{p'}$. See
Theorem~\ref{homotopy-finiteness} for a more general version.
\end{rem}

\begin{rem}[Equivariant complex line bundles on $\calS_p(G)$] \label{reltobalmerremark}
Via a remark of Totaro \cite[Rem.~2.7]{balmer15},
Corollary~\ref{brown} easily implies a
very recent theorem of Balmer \cite[Thm.~1.1]{balmer15},
that identifies $\TS$ with the $p'$--torsion
part of the group of $G$--equivariant complex line
bundles on $|\calS_p(G)|$, under the assumption that $k$ is algebraically closed.
Namely, in this case we have an embedding
$\Tors_{p'}(\Q/\Z) \cong \mu_\infty(k) \subseteq k^\times$, where
$\Tors_{p'}$ means the subgroup of elements of finite order prime to
$p$, and $\mu_\infty(k)$ are the units of finite order. Hence we have isomorphisms
$$\Tors_{p'}(H^2(|\calS_p(G)|_{hG};\Z))  \xleftcong
\Tors_{p'}(H^1(|\calS_p(G)|_{hG};\Q/\Z)) \xrightcong
H^1(|\calS_p(G)|_{hG};k^\times)$$
where the first is induced by the exact sequence $0 \to \Z \to
\Q \to \Q/\Z \to 0$ and the second uses that
$H_1(|\calS_p(G)|_{hG})$ is finite, e.g., by
\eqref{naivebounds2}, \eqref{pprimequotient}, and \eqref{grothendieck}.
But now the left-hand term identifies with the $p'$--torsion part of
the $G$--equivariant complex line bundles on $|\calS_p(G)|$, as remarked by
Totaro \cite[Rem.~2.7]{balmer15}, and the right-hand side identifies
with $\TS$ by  Corollary~\ref{brown}.
More generally, since \eqref{ampleequation} shows that 
$H^2(|\calS_p(G)|_{hG};\Z)$ is a finite abelian group with
$p$--torsion part $H^2(G)_{(p)}$, we can describe {\em all}
$G$--equivariant complex line bundles on $|\calS_p(G)|$ as
\begin{equation}\label{general-pic}
{\operatorname{Pic}}^G(|\calS_p(G)|) \cong H^2(G)_{(p)} \oplus \TS.
\end{equation}
Here $k$ should in fact just be large enough so that the one-dimensional representations of
$\pi_1(\Oep(G))$ do not depend on $k$, e.g., containing all
$|N_G(S):S|$th roots of unity.
\end{rem}

\subsection{Fundamental groups of orbit categories: further structural
  results}\label{fund-orbit-subsecII}
The fundamental group  $\pi_1(\Oep(G))$ can be described in a purely group theoretic
way.

\begin{thrm}[Group theoretic description of
  $\pi_1(\Oep(G))$] \label{grouptheory-pi1} Let $K_\bO$ be the subgroup of $N_G(S)$
  generated by elements $g \in N_G(S)$ such that there exist
  nontrivial subgroups $1<Q_0, \ldots, Q_r \leq S$  and a factorization  $g = x_1
 \cdots x_r$ in $G$, where $x_i \in O^{p'}(N_G(Q_i))$ 
and $Q_0^{x_1\cdots x_i} \leq Q_{i+1}$ for $i\geq 0$.
Then 
$$N_G(S)/K_\bO \xrightcong \pi_1(\Oep(G))$$
In particular $N_G(S)\cap
  O^{p'}(N_G(P)) \leq \ker\left( N_G(S)
\twoheadrightarrow \pi_1(\Oep(G))\right)$ for $1 < P \leq S$.
\end{thrm}

\begin{proof}
   The proof amounts to a careful study of the proof of Proposition~\ref{boundslemma}.
  The canonical map $N_G(S)  \to \pi_1(\Oep(G))$ is surjective
  by Proposition~\ref{boundslemma}.  We first justify that it
  factors through $K_\bO$, to induce a homomorphism
$\phi\co N_G(S)/K_\bO  \twoheadrightarrow \pi_1(\Oep(G))$.
Let $g \in K_\bO$ be a generator with a decomposition $g = x_1 \cdots
x_r$ as in the
theorem. Then $G/S \xrightarrow{[g]} G/S$ is equal to $G/Q_0 \xrightarrow{[g]} G/Q_0^g$ in $\pi_1(\Oep(G))$, as
inclusions go to the identity, as explained in
Section~\ref{pi1-subsec}. Furthermore, for each $i$, we have a
commutative diagram
$$\xymatrix@R-5pt{G/Q_i \ar[r]^{[x_i]} & G/Q_i\\
G/Q_0^{x_1\cdots x_{i-1}} \ar[r]^{[x_i]} \ar[u]& G/Q_0^{x_1\cdots
  x_{i}} \ar[u]}$$
equating $G/Q_i \xrightarrow{[x_i]} G/Q_i$ and  $G/Q_0^{x_1\cdots x_{i-1}} \xrightarrow{[x_i]}  G/Q_0^{x_1\cdots
  x_{i}}$ in $ \pi_1(\Oep(G))$.
But $G/Q_i \xrightarrow{[x_i]} G/Q_i$ is
 zero in $ \pi_1(\Oep(G))$ as it is a product of elements of $p$--power order
 in $N_G(Q_i)/Q_i$,
 and  $ \pi_1(\Oep(G))$ is a $p'$--group by
 Proposition~\ref{boundslemma}. Hence $[g]$ is zero in  $\pi_1(\Oep(G))$
 as well, and $K_\bO \leq \ker\left( N_G(S)
\twoheadrightarrow \pi_1(\Oep(G))\right)$ as wanted.  The last
statement in the theorem
follows from this.

To show the whole theorem, i.e., that $\phi$ is an isomorphism, we
construct an inverse $\psi\co \pi_1(\Oep(G)) \to N_G(S)/K_\bO$ using
Alperin's fusion theorem.  For this, recall that $\Oep(G)$ is
equivalent to a full category $\bO^*_S(G)$ with objects $G/P$ for $P \leq S$, and
consider $G/P \xrightarrow{[g]} G/P^g$, with $P$, $P^{g}
\leq S$. Now, by Alperin's fusion theorem \cite[\S3]{alperin67}, we can
find $x_1, \ldots, x_r$ and non-trivial $p$--subgroups $Q_1, \ldots,
Q_r \leq S$ such that $P=Q_0 \leq  Q_1$, $x_i \in N_G(Q_i)$
is of $p$--power order, $P^{x_1\cdots x_{i}} \leq Q_{i+1} \leq S$, and
$S^g = S^{x_1\cdots x_r}$, i.e., $g = x_1 \cdots x_r n$ with $n
\in N_G(S)$. (Our $Q_i$ correspond to ``$Q_i \cap P$'', when $1 \leq i \leq
r$, in the notation of  \cite[\S3]{alperin67}.) We claim that the map
sending $[g]$ to $nK_\bO$ gives a well
defined functor $\bO^*_S(G) \to N_G(S)/K_\bO$. 
If $g = x_1 \cdots x_r n = y_1 \cdots y_s m$ then $mn^{-1} =
y_s^{-1}\cdots y_1^{-1}x_1 \cdots x_r$, where the right-hand side lies
in $K_\bO$, so $n$ equals $m$ in the quotient. Also note that changing $g$
by a different coset representative will not change this image. Likewise the map
is a functor since if we have a composite $G/P \xrightarrow{[g]} G/P^g
\xrightarrow{[h]} G/P^{gh}$ with  $g = x_1 \cdots x_r n$ and $h = y_1
\cdots y_s m$, then $gh = x_1 \cdots x_r n y_1 n^{-1}ny_2\cdots y_{s-1}n^{-1}ny_sn^{-1}nm$ and hence is sent to $nm$ as wanted.
By the universal property of the fundamental group (see \S\ref{pi1-subsec}), we hence get an
induced group homomorphism $\psi\co\pi_1(\Oep(G)) \to N_G(S)/K_\bO$, which is
clearly a left and a right inverse to $\phi$, as both maps
commute with the surjection from $N_G(S)$.
\end{proof}
\begin{rem} Let us briefly discuss the assumptions in
  Theorem~\ref{grouptheory-pi1}. First note that the relationship between the
subgroups $Q_i$ is an important part of the statement, i.e., we cannot
just consider the bigger subgroup $N_G(S) \cap \langle O^{p'}(N_G(P)) |
1< P\leq S\rangle$ as examples such as  $\sym_7$ at $p=3$ show. 
Second, one could ask if the
  elements $x_i$ could be assumed
  to lie in $N_G(S)$, i.e., if the subgroups of the `in particular'
  generate the kernel. This often holds but fails e.g., for $G_2(5)$ at $p=3$, as we shall analyze in connection with the
Carlson--Th\'evenaz conjecture, Theorem~\ref{carlsonthevenazconj} (see
Proposition~\ref{G_25p3} and its proof).
Finally we remark that one cannot just assume that all
$Q_i$ are $p$--essential, i.e., $N_G(Q_i)$ cannot be replaced by
$O^{p'}(N_G(Q_i))$ in the Goldschmidt--Miyamoto--Puig
\cite[Cor.~1]{miyamoto77} version of the fusion theorem, as $G_2(5)$ at
$p=3$ is again a counterexample by
Example~\ref{G25-counterex}. Which subgroups are needed is analyzed in
detail in Appendix~\ref{propagating-sec} in terms
of homotopy properties of the collection $\calC$. 
\end{rem}

\begin{rem}\label{transportersystem-rem}
The subcategories of $\Tep(G)$ and $\Oep(G)$ obtained as
preimages of subgroups of $\pi_1(\Oep(G))$ can be thought of as
subcategories ``of $p'$--index'' analogous to the results in
\cite[\S5]{bcglo2} on fusion system and linking systems, and the above
proof may be compared to \cite[Pf.~of~Prop.~5.2]{bcglo2}. Furthermore $\Tep(G)$ is an example of an abstract
  transporter system as defined in \cite[Def.~3.1]{OV07}, so in light of
  \eqref{transportfundgrp}, it would be interesting
  to further understand the group theoretic significance of these subcategories.
\end{rem}

Let us describe the relationship between $\pi_1(\Oep(G))$
and $G_{p'} = G/O^{p'}(G)$ in simple cases.
\begin{cor}\label{sc-fundgrp}
 If $|\calS_p(G)|$ is simply connected
then $$\pi_1(\Tep(G)) \xrightcong G \mbox{\,\,\,  and \,\, }
\pi_1(\Oep(G))
\xrightcong G_{p'}.$$
If we only assume that $|\calS_p(G_0)|$ is simply connected, then $\pi_1(\Tep(G)) \xrightcong G_0$ and
$\pi_1(\Oep(G))
\xrightcong (G_0)_{p'}$.
\end{cor}
\begin{proof}
By \eqref{fundseq} we have an exact sequence
$1 \to \pi_1(\calS_p(G_0)) \to \pi_1(\Tep(G)) \to G_0 \to 1$.
This gives the statements about $\Tep(G)$ and $\Tep(G_0)$ and
the statements about $\pi_1(\Oep(G))$ and $\pi_1(\Oep(G_0))$ now
follows using Propositions~\ref{boundslemma} and \ref{pprime}. 
\end{proof}

The next corollary recovers and extends ``classical'' facts about
$\TS$, by Carlson, Mazza, Nakano, and Th\'evenaz, via
Theorem~\ref{main}. (See
Remark~\ref{stdprop-hist} below for some historical references.)
\begin{cor}[Basic calculations of $\TS$ via $\pi_1(\Oep(G))$] \label{standardprop} $\mbox{ }$
\begin{enumerate}
\item \label{nonpreduced} If for all $g \in G$,  $S\cap S^g \neq 1$
  (e.g., if $O_p(G) \neq 1$) then $\pi_1(\Oep(G)) \xrightcong
  G_{p'}$ and hence $\TS \cong \Hom(G,k^\times)$.

\item \label{cyclic} If $S$ has $p$--rank one, then  
 $\pi_1(\Oep(G)) \xrightcong
  (G_0)_{p'}$, with $G_0 = N_G(\Omega_p(Z(S)))$, and hence $\TS \cong
  \Hom(N_G(\Omega_p(Z(S))),k^\times)$, with $\Omega_p(Z(S))$ 
  the elements of order at most $p$ in $Z(S)$.

\item \label{ti} If $G$ is a trivial intersection (T.I.) group, i.e.,
  a group where unequal Sylow $p$-subgroups intersect trivially, then
  $G_0 = N_G(S)$,  
 $\pi_1(\Oep(G)) \cong
  N_G(S)/S$, and hence
$\TS \cong
  \Hom(N_G(S),k^\times)$.

\item\label{products} If  $G \trianglelefteq H \circ_C K$ is a normal
  subgroup of a central product, with
  $p | |H\cap G|$ and $p | |K\cap G|$, then $\pi_1(\Oep(G)) \xrightcong G_{p'}$ and
  $\TS \cong \Hom(G,k^\times)$.

\item\label{wreath} If $ G = H \wr \sym_n$ is a wreath product, with 
 $p | |H|$ and $n \geq 2$, then $\pi_1(\Oep(G)) \xrightcong G_{p'}$ and 
  $\TS \cong \Hom(G,k^\times)$.
\end{enumerate}
\end{cor}
\begin{proof}
\eqref{nonpreduced}: Suppose $G/P \xrightarrow{[g]} G/Q$ goes to $1 \in
G_{p'}$, so $g$ can be written as a product in $G$ of elements of
$p$-power order. We can thus without loss of generality assume that
$g$ is itself of $p$--power order. Furthermore, by changing $g$ up to
conjugacy we can assume that $Q \leq S$, and it is enough to prove
that  $G/S \xrightarrow{[g]} G/S^{g}$ is the identity in
$\pi_1(\Oep(G))$. However here it factors as $G/S \leftarrow G/(S\cap
S^{g}) \xrightarrow{[g]}  G/(S\cap
S^{g}) \to G/S^{g}$,which is the identity in $\pi_1(\Oep(G))$, as $g$
has $p$--power order in $N_G(S\cap S^{g})/S\cap S^g$.

\eqref{cyclic}: Since $|\cA_p(G)| = G/N_G(\Omega_p(Z(S)))$, a discrete
$G$--space we have $G_0  = N_G(\Omega_p(Z(S)))$. In particular $O_p(G_0)
\neq 1$ and the claim follows from \eqref{nonpreduced}, using Proposition~\ref{boundslemma}.

\eqref{ti}: Since $\calS_p(G)$ is $G$--homotopy equivalent to the
collection of non-trivial Sylow-intersections, e.g., by
\cite[Thm.~1.1]{GS06}, $|\calS_p(G)| \simeq G/N_G(S)$ and $G_0 =
N_G(S)$ and the claim again follows from \eqref{nonpreduced}.

\eqref{products}:
Let $S$ be a Sylow $p$--subgroup of $G$, and note that $Z(S)\cap H
\neq 1$, by a basic property
of $p$--groups, and similarly for $K$. We claim that $O^{p'}(G)\cap
N_G(S)$ is generated by $O^{p'}(N_G(Z(S)\cap K))\cap N_G(S)$ and
$O^{p'}(N_G(Z(S) \cap H)) \cap N_G(S)$, which will finish the proof
by the `in particular' part of Theorem~\ref{grouptheory-pi1}.
To see this, set $\tilde H = H \cap G$, $\hat H = O^{p'}(\tilde H S)
\cap \tilde H$, and define $\tilde K$ and $\hat K$ symmetrically. Note
that for $h \in H$, $s \in S$, we have $hsh^{-1}s^{-1} \in \tilde H$ so
$hsh^{-1} \in \tilde H S$ and hence $hsh^{-1} \in O^{p'}(\tilde H S)$
as $s$ is of $p$--power order.
Thus $hsh^{-1}s^{-1} \in \hat H$ and ${}^hS \leq \hat H S$.
In particular ${}^h(\tilde H S) \leq \tilde H({}^hS) \leq
\tilde H S$ and hence  ${}^h(O^{p'}(\tilde H S)) \leq O^{p'}(\tilde H
S)$,  so $\hat H
\trianglelefteq H$.
As the similar statements hold for $\hat K$ we conclude that $\hat H \hat K
\trianglelefteq HK$ and $\hat H \hat K S \trianglelefteq HK$.
In particular $\hat H \hat K S = O^{p'}(G)$ as it, by the above, is a normal subgroup
generated by elements of $p$--power order, and of $p'$ index. The
factorization  $O^{p'}(G) = \hat H \hat K S$ also
holds in $N_G(S)$: If $hk \in \hat H \hat K \cap
N_G(S)$, then ${}^kS = S^h \leq \hat H S \cap \hat K S$. But
$\hat H S \cap \hat K S = (\hat H \cap \hat K)S \leq CS$, since if $x
\in \hat H S \cap \hat K S$ of order prime to $p$ then $x \in \hat H
\cap \hat K \leq C$, as $\hat H$ and $\hat K$ are normal. Hence $h,k
\in N_G(S)$, so $$h \in \hat H \cap N_G(S) \leq O^{p'}(N_G(Z(S)\cap
K))\cap N_G(S),$$ and symmetrically for $k$ as wanted.

\eqref{wreath}: Suppose $G = H \wr \sym_n$. Let $H_i$ denote the $i$th copy of $H$
and $S_i$ its Sylow $p$--subgroup, and let $\Delta$ denote Sylow $p$--subgroup of $H$
embedded diagonally. Since $\sym_n \leq C_G(\Delta)$, $O^{p'}(\sym_n)
\leq O^{p'}(N_G(\Delta))$. Likewise since $H_i \leq C_G(H_j)$ for $i
\neq j$, $O^{p'}(H_i) \leq O^{p'}(N_G(S_j))$. But

\begin{eqnarray*} N_G(S) \cap
O^{p'}(G)  &=&  N_G(S) \cap (\prod_i O^{p'}(H_i) \semi O^{p'}(\sym_n))
  \\
&=& (\prod_i N_G(S)\cap
O^{p'}(H_i)) \semi (N_G(S) \cap O^{p'}(\sym_n)).
\end{eqnarray*}
So the result again follows from the `in particular' in
Theorem~\ref{grouptheory-pi1}, as $N_G(S) \cap
O^{p'}(G)$ is generated by groups that are trivial in $\pi_1(\Oep(G))$
by the above rewriting.
\end{proof}

\begin{rem} \label{stdprop-hist}
  The statement about $\TS$ in \eqref{nonpreduced} above is
  \cite[Lem.~2.6]{MT07}. (Note also
  that if $O_p(G) \neq 1$,  $|\calS_p(G)|$ is contractible by
\cite[Prop.~2.4]{quillen78}, and compare also to Proposition~\ref{onedim}.)
For the statement about  $\TS$ in \eqref{cyclic} see \cite[Lem.~3.5]{MT07}
(noting that the proof there works also for $S$ of
  $p$--rank one, not just cyclic) and
compare also \cite[\S6]{CT15}. For \eqref{ti} see \cite[Prop.~2.8 and
Rem.~2.9]{CMN06} and also
\cite[\S3.3]{LM15sporadic}. The statement about $\TS$ in \eqref{products}
is a slight strengthening of the recent \cite[Thm.~2.4]{CMN16}. Note furthermore that
$|\calS_p(H \times K)|$ is simply connected unless $H$ and $K$ both
  contain strongly $p$--embedded subgroups by \cite[Prop.~2.6]{quillen78} and Lemma~\ref{connectivitylemma}.
  \end{rem}

\begin{cor}[Subgroups of $p'$--index] \label{pprimesubgroups} If $H \triangleleft G$ is of $p'$ index,
  there is a diagram with exact rows
  \begin{equation}\label{pprime-diagram}
\vcenter{\xymatrix{ 1 \ar[r] & \pi_1(\Oep(H)) \ar[r] \ar@{->>}[d] & \pi_1(\Oep(G))
  \ar[r] \ar@{->>}[d] & G/H \ar[r] \ar@{=}[d] & 1 \\
1 \ar[r] &  H_{p'} \ar[r]  & G_{p'}
  \ar[r]  & G/H \ar[r] & 1 }}
\end{equation}
In particular $\pi_1(\Oep(H))
  \xrightcong H_{p'}$ if and only if  $\pi_1(\Oep(G))
  \xrightcong G_{p'}$, an isomorphism $H_1(\Oep(H)) \xrightcong
  H_1(H)_{p'}$ implies $H_1(\Oep(G)) \xrightcong
  H_1(G)_{p'}$, and $T_k(H,S) \cong \Hom(H,k^\times)$ implies
  $\TS \cong \Hom(G,k^\times)$.
\end{cor}

\begin{proof}
Note that $|\calS_p(H)| = |\calS_p(G)|$.
Hence taking Borel constructions and using \eqref{grothendieck}
produces the following diagram, where the rows are fibration sequences
\begin{equation}\label{pprime-fibration}
\vcenter{\xymatrix{ |\Tep(H)| \ar[r] \ar[d] &
 |\Tep(G)| \ar[r] \ar[d] & B(G/H) \ar@{=}[d] \\
BH \ar[r]& BG \ar[r] & B(G/H)}}
\end{equation}
Diagram \eqref{pprime-diagram} is the bottom of the associated long exact
sequence of homotopy groups, after dividing out elements of finite $p$-power order.
Diagram \eqref{pprime-diagram} implies the consequence about
$\pi_1$ by the $5$-lemma. The homology
consequence also follows from the $5$-lemma, now applied to the
associated ladder of exact sequences coming from \eqref{pprime-diagram}
\begin{equation*}
  \xymatrix{
        H_2(G/H) \ar[r] \ar@{=}[d]& H_1(\Oep(H))_{G/H} \ar[r] \ar@{->>}[d] & H_1(\Oep(G))
  \ar[r] \ar@{->>}[d] & H_1(G/H) \ar[r] \ar@{=}[d] & 0 \\
    H_2(G/H) \ar[r] &  (H_1(H)_{p'})_{{G/H}} \ar[r]  & H_1(G)_{p'}
  \ar[r]  & H_1(G/H) \ar[r] & 0 }
\end{equation*}
Finally the consequence about $\TS$ follows from the dual sequence in cohomology
with $k^\times$--coefficients, and using Theorem~\ref{main}.
\end{proof}
\begin{rem}

The converse to the last two implications in Corollary~\ref{pprimesubgroups}  fail for
  $G = \SL_2(\F_8)\semi C_3$, with $C_3$ acting via the Frobenius, 
  $H= \SL_2(\F_8)$, and $p=2$, where $\pi_1(\bO^*_2(G)) = C_7 \semi C_3$ and
  $\pi_1(\bO^*_2(H)) = C_7$ (see also Remark~\ref{endoessential-remark}).
  Corollary~\ref{pprimesubgroups} says that a full calculation of
  $\pi_1(\Oep(G))$ not only allows the determination of the
  Sylow-trivial modules for $G$ but also those of its normal subgroups
  of $p'$ index, which is not possible from just knowing $H_1$. We illustrate this in \S\ref{symmetricgroups} with
  the symmetric and alternating groups, in fact correcting a small
  mistake in the literature.
\end{rem}

\begin{cor}[Central $p'$--extensions]\label{pprimecentral}
Suppose $Z$ is a central $p'$--subgroup of $G$. Then there is
a diagram of spaces, with horizontal maps fibration
sequences
$$
\xymatrix{ BZ \ar[r] \ar[d] & |\Tep(G)| \ar[r] \ar[d] & |\Tep(G/Z)| \ar[d] \\
  BZ \ar[r] & BG_0 \ar[r] & B(G_0/Z)}$$ 
and hence a ladder of exact sequences
$$\xymatrix@C-6pt{
H_2(\Oep(G)) \ar[r] \ar[d] &H_2(\Oep(G/Z)) \ar[r]^(.7){\partial} \ar[d] & Z \ar@{=}[d] \ar[r] & H_1(\Oep(G)) \ar@{->>}[d]^\phi
\ar[r] & H_1(\Oep(G/Z)) \ar[r] \ar@{->>}[d]^{\bar \phi} & 0 \\
H_2(G_0)_{p'} \ar[r] & H_2(G_0/Z)_{p'} \ar[r]^(.7){\partial'} & Z \ar[r] & H_1(G_0)_{p'} \ar[r] & H_1(G_0/Z)_{p'}
\ar[r] & 0}$$
In particular $ 0 \to \im(\partial')/\im(\partial) \to
\ker(\phi) \to \ker(\bar \phi) \to 0$ is exact. And if $H_2(\Oep(G/Z)) \to H_2(G/Z)_{p'}$ is surjective,
e.g., if $\oplus_{[P] \in \calS_p(G)/G}H_2(N_{G/Z}(P))_{p'} \twoheadrightarrow H_2(G/Z)_{p'}$,
then $\ker(\phi) \xrightcong \ker(\bar \phi)$, i.e., $G$ and
$G/Z$ have isomorphic groups of ``truly exotic'' Sylow-trivial
modules, as defined after Theorem~\ref{sequence}.
\end{cor}
\begin{proof}
Since $Z$ is central, we have a principal fibration sequence
$$BZ \to (|\calS_p(G)| \times EG)/G \to (|\calS_p(G)| \times
E(G/Z))/G$$
with the canonical action of $BZ$ on  $(|\calS_p(G)| \times EG)/G$
 (see e.g., \cite[V.3]{GJ99}).
This fibration sequence 
identifies with the top fibration
sequence of the corollary using
\eqref{grothendieck} and the fact that the projection map $\calS_p(G)
\xrightcong \calS_p(G/Z)$ is a bijection. The map to the standard bottom fibration
sequence follows as $\Tep(G)$ is equivalent to $\Tep(G_0)$ by
Proposition~\ref{boundslemma} and $(G/Z)_0 = G_0/Z$. The ladder of exact sequences is now
the five-term exact sequence in homology of a fibration with
coefficients in $\Z[\frac1p]$, using
that we can replace $\Tep$ by $\Oep$ when considering
$p'$--coefficients by Proposition~\ref{pprime-higher} and
Theorem~\ref{homotopy-finiteness}. The stated consequences follow from
the snake lemma.
\end{proof}
\begin{rem} Note that the kernel of $\phi$ is described via
 Proposition~\ref{prop:homologyseq}, and this can be used for an alternative
  derivation of the last part of Corollary~\ref{pprimecentral}.
\end{rem}

\begin{example} Let us illustrate Corollary~\ref{pprimecentral} by the
  symmetric groups: Recall that by homological stability, $\F_2 \cong H_2(\sym_4)
  \xrightcong H_2(\sym_{n+4})$  for any $n\geq 0$  \cite[Thm.~2]{kerz05}, so
  $H_2(N_{\sym_n}((1\cdots p))) \twoheadrightarrow H_2(\sym_n)$ when
  $n \geq p+4$. Hence $T_k(2^{\pm}\sym_n,S) \xrightcong
  T_k(\sym_n,S)$ for $n \geq p+4$ and $p$ odd, by
  Corollary~\ref{pprimecentral}, where $2^{\pm}\sym_n$ denotes
  the two double covers of $\sym_n$, following $\atlas$ \cite{CCNPW85} notation.
  This was first proved in
  \cite[Thm.~B(2)]{LM15schur} (under an algebraically closed assumption), by lifting to characteristic zero and
  examining  the list of possible characters.
\end{example}

\begin{rem} In  \cite[Thm.~1.1]{LT17} a reduction of the problem describing Sylow-trivial modules
  for arbitrary $p'$--extensions to that
  of central $p'$--extensions, at least as far as obtaining an upper
  bound on Sylow-trivial modules of the extension. The proof relies,
  through reference to earlier results, on the classification of finite simple
  groups. It would be interesting to rework and extend this result
  in light of the methods of the present paper.
\end{rem}
\subsection{Fundamental groups of fusion categories: proof of Theorem~\ref{fusionthm}} \label{fund-fusion-subsec}
We now analyze 
$\pi_1(\calF_p^*(G))$ and related categories,  and use this to prove
Theorem~\ref{fusionthm}, and set the stage for later
calculations involving the centralizer decomposition,
Theorem~\ref{centralizerthm}.
The starting point 
is bounds on $\pi_1(\calF_p^*(G))$ analogously to
\eqref{naivebounds2} for
$\pi_1(\Oep(G))$:
\begin{equation}
N_G(S)/C^{p'}(N_G(S)) \twoheadrightarrow \pi_1(\calF_p^*(G)) \twoheadrightarrow
G_0/C^{p'}(G_0) \twoheadrightarrow G/C^{p'}(G)
\end{equation}
where
\begin{equation} C^{p'}(G) = \langle x \in G | \, p \mbox{ divides }
  |C_G(x)|\rangle,
\end{equation}
the group generated by ``positive defect''
elements. This is a special case of Proposition~\ref{fusion-boundslemma} below.

Recall that a $p$--subgroup $P$
is called  {\em $p$--centric} if $Z(P)$ is a Sylow $p$--subgroup of
$C_G(P)$ (hence $C_G(P) \cong Z(P) \times O_{p'}(C_G(P))$ by
Schur-Zassenhaus). It is called {\em $\calF$--essential} if it is not Sylow and
$W_0PC_G(P)/P$ is a proper subgroup of $W =N_G(P)/P$, or equivalently if it is $p$--centric and
$\calS_p(N_G(P)/PC_G(P))$ is disconnected (see \S\ref{modelsF} and Lemma~\ref{lem:Fessentialdef}). Continuing the notation from \eqref{pembeddedC}, the analog of
Proposition~\ref{boundslemma} states:

\begin{prop}[Bounds on $\pi_1(\calF_\calC(G))$]\label{fusion-boundslemma}
Suppose that $\calC$ is a collection of $p$--subgroups closed under
passage to $\calF$--essential and Sylow overgroups. Then
$\calF_\calC(G)$ is equivalent to $\calF_{\calC_0}(G_{0,\calC})$ and
with $\bar C^{p'}_\calC(H) = \langle
PC_H(P)| P \leq H, P \in \calC\rangle$, for $S \leq H \leq G$, we have canonical surjections
\begin{multline*}  N_G(S)/SC_G(S) \twoheadrightarrow N_G(S)/\bar C^{p'}_\calC(N_G(S)) \xrightcong
  \pi_1(\bar \calF_\calC(N_G(S))) \\ \twoheadrightarrow \pi_1(\bar \calF_\calC(G))
\twoheadrightarrow G_{0,\calC}/\bar C^{p'}_\calC(G_{0,\calC})
\end{multline*}
If $\calC$ contains all minimal $p$--centric
subgroups, then $\pi_1(\calF_\calC(G))\xrightcong
\pi_1(\bcalF_\calC(G))$, and thus  $\pi_1(\calF_\calC(G))$ is a finite $p'$--group. 
\end{prop}

\begin{proof}
First note that $ \calF_{\calC_0}(G_{0,\calC}) \to
\calF_\calC(G)$,  and hence also $\bar \calF_{\calC_0}(G_{0,\calC}) \to \bar
\calF_\calC(G)$, are equivalences of categories by Alperin's fusion
theorem, either in the version of  Goldschmidt--Miyamoto--Puig
\cite[Cor.~1]{miyamoto77}, or more generally for
fusion systems \cite[Thm.~I.3.5]{AKO11} (see again \cite[\S10]{grodal02} for a
discussion of versions of the fusion theorem). For the rest of the proof we can hence
without loss of generality assume that  $G = G_{0,\calC}$ so the
claimed sequence becomes
\begin{multline*}  N_G(S)/SC_G(S) \twoheadrightarrow N_G(S)/\bar C^{p'}_\calC(N_G(S)) \xrightcong
  \pi_1(\bar \calF_\calC(N_G(S))) \\ \twoheadrightarrow \pi_1(\bar \calF_\calC(G))
\twoheadrightarrow G/\bar C^{p'}_\calC(G).
\end{multline*}
We model $\pi_1(\bar \calF_\calC(G))$ analogously to the proof of
Proposition~\ref{boundslemma} with generators maps in $\bar \calF_\calC(G)$
between subgroups $P,Q$ of $S$ with $P,Q \in \calC$, via the functor
$\bar \calF_\calC(G) \to \pi_1(\bar \calF_\calC(G))$, taking $S$ as
basepoint, cf.\ again \S\ref{pi1-subsec}. In this notation the right-most epimorphism
$\pi_1(\calF_\calC(G)) \twoheadrightarrow
G/\bar C^{p'}_\calC(G)$ is the map sending $(c_g\co P \to Q)$ to $[g] \in
G/\bar C^{p'}_\calC(G)$. 
Since  $\pi_1(\bar \calF_\calC(G))$ is a
quotient of $\pi_1(\bO_\calC(G))$, we have a surjection $N_G(S) \twoheadrightarrow 
\pi_1(\bar \calF_\calC(G))$ by  
Proposition~\ref{boundslemma}, and by the same argument  $N_G(S) \twoheadrightarrow 
\pi_1(\bar \calF_\calC(N_G(S)))$. We have hence shown that the map
from $N_G(S)$ to all the terms in the sequence is
surjective, so all maps in the sequence are surjections as claimed. Likewise $SC_G(S) \leq
C^{p'}_\calC(N_G(S)) = \langle
PC_{N_G(S)}(P)| P \leq N_G(S), P \in \calC\rangle$, as $S \in \calC$, so the first map is
well-defined.

To finish showing that we have the stated sequence of maps we therefore
just have to show that $\bar C^{p'}_\calC(N_G(S))$ is exactly the kernel of
$N_G(S) \to \pi_1(\bar \calF_\calC(N_G(S)))$. It is in the kernel
since if $g \in
PC_{N_G(S)}(P)$ then we have a commutative diagram
$$\xymatrix{ P \ar[d]^{id} \ar[r] & S \ar[d]^{c_g}\\
 P  \ar[r] & S}$$
in $\bar \calF_\calC(G)$
and hence $c_g\co S \to S$ represents the identity in 
$\pi_1(\calF_p^*(N_G(S)))$. However, then the kernel has to be
exactly $\bar C^{p'}_\calC(N_G(S))$, since taking $G = N_G(S)$, the second and fifth term of the sequence of the proposition agree and the map is the
identity.

To finish the proof assume now that $\calC$ contains all minimal
$p$--centric subgroups. We
will show that $\pi_1(\calF_\calC(G))\xrightcong
\pi_1(\bcalF_\calC(G))$. This in fact follows easily from results about fusion theorems (see
Remark~\ref{bcglo2-remark}), but let us include a stand-alone argument for completeness:
By Proposition~\ref{remove-pione-F} we may
assume that $\calC$ is closed under passage to all $p$--overgroups,
and in particular
contains all $p$--centric subgroups. 
By the fusion theorem,  again, $\pi_1(\calF_\calC(G))$ is
generated by self-maps of elements $P \in \calC$, $P \leq S$, so we just need to see that all elements of $p$--power order in
$\Aut_\calF(P) \cong N_G(P)/C_G(P)$
are trivial in $\pi_1(\calF_\calC(G))$.
We can without loss of generality assume that $N_S(P)$ is a Sylow
$p$--subgroup of $N_G(P)$ (i.e., that $P$ is ``fully $G$--normalized'' in $S$),
$G$--conjugating $P$ if necessary. By
conjugation in $\Aut_\calF(P)$ it is furthermore enough to prove that all elements
in the image of $N_S(P)$ in $\Aut_\calF(P)$ are trivial in
$\pi_1(\calF_\calC(G))$. In $\pi_1(\calF_\calC(G))$,  such elements
are equal to elements of $\Inn(S)
\leq \Aut_\calF(S)$,  since
inclusions are trivial in $\pi_1(\calF_\calC(G))$. The claim is hence reduced to seeing that
$\Inn(S) \leq \Aut_\calF(S)$ map to the identity in $\pi_1(\calF_\calC(G))$. 
Now, any element $x \in S$ is $G$--conjugate to an element $x' \in S$ such that
$C_S(x')$ is a Sylow $p$--subgroup of $C_G(x')$ (i.e., $x'$ is ``fully
$G$--centralized'' in $S$). Note that $Q =C_S(x')$ is $p$--centric in $G$, since
$C_G(Q) \leq C_G(x')$, and $Q$ is obviously $p$--centric in $C_G(x')$,
it being a Sylow $p$--subgroup. The image of $x' \in S$ in $\Aut_\calF(S)$
identifies with the image of $x'$ in $\Aut_\calF(Q)$, as elements of
$\pi_1(\calF_\calC(G))$, again since inclusions map to the identity.  But $x'$ goes to the
identity in $\Aut_\calF(Q)$, so $x' \in S$ represents the
identity in $\pi_1(\calF_\calC(G))$. Hence so does the $G$--conjugate
element $x$,
as wanted.
\end{proof}

Let us spell out what Proposition~\ref{fusion-boundslemma} says for
$\calC =\calS_p(G)$, in classical
group-theoretic terms, as used in the proof of Theorem~\ref{centralizerthm} in the next section.

\begin{cor}[{A vanishing condition for
    $\pi_1(\calF^*_p(G))$}] \label{vanishinglemma}
Let $L$ be a complement to $S$ in $N_G(S)$, and
set $L_0 = \langle C_{L}(x) | x\in S \setminus
1\rangle \unlhd L$. Then 
$$L/L_0\xrightcong \pi_1(\calF_p^*(N_G(S)))
\twoheadrightarrow \pi_1(\calF_p^*(G))$$
In particular if $L$ is generated by elements
that commute with at least one non-trivial element in $S$, then
$\pi_1(\calF_p^*(G)) = 1$. If $H_1(L)$ is generated by such elements
then $H_1(\calF_p^*(G))=0$. \QED
\end{cor}

We also get the following corollary, analogous to Corollary~\ref{sc-fundgrp}.
\begin{cor} \label{sc-fusion-fundgrp} If $|\calS_p(G)|$ is simply connected, e.g., if $O_p(G)
  \neq 1$, then $$\pi_1(\calF_p^*(G))\xrightcong G/C^{p'}(G).$$ If just
  $|\calS_p(G_0)|$ is simply connected then  $\pi_1(\calF_p^*(G))\xrightcong G_0/C^{p'}(G_0)$.
\end{cor}
\begin{proof}
The second case reduces to the first since $\calF^*_p(G_0) \to
\calF^*_p(G)$ is an equivalence of categories by
Proposition~\ref{fusion-boundslemma}, also using \eqref{redtoG_0}.
Now, consider the diagram obtained also using \eqref{fundseq}
$$\xymatrix{\pi_1(\Tep(G)) \ar[r]^{\simeq} \ar@{->>}[d]& G \ar@{->>}[d]\\
\pi_1(\calF_p^*(G)) \ar@{->>}[r] & G/C^{p'}(G)}$$
We just need to
see that any element $x \in G$ which lies in $C^{p'}(G)$ goes to
zero in $\pi_1(\calF_p^*(G))$ under the composite given by the top isomorphism and the left-hand epimorphism. Assume that $x$
is a generator of $C^{p'}(G)$, i.e., we can find a
non-trivial $p$-subgroup $P$ such that $x \in C_G(P)$. Lift $x$
to $x\co P \to P$ in $\pi_1(\Tep(G))$, which maps to the identity
in $\pi_1(\calF_p^*(G))$ as wanted.
\end{proof}

\begin{rem} Note that Corollary~\ref{sc-fusion-fundgrp} can also
  be applied to subgroups $H$ of our group $G$. And if 
$N_G(S) \leq H \leq G$ then
$\pi_1(\calF_p(H)) \twoheadrightarrow \pi_1(\calF_p^*(G))$ by naturality, as in 
Remark~\ref{restriction-to-subgroups}. 
\end{rem}

We also state the following corollary, suggested to us by Ellen Henke.
\begin{cor} If $N_G(S)/SC_G(S)$ is abelian, then
  $\pi_1(\calF^*_p(G))$ is cyclic.
\end{cor}
\begin{proof}

  Set $H= N_G(S)/SC_G(S)$ and
choose an irreducible $H$--submodule $V$ of
$\Omega_p(Z(S))$, the elements of order at most $p$ in $Z(S)$. Then
$H/C_H(V)$ is cyclic by elementary representation theory \cite[Thm.~3.2.3]{gorenstein68},
which implies the claim by Corollary~\ref{vanishinglemma}.
\end{proof}

\begin{rem} As hinted above, the quotient $G/C^{p'}(G)$  is often trivial
  or a rather small group, and bounding its size is an interesting and
  so far apparently unexplored group theoretical
  question about $p'$--actions. We study this in joint work in
  progress with Geoffrey Robinson.
\end{rem}

Analogous to Theorem~\ref{grouptheory-pi1} we can also describe 
$\pi_1(\calF_p^*(G))$ purely
group theoretically. 

\begin{thrm}[Group theoretic description of
  $\pi_1(\calF_p^*(G))$] \label{grouptheory-pi1F} Let $K_\calF$ be the subgroup of $N_G(S)$
  generated by elements $g \in N_G(S)$ such that there exist
  nontrivial subgroups $1<Q_0, \ldots, Q_r \leq S$  and a factorization  $g = x_1
 \cdots x_r$ in $G$, where $x_i \in C^{p'}(N_G(Q_i))$ 
and $Q_0^{x_1\cdots x_i} \leq Q_{i+1}$ for $i\geq 0$.
Then 
$$N_G(S)/K_\calF\xrightcong \pi_1(\calF_p^*(G))$$
In particular $N_G(S)\cap
  C^{p'}(N_G(P)) \leq \ker\left( N_G(S)
\twoheadrightarrow \pi_1(\calF_p^*(G))\right)$ for $1 < P \leq S$.
\end{thrm}
\begin{proof} Recall that in the proof of
  Theorem~\ref{grouptheory-pi1} we showed the isomorphism $$\phi \co
  N_G(S)/K_\bO \xrightcongdbl \pi_1(\Oep(G)),$$ by constructing a left
  inverse $\psi$.  And
  by
  Proposition~\ref{fusion-boundslemma} we have a surjection $\pi_1(\Oep(G))
  \twoheadrightarrow \pi_1(\calF_p^*(G))$. As $K_\bO \leq K_\calF$  by definition we can thus establish the theorem by verifying that $\phi$ and
  $\psi$ induce well-defined maps $\bar \phi$ and $\bar \psi$ on the quotients
  $$\xymatrix{ N_G(S)/K_\bO \ar@{->>}[d] \ar@{->>}[r]^\phi_\cong &  \pi_1(\Oep(G))
    \ar@{->>}[d]  \ar@/^1pc/[l]^\psi\\
    N_G(S)/K_\calF \ar@{.>>}[r]^{\bar \phi} & \pi_1(\calF_p^*(G))
    \ar@{.>}@/^1pc/[l]^{\bar \psi}} $$
  which will then necessarily be inverse equivalences. This amounts to
  unravelling the definitions:
  
  To see that $\bar \phi$ is well-defined, we need to check that
  $K_\calF$ lies in the kernel of $N_G(S) \to \pi_1(\calF_p^*(G))$.
  Given a generator $g$ of $K_\calF$ as in the theorem, we have the
  following commutative diagram
$$\xymatrix{ S \ar[rrrr]^{c_{g^{-1}}} & & & &  S\\
  Q_0 \ar[u] \ar[r]^{c_{x_1^{-1}}} & Q_1\ar[r] & \cdots \ar[r] &
  Q_{r-1} \ar[r]^{c_{x_r^{-1}}}&Q_r\ar[u]}$$
in $\calF_p^*(G)$, and $c_{x_i^{-1}} \co Q_{i-1}
\to Q_i$ is equivalent to
$c_{x_i^{-1}} \co Q_i\to Q_i$ in  $\pi_1(\calF_p^*(G))$.
The kernel of $N_G(Q_i) \to
  \pi_1(\calF_p^*(N_G(Q_i)))$ is $C^{p'}(N_G(Q_i))$ by
  Corollary~\ref{sc-fusion-fundgrp}, so the factorization  $N_G(Q_i)/C_G(Q_i)  \to \pi_1(\calF_p^*(N_G(Q_i)))  \to
  \pi_1(\calF_p^*(G))$
shows that
  $c_{x_i^{-1}} \co Q_i\to
  Q_i$, and hence $c_g\co S \to S$, is zero in $\pi_1(\calF_p^*(G))$
  as wanted. This shows that we have a well-defined map $\bar \phi \co
  N_G(S)/K_\calF \twoheadrightarrow \pi_1(\calF_p^*(G))$. It also
  establishes the last part of the theorem.

 To see that $\bar \psi$ is well-defined we have to justify that for
 any $g \in C_G(P)$, the element
  $G/P \xrightarrow{[g]} G/P$ in $\pi_1(\Oep(G))$ is mapped to an element in $K_\calF/K_\bO$ under
  $\psi$. We can without restriction assume that $P \leq S$. As
  explained in the proof of
  Theorem~\ref{grouptheory-pi1}, the map
  $\psi$ is given by sending $[g]$ to $nK_\bO$, where we express
$g = x_1 \cdots x_r n$,
with $x_i \in O^{p'}(N_G(Q_i))$ and $n \in N_G(S)$, where $1<Q_0, \ldots,
Q_r \leq S$,  $P = Q_0$, $Q_0^{x_1\cdots x_i} \leq Q_{i+1}$ for $i\geq 0$, using Alperin's fusion theorem.
But then $n^{-1} = g^{-1}x_1 \cdots x_r$, with $P^{g^{-1}} = P$ and $g^{-1} \in C_G(P) \leq C^{p'}(N_G(P))$, which shows that $n^{-1} \in
K_\calF$ as wanted.
\end{proof}

\begin{proof}[\thmunderline{Proof of Theorem~\ref{fusionthm}}]
Consider the following commutative diagram
\begin{equation}\label{fusionfundgrp-diagram}
\vcenter{\xymatrix{ N_G(S)/S  \ar@{->>}[r]   \ar@{->>}[d]  &
  \pi_1(\bO_p^c(G))   \ar@{->>}[r]   \ar@{->>}[d] &
  \pi_1(\Oep(G))   \ar@{->>}[d] \\
 N_G(S)/C_G(S)S  \ar@{->>}[r]   &
  \pi_1(\bar \calF_p^c(G))   \ar@{->>}[r]  &
  \pi_1(\bar \calF_p^*(G))  \\
 &
  \pi_1( \calF_p^c(G))   \ar@{->>}[r] \ar[u]^\simeq &
  \pi_1( \calF_p^*(G))  \ar[u]^\simeq
}}
\end{equation}
The top horizontal maps in \eqref{fusionfundgrp-diagram} are epimorphisms by Proposition~\ref{boundslemma}, as the collection of $p$--centric
subgroups is closed under passage to $p$--overgroups. The surjections
between the top and middle row follow by definition, and the 
properties of the maps in and between the second and third row follow by Proposition~\ref{fusion-boundslemma}.
Applying $\Hom(-,k^\times)$ and Theorem~\ref{main} now
gives the first part of Theorem~\ref{fusionthm}.

Now suppose that all $p$--radical $p$--centric subgroups are
centric: We claim that then in fact {\em all} $p$--centric subgroups are
centric. Namely suppose that $P$ is $p$--centric so $C_G(P) \cong ZP \times
K$ where $K = O_{p'}(C_G(P))$.  Then $K$ is normalized by
$O_p(N_G(P))$ and vice versa. But as $K$ and $O_p(N_G(P))$ are of
coprime order they then have to commute. If $O_p(N_G(P))$ is not
$p$--radical, we can repeat the process of taking $O_p(N_G(-))$,
until we arrive at a $p$--radical subgroup (the $p$--radical closure,
which will also be used in Section~\ref{homologydec-sec}). Hence $K$
centralizes a $p$-centric $p$-radical subgroup and is hence trivial by
assumption, which is what we wanted. Hence $\bO^{c}_p(G) \cong \bar
\calF_p^c(G)$ by definition, and thus $ \pi_1(\bO^{c}_p(G)) \cong \pi_1(\bar
\calF_p^{c}(G))\xleftcong \pi_1(\calF_p^{c}(G))$, using
Proposition~\ref{boundslemma} again. With this isomorphism, the last
part of the theorem now follows from the first part.
\end{proof}

\begin{rem} Appendix~\ref{propagating-sec}, e.g., Theorem~\ref{remove-non-essential} and
Proposition~\ref{remove-pione-F}, can be used to further analyze
which $p$--centric subgroups need to be centric, for
$\pi_1(\bO^{c}_p(G))$ and $\pi_1(\calF^{c}_p(G))$ to agree.
\end{rem}

\begin{rem} \label{bcglo2-remark}  The $p'$--quotient groups of $\pi_1(\calF^c(G))$, were originally studied
  in \cite[\S5.1]{bcglo2}, where they were related to subsystems of
  the fusion system of $p'$--index. It was remarked to the authors by
  Aschbacher that the group itself is a finite
  $p'$--group, see \cite[p.~3839]{bcglo2},
  \cite[Ch.~11]{aschbacher11} and \cite[Prop.~III.4.19]{AKO11}.
 \end{rem}

Let us round off our discussion of $\pi_1(\calF_{\calC}(G))$ for now by giving a few computational
examples---see also Appendix~\ref{propagating-sec} for more information on
how it depends on $\calC$.

\begin{example}\label{cyclicauto}
  For $\phi$ an automorphism, of order prime to $p$, of a
finite $p$--group $S$, $\pi_1(\calF_p^*(S \semi \langle \phi \rangle)) \cong \Z/r$,
generated by $\phi$,  with $r$ the greatest common divisor of all
natural numbers $s$ such that
$\phi^s$ acts with a fixed-point on $S\setminus 1$, 
by Corollary~\ref{vanishinglemma}.
\end{example}

\begin{prop}\label{extraspecial27}
Suppose that $\calF$ is a fusion system over $S=3^{1+2}_+$. Then
$\pi_1(\calF^*) = 1$ unless $\calF =\calF_3(G)$ for $G = 3^{1+2}_+ :8$, in
which case $\pi_1(\calF^*) \cong  \Z/2$. 
The sporadic group $J_2$ is the unique finite simple group with $3$--fusion
system  $\calF_3(3^{1+2}_+ :8)$, and hence for all other finite simple
groups  $G$ with Sylow $3$--subgroup $3^{1+2}$,  $\pi_1(\calF^*_3(G))=1$.
\end{prop}
\begin{proof} 
We have  that $\Out(S) \cong \GL_2(\F_3)$ of order $48$, with the canonical action on
$S/Z(S) \cong (\F_3)^2$ and action on $Z(S)$ through the
determinant map, described in detail in \cite{winter72}.
As we can lift $S/Z(S)$ to an $\Out(S)$ invariant subset of $S$, an
element of $\Out(S)$ acts freely on $S \setminus 1$ iff it acts freely
on $S/Z(S) \setminus Z(S)$ and on $Z(S) \setminus 1$.
We have that $L = N_\calF(S)/S$ is a subgroup of the Sylow
$2$--subgroup $SD_{16}$, which identifies with the semi-linear
automorphisms of $\F_{9}$, generated by a generator $\sigma$ of
$\F_{9}^\times$ and the Frobenius $\tau$, subject to the relations
$\sigma^8=\tau^2=1$, and $\tau \sigma \tau = \sigma^3$.
Both $\sigma$
and $\tau$ have determinant $-1$ inside $\GL_2(\F_3)$.
We claim that elements of the form $\sigma^{2k+1}$ are the only
elements of $SD_{16}$ that act freely on $S \setminus 1$. Such elements act freely, as they
act freely on
$\F_9^\times $ and on $Z(S) \setminus 1$.
The elements $\sigma^{2k}$ and  $\sigma^{2k-1}\tau$, $k \geq 1$, act with a
fixed-point, since they act trivially on $Z(S)$, and so does $\tau$ as
it fixes $\F_3^\times \subseteq \F_9^\times$. The element
$\sigma^{2k}\tau$ has a fixed-point as well, as it is conjugate to $\tau$ via
$\sigma^k(\sigma^{2k}\tau)\sigma^{-k} = \tau$.
The only two subgroups that contain $\sigma^{2k+1}$ are $\langle \sigma \rangle$ and $SD_{16}$. For
$L = SD_{16}$ we have $\sigma = (\sigma \tau) \tau$, in the notation
of  Corollary~\ref{vanishinglemma}, so $L = L_0$. For $L= \langle \sigma
\rangle$, $L/L_0 \cong \Z/2$ by Example~\ref{cyclicauto}.

Now by \cite[Thm.~1.1]{RV04} all abstract fusion
systems on  $S=3^{1+2}_+$ arise from groups, and by \cite[Tables~1.1
and 1.2]{RV04}  $\calF_3(3^{1+2}_+:8)$ is the only fusion system on
$3^{1+2}_+$ with $N_\calF(S)/S \cong \Z/8$.
Hence Corollary~\ref{vanishinglemma}, combined with the analysis above, shows
that $\calF = \calF_3(3^{1+2}_+:8)$ is the only fusion system with
$\pi_1(\calF^*) \neq 1$, and that $\pi_1(\calF^*) \cong \Z/2$ in that case.
Furthermore $J_2$ is the unique finite simple group
realizing $\calF$ by  \cite[Rem.~1.4]{RV04}. 
\end{proof}

\begin{example}
By \cite{CCNPW85} the centralizers of $3$--elements in $J_2$ are
$C_{J_2}(3A) = 3 \cdot \PSL_2(9)$ and $C_{J_2}(3B) = 3 \times A_4$. In
particular they satisfy $H_1(-)_{3'}=0$. Hence Proposition~\ref{extraspecial27}
combined with Theorem~\ref{centralizerthm} gives $$T_{\F_3}(J_2,S) \cong
\Hom(\pi_1(\calF_3^*(J_2)),\F_3^\times) \cong \Z/2$$ in this case. See
\cite[6.4]{LM15sporadic} for a very different derivation of this result.
\end{example}

\subsection{Higher homology groups} \label{higher-homotopy} We conclude
the section by  briefly considering the higher homology of $|\Tep(G)|$ and
$|\bO_p^*(G)|$.
Higher homology
groups occur naturally, as obstructions
to extending compatible elements, as in Theorem~\ref{centralizerthm}. We already made a forward reference
to
this  subsection in the proof of Corollary~\ref{pprimecentral}, but otherwise the
main results of the paper do not rely in them.

\begin{prop}\label{pprime-higher}
For $\calC$ any collection of $p$--subgroups of $G$,
$$H_*(\calT_\calC(G)) \otimes \Z[\tfrac1p]\xrightcong H_*(
\bO_\calC(G)) \otimes \Z[\tfrac1p]  \mbox{, \,\,\,\,}H_*(\calF_\calC(G))\otimes
\Z[\tfrac1p] \xrightcong H_*(
\bcalF_\calC(G)) \otimes \Z[\tfrac1p], $$ 
$$\mbox{ and \hspace{0.5cm}} H_*(|\calC|/G) \otimes \Q \xleftcong H_*(\calT_\calC(G))\otimes \Q\xrightcong H_*(\calF_\calC(G)) \otimes \Q.$$
\end{prop}

\begin{proof}
The statements follow by a Grothendieck composite functor spectral
sequence argument, since the morphisms differ by finite  $p$--groups
or finite groups. More
precisely,  \cite[Lem.~1.3]{BLO03inv} implies that the two first maps are equivalences in
homology with $\F_\ell$--coefficients for all primes $\ell \neq
p$. Since the spaces are of finite type this implies equivalence in
homology with $\Z_{(\ell)}$--coefficients, for all primes $\ell \neq p$,
and hence an isomorphism in homology with
$\Z[\frac1p]$-coefficients, and after tensoring with
$\Z[\frac1p]$. The third isomorphism holds since the isotropy spectral
sequence, Proposition~\ref{isotropy-ss}, for the $G$--action on
$|\calC|$ converging to  $|\calC|_{hG} \simeq |\calT_\calC(G)|$, collapses, and the fourth isomorphism
again follows from the proof of
\cite[Lem.~1.3]{BLO03inv}, now using that $\tilde H_*(C_G(P);\Q) = 0$.
\end{proof}

\begin{thrm}\label{homotopy-finiteness}
Let $\calC$ be collection of $p$--subgroups, closed under passage to $p$--radical
overgroups. For $i>0$ the groups $H_i(\bO_\calC(G))$,
$H_i(\calT_\calC(G))$, and
$H_i(\calF_\calC(G))$ are 
finite, and $H_i(\bO_\calC(G))$ and $H_i(\calF_p^*(G))$ are of order prime to $p$.
When $\calC$ is ample (i.e., $H_*(\calT_\calC(G))_{(p)} \xrightcong
H_*(G)_{(p)}$, see Remark~\ref{amplerem}),
$$H_i(\calT_\calC(G))\xrightcong
H_i(\bO_\calC(G)) \oplus H_i(G)_{(p)}, \mbox{\,\,\,\,\,} i>0.$$
\end{thrm}

\begin{proof}
First note that all the groups are finitely generated,
since there are finitely many simplices in each dimension, the spaces being
nerves of finite categories. Furthermore, by
Proposition~\ref{webbsconj}, $H_i(|\calC|/G) = 0$ for $i>0$. Now the
finiteness of all the groups follow from
Proposition~\ref{pprime-higher}.

The statements about absence of $p$--torsion are well known consequences of the
 theory of mod $p$ homology
decompositions: To see that  $H_i(\bO_\calC(G))$ is of order prime to $p$
we just have to see that $H^i(\bO_\calC(G);\Z_{(p)}) = 0$ for
$i>0$.  By definition $H^i(\bO_\calC(G);\Z_{(p)}) \cong
\lim^i_{\bO_\calC(G)}\Z_{(p)}$ (see \cite[Prop.~2.6]{grodal02}). But 
\cite[Thm.~1.3]{grodal02} gives a spectral sequence for calculating
$\lim^i$, whose $E^1$--term in this case, by \cite[Cor.~5.4]{grodal02}, is zero except for one $\Z_{(p)}$,
coming from the Sylow $p$--subgroup, responsible for $
\lim^0_{\bO_\calC(G)}\Z_{(p)} \cong
\Z_{(p)}$, showing the claim. Likewise $H_i(\calF_p^*(G);\Z_{(p)}) =0$, $i>0$, by \cite[7.3]{dwyer98sharp} and
\cite[Ex.~8.6]{grodal02} (variants on the classical
\cite[{Prop.~2.1}]{JM92}). 

Finally, if
$H_*(\calT_\calC(G))_{(p)}\mkern-1mu \xrightcong \mkern-1mu H_*(G)_{(p)}$ then
$H_i(\calT_\calC(G)) \mkern-1mu \xrightcong \mkern-1mu
H_i(\bO_\calC(G)) \oplus H_i(G)_{(p)}$ for $i>0$ by Proposition~\ref{pprime-higher}.
\end{proof}

\begin{rem} \label{allpsubgroups-rem} In the degenerate case when $\calC$ is the collection of all
  $p$--subgroups, including the trivial one,  Theorem~\ref{homotopy-finiteness} says that
$H_i(\bO_p(G)) \cong
H_i(G)_{p'}$, $i>0$, as $|\calT_\calC(G)| \simeq BG$.
  \end{rem}

\begin{rem}The homology groups of  $\calT_\calC(G)$, $\bO_\calC(G)$,
  or $\calF_\calC(G)$
   will generally not be finite without assumptions on $\calC$. For example if $G$ is abelian then $\calF_\calC(G) =
  \calC$ and $|\calC|/G = |\calC|$, so examples can be
  constructed using Proposition~\ref{pprime-higher}. Taking  $G = (\Z/p)^r$ and $\calC$ the collection of proper
  non-trivial subgroups of $G$, then $|\calC|$ has homotopy type a wedge of spheres (e.g., $p+1$
  points and a wedge of $p^3$ circles, for $r=2,3$ respectively).
\end{rem}

\begin{cor} For any field $k$ of characteristic $p$,
$$H^i(\Tep(G);k^\times) \xleftcong H^i(\Oep(G);k^\times) \oplus
H^i(G;k^\times)_{(p)} \mbox{ for $i>0$. }$$ If $k$ is
perfect, then $k^\times$ is uniquely $p$--divisible, and $$H^*(\Tep(G);k^\times) \xleftcong H^*(\Oep(G);k^\times).$$
\end{cor}
\begin{proof}
The first claim is a consequence of Theorem~\ref{homotopy-finiteness},
the Universal Coefficient Theorem, and the five-lemma, using
that $H_i(\Tep(G))_{(p)}\xrightcong H_i(G)_{(p)}$ by \eqref{ampleequation}. That $k$ is perfect means that the Frobenius map $(-)^p\co k^\times \to k^\times$ is
not only injective, but also surjective, i.e., $k^\times$ is uniquely
$p$--divisible, and hence $H^i(G;k^\times)_{(p)} =
0$, for $i>0$ by an application of the transfer.
\end{proof}
\begin{rem} Any finite field or any algebraically closed field is of
  course perfect. For any field of characteristic $p$ we have $H^1(G;k^\times)_{(p)} = 0$, as $k^\times$ is
  $p$--torsion-free, but the higher groups are non-trivial in general
  for non-perfect fields. E.g., if
  $k = \F_p(x)$, rational functions in one variable over $\F_p$, the units
 $k^\times$ is isomorphic to $\F_p^\times \times \Z^{({\mathbb N})}$, as an abelian group,
  with a basis for the torsion-free part given by monic irreducible
  polynomials, so $H^i(G;k^\times)_{(p)} \cong
  (H^i(G;\Z)_{(p)})^{({\mathbb N})}$ in that case.
\end{rem}

\begin{rem}[Interpretation of $H^2(\Oep(G);k^\times)$] \label{gluing-rem}
The group $H^2(\Oep(G);k^\times)$ may be thought of as a ``$p$--local Schur
multiplier'', analogous to $H^2(G;k^\times)$. One may also ask if it also
has a representation theoretic interpretation, as a suitable
Brauer group? Note in this connection that
  $H^2(\calF^c;k^\times)$ occurs in connection with the so-called
  gluing problem for blocks, see
  \cite{linckelmann04,linckelmann05,linckelmann09}. 
\end{rem}

The last two remarks explain the
underlying picture on the level of homotopy.

\begin{rem}[Higher homotopy groups]\label{homotopy-remark}
That $H_i(\bO_\calC(G))$, $i>0$, is finite of order prime
to $p$,  for $\calC$ a collection of $p$--subgroups closed under passage to $p$--radical
overgroups, in fact has a strengthening, which also has the finiteness
in Proposition~\ref{boundslemma} as a special case:  For such $\calC$,  $\pi_i(\bO_\calC(G))$ is a finite $p'$--group for
{\em all}
$i$. This follows by a slight modification of the argument above:
As
$\pi_1(\bO_\calC(G))$ is a finite $p'$--group by
Proposition~\ref{boundslemma}, it is sufficient to show that
$H^i(X;\Z_{(p)}) =0$ for all $i>0$, for $X$ the
universal cover of $|\bO_\calC(G)|$, by the Hurewicz theorem modulo
Serre classes \cite[Thm.~20.6.1]{tomdieck08}. But
$H^*(X;\Z_{(p)}) \cong
H^*(|\bO_\calC(G)|;\Z_{(p)}[\pi_1(\bO_\calC(G))])$, equipping
$|\bO_\calC(G)|$ with the canonical twisted coefficient system (since
on chains
\begin{eqnarray*} C^*(|\bO_\calC(G)|;\Z_{(p)}[\pi_1(\bO_\calC(G))]) &\cong&
\Hom_{\pi_1(\bO_\calC(G))}(C_*(X), \Z_{(p)}[\pi_1(\bO_\calC(G))])\\
                                                                    &\cong& \Hom(C_*(X), \Z_{(p)}),
                                                                            \end{eqnarray*}
                                                                            by definition
\cite[Sec.~3.H]{hatcher02} and the finiteness of $\pi_1(\bO_\calC(G))$). This again equals
$\lim^*_{\bO_\calC(G)}\Z_{(p)}[\pi_1(\bO_\calC(G))]$, which vanishes
in positive degree as in the proof of Theorem~\ref{homotopy-finiteness} (all elements of order $p$ in $N_G(P)/P$ act trivially on
$\Z_{(p)}[\pi_1(\bO_\calC(G))]$, as $\pi_1(\bO_\calC(G))$ is a finite $p'$--group).
Note that this is in contrast to $\calT_\calC(G)$, where $\pi_i(\calT_\calC(G))
\xleftcong \pi_i(\calC)$, for $i \geq 2$, which is in general not
finite; e.g.,
$|\calS_p(G)|$ is homotopy equivalent to a wedge of spheres when $G$ is
a finite group of Lie type in characteristic $p$.
\end{rem}

\begin{rem}[Inverting $p$ on $\calT_\calC(G)$]\label{localization-rem}
  The relationship between the homotopy (or homology) groups of
  $\calT_\calC(G)$ and  $\bO_\calC(G)$ from above can be stated
  on the level of spaces: For $\calC$ as in Theorem~\ref{homotopy-finiteness},
\begin{equation}
 |\bO_\calC(G)| \simeq L_{p'} |\calT_\calC(G)|
\end{equation}
where $L_{p'}$ denotes localization with respect to the multiplication by
$p$ map $S^1\xrightarrow{p} S^1$ \cite{farjoun96}.
Namely, recall that $|\calT_\calC(G)| \simeq
\hocolim_{G/P \in \bO_\calC(G)} EG \times_G G/P$, see
\cite[\S\S1.7,3.2]{dwyer97}.
Hence 
\begin{multline}L_{p'} |\calT_\calC(G)| \simeq L_{p'} (\hocolim_{G/P \in \bO_\calC(G)} L_{p'}(EG
\times_G G/P)) \\ \xrightsimeq L_{p'}(\hocolim_{G/P \in \bO_\calC(G)}*)
\simeq  L_{p'}|\bO_\calC(G)|.
\end{multline}
Here we used that
$L_{p'}$ is a left adjoint \cite[Prop.~1.D.3]{farjoun96} for the first
homotopy equivalence and that
$L_{p'}(BP) \simeq *$ for a finite $p$-group $P$ (as is seen from the
definition or \cite[Thm.~3.2]{CP93}) for the second. Now the claim follows by observing
that  $L_{p'}|\bO_\calC(G)| \simeq  |\bO_\calC(G)|$, as $|\bO_\calC(G)|$ is space
with whose homotopy groups are finite $p'$--groups by
Remark~\ref{homotopy-remark}, which implies that it is $L_{p'}$--local
(see e.g., \cite[Cor.~2.13]{CP93}).
A very special case is if $\calC = \calS_p^e(G)$  where
$|\calT_\calC(G)| \simeq BG$, and thus
\begin{equation}|\bO_p(G)| \simeq L_{p'}BG
  \end{equation}
elaborating $H_i(\bO_p(G)) \cong
H_i(G)_{p'}$, $i>0$, from Remark~\ref{allpsubgroups-rem}.
\end{rem}

\section{Homology decompositions and the Carlson--Th\'evenaz conjecture}
\label{homologydec-sec}
In this section we establish the results about homology decompositions
stated in the introduction, and show how they imply the Carlson--Th\'evenaz
conjecture.  The key tool is the isotropy spectral sequence, recalled
below. Applied to the space $|\calC|$ this 
give us the normalizer decomposition (Theorem~\ref{limitformula}). For
the centralizer decomposition  (Theorem~\ref{centralizerthm})
we instead use the space $|E\A_\calC|$, where $E\A_\calC$ is the
overcategory $\iota \downarrow G$ for $\iota\co \calF_\calC \to \calF_{\calC
  \cup G}$ (see \S\ref{Gcategories} for details). (There is also a third decomposition, the subgroup
decomposition, based on a space $|E\bO_\calC|$, but since the isotropy
subgroups are $p$--groups, it does not provide us with
new information when taking coefficients prime to $p$.) 
We work in both homology and cohomology---these are essentially
equivalent, but from a practical viewpoint it may feel more convenient
to work in homology, only mapping into $k^\times$ at the end, so we
give both versions.

\subsection{Homology decompositions: proof of
  Theorems~\ref{limitformula} and \ref{centralizerthm}}\label{homologydec-subsec}
A Bredon {\em $G$--isotropy coefficient systems} is a functor
$\fF\co\bO(G) \to \Rmod$. It induces a $G$--coefficient system, as in \S\ref{coeff-subsec}, on $X$ via the canonical functor $(\Delta X)_G \to \bO(G)$ on objects sending $\sigma \mapsto
G/G_\sigma$.
Such coefficient systems are $G$--homotopy invariants, in the sense that a
$G$--homotopy equivalence $Y \to X$ induces a $RG$--chain homotopy
equivalence $C_*(Y;\fF) \to C_*(X;\fF)$.
Let $H_*^G(X;\fF) = H(C_*(X;\fF)^G)$ denote Bredon homology equipped
with an isotropy coefficient system $\fF$  (see e.g., \cite{bredon67}
for more details).
\begin{prop}[The isotropy spectral sequence] \label{isotropy-ss} Let
  $G$ be a finite group, $X$ a $G$--space, and $A$ an abelian group.
We have a homological isotropy spectral sequence for the action of $G$ on $X$
$$E^2_{i,j} = H^G_i(X;H_j(-;A)) \Rightarrow H_{i+j}(X_{hG};A)$$
The bottom right-hand corner produces an exact sequence
\begin{multline*}H_2(X_{hG};A) \to H_2(X/G;A) \to H_0^G(X;H_1(-;A))\\ \to
H_1(X_{hG};A) \to H_1(X/G;A) \to 0
\end{multline*}
The dual spectral sequence in cohomology produces
\begin{multline*}0 \to H^1(X/G;A) \to H^1(X_{hG};A) \to
  H^0_G(X;H^1(-;A)) \\ \to
H^2(X/G;A) \to H^2(X_{hG};A).
\end{multline*}
If $H_1(X/G;A) = H_2(X/G;A) = 0$ then this degenerates to 
$$H_1(X_{hG};A) \cong  H_0^G(X;H_1(-;A)).$$  Dually if 
 $H^1(X/G;A) = H^2(X/G;A) = 0$ then 
$H^1(X_{hG};A) \cong  H^0_G(X;H^1(-;A))$.
By definition
$$H_0^G(X;H_1(-;A)) = \coker\left( \oplus_{\sigma \in X_1/G} H_1(G_\sigma;A)
\xrightarrow{d_0 - d_1} \oplus_{\sigma \in X_0/G}   H_1(G_\sigma;A)\right)$$
for $X_i$ the non-degenerate $i$--simplices, and dually for
cohomology.

\end{prop}
\begin{proof} As explained in standard references such as
  \cite[\S2.3]{dwyer98sharp} \cite[VII(5.3)]{brown94}, the (homology) isotropy spectral sequence is constructed as the
  spectral sequence of the double complex $C_*(EG) \otimes_G C_*(X;A)$,
  filtered via the skeletal filtration of $X$. Hence $E^1_{*j} =
  H_j(G;C_*(X;A))$, and the $E^2$--term is obtained by taking homology induced by
  the differential on $C_*(X;A)$. 
The stated properties now follow from the definitions.
\end{proof}

We would like to alternatively view the $H_0^G$ in
Proposition~\ref{isotropy-ss} as a colimit, so we also recall the general
principle behind this: Recall from \S\ref{coeff-subsec}, that a
general (covariant) coefficient system on $X$ is just a functor
$\Delta X
\to \Rmod$, where $\Delta X$ is the category of simplices.

It is convenient to say that a space is {\em complex-like} if 
every non-degenerate simplex $\Delta[n] \to
X$ is an injection on sets, i.e., if it ``looks like'' an ordered
simplicial complex \cite[p.~311]{thomason80}. For a complex-like
space, the subdivision category $\sd X$ is the full subcategory of $\Delta X$ on the
non-degenerate simplices; it has a unique morphism $\sigma \to \tau$
if $\tau$ can be obtained from $\sigma$ via face maps, and no other
morphisms (see \cite[\S5]{DK83}). 
The following classical proposition gives the relationship we need,
stated also for higher homology for clarity:
\begin{prop}\label{limit-desc} Let $\D$ be a small category, and $R$
 a commutative ring.
\begin{enumerate}

\item \label{covar}For any functor $F\co \D \to \Rmod$,  $\colim^{\D}_*F  \cong
  H_*(|\D|;\scrF)$, where $\scrF$ is the coefficient system induced
  via $\Delta |\D| \to \D$, $(d_0 \to \cdots
\to d_n)\mapsto d_0$.

\item \label{contravar} For any functor $F\co \D^{\op} \to \Rmod$,
  $\colim^{\D^{\op}}_*F  \cong H_*(|\D|;\scrF)$, where $\scrF$ is the
  coefficient system induced via $\Delta |\D| \to \D^{\op}$, $(d_0 \to \cdots
\to d_n)\mapsto d_n$.

\item \label{simplex} Suppose $X$ is a complex-like space. For any functor $F\co \sd X
\to \Rmod$,  $\colim^{\sd X}_* F \cong H_*(X;\scrF)$, where $\scrF$ is
induced from $F$ via $\Delta X \to
\sd X$, the functor sending all degeneracies to identities (see
\cite[\S5]{DK83}).
\end{enumerate}
\end{prop}
\begin{proof}
We shall only need non-derived $*=0$ part of these statements, which
follows easily by writing down the definitions (for the last point
also using cofinality), which we invite the reader to do. 
For \eqref{covar} and
\eqref{contravar}, in the general case, see \cite[App.~II.3.3]{GZ67}, and also
\cite[Prop.~2.6]{grodal02}. (The point is that both
sides can be seen as homology of $C_*(|-\downarrow \D|) \otimes_\D F$
respectively $F \otimes_\D C_*(|\D \downarrow -|)$.) For
\eqref{simplex}, notice that both sides can be seen as the homology of
$C_*(|\sigma|) \otimes_{\sd X} F$, where $|\sigma|$ is the
$n$--simplex defined by the vertices of
$\sigma$ (by assumption distinct). It is a contravariant functor on $\sd X$ by to $(\sigma \to \tau)$
assigning the map induced by the unique face inclusion of $\tau$ in
$\sigma$ (i.e., the extra structure on $\sd X$ allows us to `avoid a
subdivision'; see also \cite[Prop.~7.1]{grodal02}).
\end{proof}

\begin{prop}\label{homformula}
Let $\calC$ be a collection of $p$--subgroups such that $H_1(|\calC|/G)_{p'} =
H_2(|\calC|/G)_{p'}=0$ then 
\begin{eqnarray*}
H_1(\bO_\calC(G))_{p'}\hspace{-7pt} &\xleftcong \hspace{-7pt}&  \coker\left( d_0 - d_1 \co\!\! \oplus_{[P<Q]} H_1(N_G(P<Q))_{p'} \to
\oplus_{[P]} H_1(N_G(P))_{p'}\right)\\
&\cong \hspace{-7pt}& \colim_{[P_0 < \cdots <P_n]}H_1(N_G(P_0 < \cdots <P_n))_{p'}
\end{eqnarray*}
where the colim is over $G$--conjugacy classes of strict chains in
$\calC$, ordered by reverse refinement.

The conditions on $\calC$ are satisfied if it is closed
under passage to $p$--radical overgroups, or just abstractly $G$--homotopy
equivalent to such a collection (e.g., 
$\calC = \calA_p(G)$).
\end{prop}
\begin{proof}
  First note that the conditions are indeed satisfied if $\calC$ is closed under passage
  to $p$--radical overgroups as $|\calC|/G$ is then contractible by Symonds' theorem,
  Proposition~\ref{webbsconj}. As the condition
  only depends on the $G$--equivariant homotopy type of $|\calC|$,
  it also holds if $|\calC|$ is only abstractly $G$--homotopy equivalent
  to $|\calC'|$ for another collection $\calC'$ which is closed under
  passage to $p$--radical overgroups. This is the case for
  $\calA_p(G)$, as $|\calA_p(G)| \to |\calS_p(G)|$ is a $G$--homotopy
  equivalence (see Theorem~\ref{propagating}\eqref{reduce-to-eltab}).

Now, if $H_1(|\calC|/G)_{p'} =
H_2(|\calC|/G)_{p'}=0$ then the isotropy
spectral sequence, Proposition~\ref{isotropy-ss}, applied to the
$G$--space $|\calC|$  with $A = \Z[\frac1p]$ gives
$H_1(|\calC|_{hG}; \Z[\frac1p]) \cong H^G_0(|\calC|;H_1(-;\Z[\frac1p]))$.
Since the right-hand side is obviously finite so is the left-hand
side, hence taking coefficients in $\Z[\frac1p]$ is the same as applying $(-)_{p'}$.
Furthermore
$ H_1(\bO_\calC(G))_{p'} \cong H_1(\calT_\calC(G))_{p'} \cong H_1(|\calC|_{hG})_{p'}$
by Proposition~ \ref{pprime} and Lemma~\ref{thomasonlemma}.
The first formula now follows by definition of $H^G_0$.
The second rewriting as a colimit over conjugacy classes of strict
chains follows from Proposition~\ref{limit-desc}\eqref{simplex},
noting that $|\calC|/G$ is complex-like. 
\end{proof}

\begin{proof}[\thmunderline{Proof of Theorem~\ref{limitformula}}]
We have that
$\TS \cong \Hom(H_1(\Oep(G)),k^\times)$, by Theorem~\ref{main} and \eqref{algtop}. Combining this with
Proposition~\ref{homformula} for $\calC = \calS_p(G)$ gives the wanted expressions, using
that $\Hom(-,k^\times)$ sends colimits to limits.  Again, it is a subset of
$\Hom(N_G(S)/S,k^\times)$ by Proposition~\ref{boundslemma}.
\end{proof}

We now prove the centralizer version.

\begin{prop} \label{centralizer-hom} For $\calC$ a collection of
  $p$--subgroups and $A$ an abelian group, we have exact sequences
$$ H_2(\calF_\calC(G);A) \to \underset{P \in \calF_\calC(G)^{\op}}\colim H_1(C_G(P);A) \to
H_1(\calT_\calC(G);A) \to H_1(\calF_\calC(G);A) \to 0$$
and
$$ 0 \to  H^1(\calF_\calC(G);A)  \to H^1(\calT_\calC(G);A) \to \lim_{P
  \in \calF_\calC(G)}H^1(C_G(P);A) \to  H^2(\calF_\calC(G);A) 
$$

Here we may replace $\calT_\calC$ by $\bO_\calC$ if $A$ satisfies the
assumptions in Proposition~\ref{pprime}.

Furthermore to calcate the limit (or colimit over the opposite category) we may replace $\calF_\calC(G)$ by a final subcategory, e.g., if $\calC$ is a collection of non-trivial $p$--subgroups
containing the elementary abelian $p$--subgroups of rank one or two
$\calA_p^2(G)$, then we may replace $\calC$ by $\calA_p^2(G)$.
\end{prop}
\begin{proof}
Consider the isotropy exact sequence in homology from Proposition~\ref{isotropy-ss},
with $X$ the $G$--space
$|E\A_\calC|$ introduced in the beginning of this section (and in more
detail in \S\ref{Gcategories}):
\begin{multline*} H_2(|E\A_\calC|/G;A) \to
H_0^G(|E\A_\calC|;H_1(-;A)) \to
H_1(|E\A_\calC|_{hG};A) \\ \to H_1(|E\A_\calC|/G;A)
\to 0
\end{multline*}
We want to identify this sequence with the first sequence of the proposition.
As remarked in \S\ref{Gcategories}, $|E\A_\calC|/G =
|\calF_\calC|$, which identifies the first and the fourth term.
For the third term, remark that, again by  \S\ref{Gcategories}, the $G$--map $|E\A_\calC| \to |\calC|$
is a homotopy equivalence, and hence induces  $|E\A_\calC|_{hG} \xrightsimeq
|\calC|_{hG}$, and the space $|\calC|_{hG}$ again identifies with
$|\calT_\calC(G)|$ by Lemma~\ref{thomasonlemma}. For the second term, notice that the stabilizer of
an $n$--simplex $i\co V_0 \to V_1 \to \cdots \to V_n \to G$ is
$C_G(i(V_n))$. Hence Proposition~\ref{limit-desc}\eqref{contravar}
also identifies $H^G_0(|E\A_\calC|;H_1(-;A))$ with the colimit as
stated.

The sequence in cohomology follows dually using the isotropy exact
sequence in cohomology from Proposition~\ref{isotropy-ss}, and the
dual version of Proposition~\ref{limit-desc}\eqref{contravar} stated
e.g., as \cite[Prop.~2.6]{grodal02}.

The addendum about replacing  $\calT_\calC$ by $\bO_\calC$ follows
directly from Proposition~\ref{pprime}, and the replacement of
categories by a (co)final subcategory is a general fact about
calculation of (co)limits, as explained e.g., in
\cite[XI.3]{maclane71}. The stated example is easily seen to be
final.
\end{proof}

\begin{proof}[\thmunderline{Proof of Theorem~\ref{centralizerthm}}]
  This follows from the cohomological sequence in
  Proposition~\ref{centralizer-hom}, taking  $\calC$ to be the
  collection of all non-trivial $p$--subgroups and $A=k^\times$, and
  utilizing the two additions at the end of Proposition~\ref{centralizer-hom}:
Using  Proposition~\ref{pprime} and Theorem~\ref{main} we may replace
 $H^1(\Tep(G);k^\times) \xleftcong H^1(\Oep(G);k^\times)
 \xleftcong \TS$, and also restrict to elementary abelian
 $p$--subgroups of rank one or two in the limit by finality.

For the final part, note that since the
centralizers of elements $x$ of order $p$ are assumed to satisfy
$H_1(C_G(x))_{p'}=0$ (i.e., are  ``$p'$--perfect''), the inverse
limit is obviously zero since the values are zero. The assumptions
on the action of $N_G(S)$ implies that $H^1(\calF_p^*(G);k^\times) =0$ by Corollary~\ref{vanishinglemma}.
\end{proof}

\begin{rem}[Isotropy versus Bousfield--Kan spectral sequence] \label{isotropyrem} The above arguments in terms of the
  isotropy spectral sequence can also be recast in terms of the
  Bousfield--Kan spectral sequence of a homotopy colimit.
As explained e.g., in \cite[\S3]{dwyer98sharp}\cite[\S3.3]{dwyer97}, the isotropy spectral
sequence identifies with the Bousfield--Kan spectral sequence
for the normalizer homology decomposition
$$|\calC |_{hG} \simeq \hocolim_{\sigma \in |\calC|/G} EG \times_G G/G_\sigma$$
It is also possible to work with $\bO_\calC(G)$ directly, instead
of passing via $|\calC|_{hG}$, 
since by \cite[Cor.~2.18]{slominska91} the orbit category admits a normalizer decomposition
\begin{equation}\label{slominska} |\bO_\calC(G)| \simeq \hocolim_{(P_0< \cdots < P_n) \in |\calC|/G} BN_G(P_0
< \cdots < P_n)/P_0
\end{equation}
(where $ BN_G(P_0< \cdots < P_n)/P_0$ has to be interpreted as
$E(G/P_0) \times_G G/N_G(P_0<\cdots<P_n)$, for $E(G/P_0)$ the translation
groupoid of the $G$--set $G/P_0$, in order to get a strict functor to spaces).
The associated spectral sequence for this homotopy
colimit, can also be obtained in a more low-tech way
as the Leray spectral sequence of the projection map $|\bO_\calC(G)| = |E\bO_\calC|/G \to
|\calC|/G$. 
\end{rem}

\subsection{The Carlson--Th\'evenaz conjecture: proof of
  Theorem~\ref{carlsonthevenazconj} and Corollary~\ref{radicalsnormal}}\label{ctconj-proofsubsec}
We now prove the results in \S\ref{ctconj-sec}, and in
particular deduce Theorem~\ref{carlsonthevenazconj} from
Theorem~\ref{limitformula}. Via  Theorem~\ref{main}, this will amount
to describing how the colimit appearing in
Proposition~\ref{homformula} can be calculated in certain cases,
simplified using a Frattini
argument. 
We
define the {\em $p$--radical closure}  $\overline{P}$ of a
$p$--subgroup $P$ in $G$ as
the subgroup obtained by successively applying
$O_p(N_G(-))$, starting from $P$ until the
process stabilizes. Let $N_{G,p}(P)$ denote a Sylow $p$--subgroup of
$N_G(P)$, well defined up to $N_G(P)$--conjugation, and write
$\calB_p^e(G)$ for the collection of all $p$--radical subgroups of $G$.

\begin{thrm} \label{homformula-nr}
For  $\calC$ a collection of $p$--subgroups closed under passage to
$p$--radical overgroups
\begin{multline*}H_1(\bO_\calC(G)) \cong \\ \coker
  \mkern-2mu\left(\mkern-2mu d_0 - d_1 \co \! \mkern-3mu
                               \oplus_{[P] \in \calC'^{<S}/G}
                               \mkern-2mu H_1(N_G(P<\overline{N_{G,p}(P)}))_{p'} \mkern-1mu\to \mkern-1mu
                               \oplus_{[P] \in \calC'/G}
                               H_1(N_G(P))_{p'}\mkern-2mu\right)
                             \end{multline*}
                             with $\calC' = \calC \cap \calB^e_p(G)$ and
                             $\calC'^{<S}/G$ denoting $G$--conjugacy classes of
                             $p$--radical subgroups in $\calC$ except $[S]$. 
\end{thrm}
\begin{proof} First note that $H_1(\bO_\calC(G))$ is a finite
  $p'$--group by Proposition~\ref{boundslemma}.
We want to see that the cokernel formula in
Proposition~\ref{homformula} can be reduced to the simpler
expression above. By Theorem~\ref{propagating}\eqref{reduce-to-rad} we can without
restriction assume that $\calC$ is closed under passage to all
$p$--overgroups.
Before we start, also note that the domain in Proposition~\ref{homformula}
runs over conjugacy
classes of pairs $P<Q$, whose elements we can view as conjugacy classes of
subgroups $P \in \calC$ together with, for each $P$, $N_G(P)$--conjugacy classes of
subgroups $Q \in \calC$, with $P<Q$.  (Recall that the formula is
well defined by the
identification of the result as a zeroth homology group; more
na\"ively one may note that $N_G(P)$ acts trivially on
$H_1(N_G(P))_{p'}$.)   Set $H(-) = H_1(-)_{p'}$ for short.

Also observe that $N_G(P) = N_G(P \leq \overline P) \leq N_G(\overline P)$
for $\overline P$ the $p$--radical closure of $P$, and we hence
have an induced map $H(N_G(P)) \to H(N_G(\overline P))$, which identifies a
summand corresponding to $[P]$ with its image in the summand
corresponding to $[\overline P]$. This enables us to view the
 cokernel in Proposition~\ref{homformula} as a quotient of  $\oplus_{[P] \in \calC'/G}
 H(N_G(P))$ via these maps.

Our task is thus reduced to showing that the image of
 $\oplus_{[P<Q] \in \calC/G} H(N_G(P<Q))$,  via the map from Proposition~\ref{homformula}, agrees with the
 image of the subgroup defined by letting the sum run over just
 $[P < \overline P]$ and
$[\overline P < \overline{N_{G,p}(\overline P)}]$ for $P \in \calC$.
We will do this by downward induction on the size of $P$. If $P$ has
index $p$ in $S$ there in nothing to show as $Q$ will necessarily be
Sylow, which is included in the above.  So assume that the statement is
true for larger subgroups. We divide the induction into steps.

We first reduce to the case where $P$
is normal in $Q$. Namely, if not set $Q' = N_Q(P)$ and note that $P<Q'$,
as $Q$ is a $p$--group. We claim that the image under $d_0-d_1$ of the summand corresponding to $P <
Q$ lies in the image under $d_0-d_1$ generated by the summands $P \lhd Q'$ and
$Q'<Q$. Namely consider the diagram
$$\xymatrix@R-5pt{ & H(N_G(P< Q' <  Q)) \ar@{=}[d] \ar[dr] \ar[dl]&\\
H(N_G(P< Q'))  \ar[d] \ar[dr] & H(N_G(P<  Q)) \ar[dl] \ar[dr] &
H(N_G(Q' <  Q)) \ar[d] \ar[dl] \\
 H(N_G(P)) &H(N_G(Q'))&  H(N_G(Q))  }
$$
and note that the image of any element $x\in  H(N_G(P<  Q))$ equals the
image of a sum $x_1 + x_2$
where $x_1 \in H(N_G(P <  Q'))$ and $x_2 \in H(N_G(Q'< Q))$ are the
images of $x \in  H(N_G(P <  Q')) = H(N_G(P < Q'<Q))$ under the
maps induced by inclusion of normalizers. As $Q'$ is strictly bigger
than $P$ we are reduced to the case where $P$ is normal in $Q$ by
induction.

Next consider the case of $P \lhd  Q$ with $P \neq \bar P$.
We claim that the image of the summand $P<Q$ is generated by the image
of the summands corresponding to $P < \overline P$, $\overline P \leq \overline P Q$ and $Q\leq
\overline P Q$, noting that $\overline P Q$ is again a $p$--group as
$Q$ normalizes $P$ and hence $\overline P$. This will prove the claim,
using the 
induction hypothesis, as $\overline P$ and $Q$ are strictly larger than
$P$. (Note that if the weak inequalities above are equalities, the map is
just zero, so we do not have to separate out this case; topologically
this corresponds to degenerate simplices.)
The generation follows from the following diagram:
$$\xymatrix@C-9pt{H(N_G(P<Q)) \ar[d] \ar[drr] \ar[dr] \ar[r] \ar@/^1pc/[rrr]
  \ar@/_1pc/[rr]&
 H(N_G(P< \overline P)) \ar[d] \ar[dl] & H(N_G(\overline P \leq \overline P Q)) \ar[dl] \ar[dr]    &
  H(N_G(Q \leq \overline P Q)) \ar[dl] \ar[d]\\
H(N_G(P)) \ar[r]& H(N_G(\overline P)) & H(N_G(Q)) & H(N_G(\overline P Q))}$$
Namely, the image under $d_0 - d_1$ of any element $x\in  H(N_G(P<  Q))$ equals
the image of $x_1 + x_2 -x_3$ where $x_1 \in H(N_G(P <  \overline P))$, $x_2
\in H(N_G(\overline P \leq \overline P Q))$, and $x_3 \in  H(N_G(Q
\leq \overline P Q))$ are the images of $x$ under the maps induced by
the inclusion of normalizers (the top horizontal arrows in the diagram).

Finally consider the case $P \lhd Q$ with $P =\overline P$, which we
want to replace with $P < \overline{N_{G,p}(P)}$. Let
$R$ be a Sylow $p$--subgroup of $N_G(P<Q)$ (thus containing $Q$).
We first claim that the image of the summand corresponding to $P<Q$ is
generated by the summands for $P <  R$ and $Q \leq R$. For this
consider the diagram
\begin{equation}\label{reduction-diagram}
\vcenter{\xymatrix@R-5pt{H(N_G(P<Q\leq R)) \ar@{->>}[d] \ar[dr] \ar[drr]& &\\
  H(N_G(P<Q)) \ar[d] \ar[dr] & H(N_G(P <  R)) \ar[dr] \ar[dl]
  & H(N_G(Q \leq R)) \ar[dl] \ar[d] \\
  H(N_G(P)) & H(N_G(Q)) & H(N_G(R))}}
\end{equation}
where  $H(N_G(P < Q \leq R))  \twoheadrightarrow H(N_G(P< Q))$
is surjective by the Frattini argument, as $A^{p'}(N_G(P<Q))$ is
normal in $N_G(P<Q)$ of $p'$ index.  Given any $x \in
H(N_G(P<Q))$, we can lift it to $x' \in H(N_G(P<Q\leq R))$ and let
$x_1$ and $x_2$ be the image of $x'$ under the inclusion of normalizers in
$H(N_G(P <  R))$ and $H(N_G(Q < R))$. Then the
image of $x$ under $d_0 -d_1$ agrees with the image of $x_1 - x_2$ by
the above diagram. If $Q=R$ then $R$ is a Sylow $p$--subgroup of
$N_G(P)$ and hence we may replace $P  \lhd Q$ with $P < R = N_{G,p}(P)$ as
the summand corresponding to $Q \leq R$ is taken care of by the
induction hyphothesis. If $Q<R$ then the argument shows that we may
replace $P<Q$ with $P<R$, and we can hence repeat until
$Q$ is indeed
Sylow in $N_G(P)$. Finally as  $N_G(P < N_{G,p}(P)) \leq N_G(P <
\overline{N_{G,p}(P)})$ we may replace $P < N_{G,p}(P)$  by $P <
\overline{N_{G,p}(P)}$ as wanted, finishing the proof.
\end{proof}

  Next we show how a cokernel such as in Theorem~\ref{homformula-nr}
  can be described iteratively.
  \begin{lemma} \label{filtration-lemma} Suppose we have an increasing filtration $\{s\}= X_0
    \subset X_1 \subset \cdots \subset X_n = X$, $n \geq 1$, of a finite set $X$,
    and a function $\phi\co X \setminus
    \{s\} \to X$, which strictly decreases filtration. Let $\oplus_{x
      \in X} A(x)$ and $\oplus_{x
      \in X} B(x)$ be abelian groups, and suppose we, for each $x \in X \setminus
    \{s\}$, are given
    two homomorphisms $f_x: A(x) \twoheadrightarrow B(x)$ and $g_x: A(x)
    \to B({\phi(x)})$, with $f_x$ surjective.
    Set $B_0(x) = 0$ and $B_i(y) =
     \sum_{x \in
      \phi^{-1}(y)} g_x(f_x^{-1}(B_{i-1}(x)))$.
  Then $$\coker\left(\oplus_{x\in X \setminus \{s\}} A(x) \xrightarrow{g_x - f_x}
    \oplus_{x\in X} B(x)\right)
  \cong B(s)/B_n(s)$$
 \end{lemma}
\begin{proof}
  We prove this by induction on $n$. The statement is true for $n=1$ since any element
  of $B(x)$, for $x \neq s$, is identified with a unique element of
  $B(s)/B_1(s)$, by the surjectivity of $f_x$ and the definition of
  $B_1(s)$. Suppose that it is true for $i<n$. Notice that
\begin{multline*}
 \coker\left(\oplus_{x\in X \setminus \{s\}} A(x) \xrightarrow{g_x - f_x}
    \oplus_{x\in X} B(x)\right)\\
  \cong \coker\left(\oplus_{x\in X \setminus \{s\}} A(x) \xrightarrow{g_x - f_x}
          \oplus_{x\in X} B(x)/B_1(x)\right)\\
  \cong \coker\left(\oplus_{x\in X_{n-1}\setminus \{s\}} A(x) \xrightarrow{g_x - f_x}
           \oplus_{x\in X_{n-1}} B(x)/B_1(x)\right)
   \end{multline*}
 Here the first isomorphism is because elements of $B_1(x)$ are
 obviously zero in the cokernel and
  the second isomorphism follows as elements of $\oplus_{x\in X_n \setminus
  X_{n-1}}B(x)$ each get identified with unique elements of $\oplus_{x\in
  X_{n-1}}B(x)/B_1(x)$, by surjectivity of $f_x$ and the definition of $B_1$. As  $\oplus_{x
      \in X_{n-1}} A(x)$ and $\oplus_{x
      \in X_{n-1}} B(x)/B_1(x)$ satisfy the assumptions of the original setup,
 we are done by induction.
  \end{proof}

For any conjugacy class of $p$--radical subgroups $[P]$ consider the chain
constructed by taking $[P_0] = [P]$ and $[P_i] =
[\overline{N_{G,p}(P_{i-1})}]$, well defined on conjugacy classes. Define the {\em normal-radical height} of
  $[P]$ in $G$ as
the smallest $i$ such that $[P_i] =[S]$, with $S$ a Sylow
$p$--subgroup. Define the {\em normal-radical class} of a collection
of $p$--subgroups $\calC$ as the maximal 
normal-radical height of a $p$--radical $[P] \in \calC/G$.
(E.g., normal-radical class $0$ means that only $[S]$ is $p$--radical.)

\begin{thrm} \label{carlsonthevenaz-general2}
For a finite group $G$ with Sylow $p$--subgroup $S$, let $\calC$ be a collection of $p$--subgroups, closed
    under passage to $p$--radical overgroups. Set $\calC' = \calC
    \cap \calB_p^e(G)$. For $Q \in \calC$, let $\nu^1_{\calC}(Q) = A^{p'}(N_G(Q))$ and define by induction
    $$\nu_{\calC}^i(Q) =  \langle (N_G(Q)\cap \nu_{\calC}^{i-1}(P))A^{p'}(N_G(Q))
| [P] \in  \calC'/G \mbox{ with } [Q] = [\overline{N_{G,p}(P)}] \rangle.$$
picking for each $[P]$ a representative $P$ such that
$N_Q(P)$ is Sylow in $N_G(P)$ and
$\overline{N_Q(P)}=Q$.
Then 
$$H_1(\bO_\calC(G)) \cong N_G(S)/\nu_{\calC}^{r}(S)$$
for $r$ at least the normal-radical class of $\calC$ plus $1$.
  \end{thrm}

  \begin{proof}
Theorem~\ref{homformula-nr} provides a formula for
$H_1(\bO_\calC(G))$. We claim that Lemma~\ref{filtration-lemma}
allows us to reformulate that expression to the one given in the
theorem. Namely take $X = \calC'/G$,
$A([P])= H_1(N_G(P<\overline{N_{G,p}(P)}))_{p'}$ and $B([P]) =
H_1(N_G(P))_{p'}$, and notice that the induced map $f_{[P]}\co A([P]) \to B([P])$ is surjective by the
diagram
$$\xymatrix{        H_1(N_G(P <  N_{G,p}(P) \leq
  \overline{N_{G,p}(P)}))_{p'} \ar[d]^= \ar[r] & H_1(N_G(P<\overline{N_{G,p}(P)}))_{p'}  \ar[d]^{f_{[P]}} \\
      H_1(N_G(P <  N_{G,p}(P)))_{p'} \ar[r]^{d_0} &
    H_1(N_G(P))_{p'}}$$
together with the surjectivity of $d_0$, which follows from the Frattini argument with respect to the normal
subgroup $A^{p'}(N_G(P))$.
Let $X_i \subset X$ consist of those conjugacy
classes of subgroups of normal-radical height at most $i$, as defined
just before the theorem.
Then by definition $B_i([P])
= \nu^{i+1}_{\calC}(P)/A^{p'}(N_G(P))$, and it now follows from
Theorem~\ref{homformula-nr}   and Lemma~\ref{filtration-lemma}
that $H_1(\bO_\calC(G))_{p'} \cong N_G(S)/\nu_{\calC}^{r}(S)$, for $r$ at
least the normal-radical class of $\calC$ plus $1$.\end{proof}

\begin{lemma}\label{nilpotency-class} Let $\calC \subseteq \calS_p(G)$ be a collection closed under passage to $p$--radical
  overgroups, and let $\calC' = \calC \cap \calB_p(G)$.  The normal-radical class of $\calC$ 
is bounded by the nilpotency
class of $S$, as well as by $\min\{\dim |\calC'|,
\dim |\calC'^c| +1\}$, where $\calC'^c$ means the $p$--centric
subgroups in $\calC'$ (cf.\ \S \ref{fusion-subsec}).
\end{lemma}
\begin{proof}
 Let $P=P_0$ be an arbitrary $p$--radical subgroup, and let $[P_0]$,
  \ldots, $[P_n]=[S]$ be the normal-radical series as above.
  For the nilpotency bound, let $Z_i$ be the
$i$th group in the lower central series of $S$, i.e., $Z_1 = Z(S)$
etc. We can assume $P_0 \leq S$. As $Z_1$
normalizes $P_0$, we can choose representative $P_1$ with $[P_1] =
[\overline{N_{G,p}(P_0)}]$ and $Z_1 \leq P_1$.  Assume by induction that
we have chosen $P_{i-1}$ with $Z_{i-1} \leq P_{i-1}$. Then $Z_i/Z_{i-1}$ centralizes $P_{i-1}/Z_{i-1}$, and in
particular $Z_i$ normalizes $P_{i-1}$ we can choose $P_i$ with $Z_i
\leq P_i$, showing the claim. The $\dim |\calC'|$ bound is obvious,
and the $\dim |\calC'^c|+1$ bound holds as $P_i$ 
will be $p$--centric for $i\geq 1$, as they contain a Sylow $p$--subgroup of $C_G(P_0)$.
\end{proof}

\begin{proof}[{\thmunderline{Proof of Theorem~\ref{carlsonthevenazconj}}}]
It follows from Theorem~\ref{carlsonthevenaz-general2} that $\nu^r_\calC(S)  = \ker( N_G(S) \to H_1(\Oep(G)))$, for
$\calC = \calS_p(G)$ and $r$ at least the normal radical class of
$\calC$ plus $1$.
This shows a version of
Theorem~\ref{carlsonthevenazconj}, with $\nu^r$ instead of $\rho^r$,
noting that the stated bounds are implied by Lemma~\ref{nilpotency-class}.
To be able to replace $\nu$ by $\rho$, notice first that  $\nu^i_\calC(Q)
\leq \rho^i(Q)$ for any $i$ and $Q \in \calC \subseteq
\calS_p(G)$, by definition (for $\rho$, unlike $\nu$, we do not assume that the subgroups
are related by inclusion). Hence, to finish the proof, we just need to
verify that also $\rho^i(S)$ lies in the kernel of $N_G(S) \to
H_1(\Oep(G))$, since then $\nu^r_\calC(S) \cong \rho^r(S)$.
However that $\rho^i(Q)$ lies in the kernel of $N_G(Q) \to
H_1(\Oep(G))$ for any $1 <Q \leq S$ and any $i$ follows essentially
by definition (like for $\nu$), as we now verify by induction on $i$: 
For $\rho^1(Q)$ it follows by the factorization $N_G(Q) \to H_1(N_G(Q))_{p'} \to
H_1(\Oep(G))$, as $H_1(\Oep(G))$ is a $p'$--group by
Proposition~\ref{boundslemma}. And,
if $g \in N_G(Q) \cap \rho^{i-1}(R) \subseteq \rho^i(Q)$, for $1<R \leq S$, then we have a diagram
\begin{equation}\label{final-zigzag}
\vcenter{\xymatrix{G/Q \ar[d]^{[g]} \ar[r] & G/QR\ar[d]^{[g]} &\ar[l] \ar[d]^{[g]} G/R\\
  G/Q \ar[r] & G/QR & G/R \ar[l]}}
\end{equation}
where $QR$ denotes the subgroup generated by $Q$ and $R$ inside $S$.
This shows that the image of $g \in N_G(Q) \cap \rho^{i-1}(R)$ in $H_1(\Oep(G))$
via $$\rho^{i-1}(R)  \to N_G(R) \to H_1(N_G(R)) \to H_1(\Oep(G))$$
equals the image of $g$ via $\rho^{i}(Q) \to N_G(Q) \to H_1(N_G(Q)) \to
H_1(\Oep(G)),$ which is hence also zero by induction. As $\rho^i(Q)$
is generated by such $g$ we conclude that $\rho^i(Q)$ maps to zero in
$H_1(\Oep(G))$ as wanted.
\end{proof}

\begin{rem}
Note that the statement in the last proof that $\rho^i(S)$
  lies in the kernel of $N_G(S) \to
H_1(\Oep(G))$, via the dictionary of Theorem~\ref{main}, 
amounts to the statement that Sylow-trivial modules split as a
trivial module $k$ direct sum a projective module upon restriction to 
$\rho^i(S)$ which was already shown by
Carlson--Thevenaz (see \cite[Thm.~4.3]{CT15}).
 \end{rem}

\begin{rem}[The bound $r$ in
  Theorem~\ref{carlsonthevenazconj} and
sparsity of Sylow-trivial modules]\label{CT-bound-rem}
In Appendix~\ref{propagating-sec} we provide  a detailed analysis of how to
find small
collections $\calC  \subseteq \calS_p(G)$ such that $H_1(\Oep(G))_{p'} \cong H_1(\bO_\calC(G))_{p'}$, and hence get other bounds on
$r$ in the Carlson--Th\'evenaz conjecture,
Theorem~\ref{carlsonthevenazconj}: 
By 
Theorem~\ref{remove-non-essential} and
  Proposition~\ref{webbsconj} we can take $\calC$ to be the smallest
  collection closed under passage to $p$--radical overgroups and
  containing all the $p$--subgroups $P$ where $N_G(P)/P$ admits an exotic
  Sylow-trivial module (a closure of the collection $\calG_p(G)$ of
  \S\ref{pruningfundgrp-subsec}).
For a finite group
of Lie type of characteristic $p$ this subcollection of $\calB_p(G)$ identifies with unipotent radicals of parabolic subgroups
of rank at most one. As far as we know, this poset could have a
uniform dimension
bound in general, independent of the finite group $G$.
The group $G_2(5)$ at
  $p=3$ is an example where both bounds in Theorem~\ref{carlsonthevenazconj} give $r=3$ and $r=2$
  does not work (see the discussion before
  Proposition~\ref{G_25p3}). 
We do not know of an example where one
  cannot take $r=3$. In fact, to the best of our knowledge,
  in all finite groups where $\TS$ has been calculated either $\rho^3(S) =
  N_{A^{p'}(G)}(S)$ (and hence $\TS = \Hom(G,k^\times)$), or $\calS_p(G)$ is $G$--homotopy equivalent to a $1$--dimensional
  complex. The results of this paper
indicate that finding bounds on $r$ in general has links to many facets of 
 $p$--local finite group theory (see also
 \S\S\ref{pruningfundgrp-subsec},\ref{components-subsec}).
\end{rem}

Finally, let us finally address in more detail when one can take $r=2$ as
bound on the filtration.
  
\begin{cor}\label{CTcor}
Let $\calC$ be a collection of $p$--subgroups, closed
    under passage to $p$-radical overgroups, and set $\calC' = \calC
    \cap \calB_p(G)$. If each $[P] \in \calC'/G$ satisfies that
  $[\overline{ N_{G,p}(P)}] =[S]$ then
$$H_1(\bO_\calC(G)) \cong N_G(S)/\langle N_G(P\leq S) \cap
 A^{p'}(N_G(P)) | [P] \in \calC'/G \rangle$$
where each $P$ is chosen in $[P]$ such that $N_S(P)$ is a Sylow $p$--subgroup
of $N_G(P)$.

More generally, if one just, for each $[P] \in \calC'/G$,
can pick $P \leq
  S$ with $N_S(P)$ Sylow in $N_G(P)$ and
 $N_G(P \leq \overline{N_S(P)} \leq S)A^{p'}(N_G( P \leq
 \overline{N_S(P)})) = N_G(P \leq \overline{N_S(P)} )$ then
 the same conclusion holds.
  \end{cor}
  \begin{proof}
   The first part follows directly from
Theorem~\ref{carlsonthevenaz-general2}, as normal-radical class of $\calC$
is at
most $1$. The 'more generally' part follows
from its proof: Under the stated assumption, diagram \eqref{reduction-diagram} in the proof of
Theorem~\ref{homformula-nr} 
shows that we can replace $P \leq
\overline{N_S(P)}$ with $P \leq S$ (taking ($Q =
\overline{N_S(P)}$  and $R =S$), so the cokernel in
Theorem~\ref{homformula-nr} can be calculated as claimed in this corollary.
  \end{proof}

  \begin{proof}[\thmunderline{Proof of Corollary~\ref{radicalsnormal}}]
If all $p$--radical
     subgroups $P \leq S$ are normal, then in particular
     $[\overline{N_{G,p}(P)}] = [S]$ and $S = N_S(P)$ is a Sylow
     $p$--subgroup of $N_G(P)$, so the first part follows from
     Corollary~\ref{CTcor}, together with Theorem~\ref{main}. For the `more generally' part, we note
     that the assumption implies that of
     Corollary~\ref{CTcor}:
By a Frattini argument $N_G(P \leq Q \leq S)A^{p'}(N_G(P \leq \overline{Q} \leq S)) = N_G(P \leq \overline{Q} \leq S)$ and $N_G(P \leq Q) A^{p'}(N_G(P \leq \overline{Q}))  = N_G(P \leq \overline{Q}) $, as $Q$ is a Sylow
     $p$--subgroup and $N_G(Q) = N_G(Q \leq \overline{Q})$. Hence
     $N_G(P \leq \overline{Q} \leq S) A^{p'}(N_G(P \leq
     \overline{Q})) =  N_G(P \leq Q \leq S) A^{p'}(N_G(P \leq
     \overline{Q})) = N_G(P \leq Q) A^{p'}(N_G(P \leq
     \overline{Q})) = N_G(P \leq \overline{Q})$ as wanted.
     \end{proof}
Let us for completeness also state a variant of Corollary~\ref{CTcor}
valid for general collections.
 \begin{cor}\label{CTcor2}
Let $\calC$ be a collection with $S \in \calC$ and $H_1(|\calC|/G)_{p'} =
H_2(|\calC|/G)_{p'}=0$.
Assume that for each
   $G$--conjugacy class of pairs $[P \leq Q]$ with $P, Q \in \calC$,
   we can pick $P
   \leq Q \leq S$ such that
   $N_G(P \leq Q \leq S)A^{p'}(N_G(P\leq Q)) = N_G(P\leq Q)$ (e.g., if
 all subgroups in $\calC_{\leq S}$ are normal in $S$), then
    $H_1(\bO_\calC(G))_{p'} \cong N_G(S)/\langle N_G(P\leq S) \cap
 A^{p'}(N_G(P)) | [P] \in \calC/G \rangle$, with representative $P$
 picked as above.
    \end{cor}
   \begin{proof}
The formula for the cokernel in Proposition~\ref{homformula} reduces
to the formula above via the diagram \eqref{reduction-diagram} (with $R =S$)
as before, where the surjective map now is by assumption.
     \end{proof}

\begin{rem}[{A strong version of \cite[Thm.~7.1]{CT15}}]\label{CT7.1rem}
Suppose that $N_G(S)$ controls
$p$--fusion in $G$, and that for each nontrivial $p$--radical subgroup $Q \leq S$
$$(N_G(S)\cap C_G(Q)) A^{p'}(C_G(Q))= C_G(Q)$$
Then the general assumption of
Corollary~\ref{CTcor} is satisfied. Namely 
\begin{eqnarray*}
N_G(P<Q) = N_G(P<Q<S)C_G(Q) &=&   N_G(P<Q<S)A^{p'}(C_G(Q))\\
&\leq&  N_G(P<Q<S)A^{p'}(N_G(P < Q))
\end{eqnarray*}
and the other inclusion is clear. Here the first equality is by control of fusion and the second by 
assumption. This provides a slightly stronger version of
\cite[Thm.~7.1]{CT15}, where the condition is only checked on $p$--radical subgroups.

\end{rem}

\section{Computations} \label{computations-sec}
By Theorem~\ref{main} calculating $\TS$ amounts to calculating
$H_1(\bO_p^*(G))$, and we have developed a number of theorem and tools
for this in the preceding sections. Formulas such as
Theorems~\ref{limitformula} and \ref{centralizerthm} make it
computable for individual groups, since the input data has often
already been tabulated, e.g., in connection with inductive approaches
to the Alperin and McKay conjectures. Similarly Theorem~\ref{sequence} allows us
to tap into the large preexisting literature on the fundamental group
of subgroup complexes, which has been studied in topological
combinatorics, due to its relationship to other combinatorial
problems, as well as in finite group theory, where it is related to
uniqueness question of a group given its $p$--local structure, and
the classification of finite simple groups.
Expanding on the summary in \S\ref{computations-subsection} we will in
this section go through different classes of groups, and
show how the strategy translates into explicit computations. We only
pick some low-hanging fruit, but with a
recipe for how to continue.

\subsection{Sporadic groups} \label{sporadic-subsec}
We complete the general discussion from \S\ref{computations-subsection} by using Theorem~\ref{centralizerthm} to determine
Sylow-trivial modules for the Monster finite
simple group, as a computational example:
\begin{thrm} \label{monsterthm} Let $G ={\mathbb M}$ be the Monster
  sporadic group, and $k$ a field of characteristic $p$. Then 
$$\TS \cong \tuborg 0 & \mbox{for } p \leq 13\\
\Hom(N_G(S)/S,k^\times) & \mbox{for } p>13 \sluttuborg$$ 
\end{thrm}
The case $p=2$ is clear since $N_G(S) = S$, and if $p>13$, $S$ is
cyclic so the formula is standard,
Corollary~\ref{standardprop}\eqref{cyclic}, (with values tabulated in \cite[Table~5]{LM15sporadic}).
We prove the remaining cases below, which were left open in the recent
paper \cite[Table~3]{LM15sporadic}, using our formulas:
\begin{proof}[{Proof of Theorem~\ref{monsterthm} for $p=3, 5, 7, 11, 13$}]
There is some choice in methods, since several of our theorems
can be used. 
For primes $p=3,5,7, 11$ the easiest is probably to observe that
  in all cases  $\colim_{V \in
   \calF_{\!\!\calA_p^2}(G)^{\op}}^0H_1(C_G(V))_{p'} = 0$ and
  $H_1(\calF_p^*(G)) =0$, and then appeal to the centralizer decomposition
 Theorem~\ref{centralizerthm}, or more precisely the homological
 Proposition~\ref{centralizer-hom} and Theorem~\ref{main}. The
 vanishing statements will follow by a coup d'{\oe}il at the standard data about the Monster from
\cite{wilson88} and \cite{AW10} (correcting \cite{yoshiara05}), and
\cite{CCNPW85}. In fact even the $H_1$'s that
appear in centralizer colimit vanish, and $H_1(\calF_p^*(G))=0$  
  by the vanishing criterion of Corollary~\ref{vanishinglemma}. In detail:

 {\em For $p=3$}: According to \cite{CCNPW85} there are 3 conjugacy classes of
subgroups of order $3$, with the following centralizers:
$C_G(3A) = 3 \cdot Fi_{24}'$, $C_G(3B) = 3^{1+12}_+\cdot 2Suz$, $C_G(3C) = 3
\times Th$. All of these have zero $H_1(-)_{3'}$, since
$Fi_{24}'$, $2Suz$ and
$Th$ are perfect.  Hence trivially $\colim = 0$.
We want to use Corollary~\ref{vanishinglemma} to see that also $\pi_1(\calF_p^*(G)) =1$.
By \cite[Table 2]{AW10} $N_G(S) = S:(2^2 \times
SD_{16})$ a subgroup of $N_G(3A^3) = 3^{3+2+6+6}.(L_3(3) \times SD_{16})$.
 Hence $SD_{16}$ acts trivially on $3A^3$ and $2^2$ acts as the
 diagonal matrices in $\SL_3(\F_3)$ and is generated by elements that
 fix a non-trivial element in $3A^3$. We conclude by
 Corollary~\ref{vanishinglemma} that $\pi_1(\calF_p^*(G)) =1$.

{\em For $p=5$}: There are two conjugacy classes of subgroups of order
$5$ with centralizers $C_G(5A) = 5 \times HN$ and $C_G(5B) =
5^{1+6}_+:4J_2$, which have zero $H_1(-)_{5'}$ since $HN$ and $4J_2$
are perfect. Hence $\colim = 0$.
For  $\pi_1(\calF_5^*(G)) =1$, note that $N_G(S) = S:(\sym_3 \times
4^2)$ inside $N_G(5B^2) = 5^{2+2+4}.(\sym_3 \times
\GL_2(5))$. Hence $\sym_3$ acts trivially on $5B^2$, and $4^2$ is
generated by elements which act with a non-trivial fixed-point on
$5B^2$, so the conclusion again follows by Corollary~\ref{vanishinglemma}.

{\em For $p=7$}: There are 2 conjugacy classes of subgroups of order
$7$ with centralizers $C_G(7A)= 7 \times He$ and $C_G(7B)=
7^{1+4}:2\alt_7$, both with vanishing $H_1(-)_{7'}$,
so $\colim = 0$. By \cite[Thm.~7]{wilson88} and \cite[Table~1]{AW10}, $N_G(S) = S:6^2$
inside $N_G(7B^2) = 7^{2+1+2}:\GL_2(7)$, so again we can use Corollary~\ref{vanishinglemma}.

{\em For $p=11$} :
We have just one conjugacy class of subgroups of order
$11$ with $C_G(11A) = 11 \times M_{12}$, which satisfy
$H_1(C_G(11A))_{11'} =0$.
By  \cite[Table~1]{AW10}, $N_G(S) = N_G(11A^2) = 11^2:(5 \times
2\alt_5)$.  We want to see that $H_1(N_G(S)/S)$ is generated by
elements which commute with a non-trivial element in $S$, so that we
can apply Corollary~\ref{vanishinglemma}. For this we describe the
action more explicitly:
Note that $(5 \times
2\alt_5)$ is not a subgroup of $\SL_2(11)$ (by the classification
of maximal subgroups of $\PSL_2(11)$, say), so $(5 \times 2\alt_5)
\cap \SL_2(11) = 2\alt_5$. Furthermore the $5$--factor has to lie in
the center of $\GL_2(11)$, since it commutes with $2\alt_5$,
and otherwise the action of $2\alt_5$ on $11^2$ would be
reducible. In matrices we can hence write a generator of the $5$--factor as $\diag(\alpha,\alpha)$, where $\alpha$ is
a primitive $5$th root of unity in $\F_{11}^\times$. However since
$2\alt_5$ is a subgroup of  $\SL_2(11)$ and has order
divisible by $5$, it contains up to conjugacy in $\GL_2(11)$ the element
$\diag(\alpha,\alpha^{-1})$. Hence $\diag(\alpha^2,1) \in N_G(S)/S
\leq \GL_2(11)$ generates $H_1(N_G(S)/S)$ and centralizes a
non-trivial element is $S$. We conclude that $H_1(\calF_{11}^*(G))=0$ by Corollary~\ref{vanishinglemma}.

{\em For $p=13$}: We use Theorem~\ref{fusionthm}. By
\cite[Table~1]{AW10} all $p$--centric $p$--radicals are centric, so
the assumptions of the last part of that theorem are satisfied, and by
the same reference $N_G(S) = N_G(13B) = 13^{1+2}:(3 \times 4\sym_4)$. But
$\pi_1(\calF^c_{13}({\mathbb M}))=1$, since otherwise there would by
\cite[Thm.~5.4]{bcglo2} need to exist a subsystem of index $3$ or $2$,
but by \cite[Thm.~1.1]{RV04} no such subsystems exist. Hence
Theorem~\ref{fusionthm} implies that $\TS =0$. (Alternatively
apply the last part of Theorem~\ref{centralizerthm}
again; the centralizer condition holds since $C_G(13A) = 13 \times
\SL_3(3)$ and  $C_G(13B) = 13^{1+2}:{2\alt_4}$ using \cite[Table~1]{AW10}, and
the $H_1(N_G(S)/S)$ condition is satisfied since $3 \times 4\sym_4 \cong N_G(S)/S \leq \Out(S) \cong \GL_2(13)$
contains all diagonal elements in some basis, being the unique
subgroup of order $2^53^2$.) 
\end{proof}

(The remaining sporadic groups left open in
\cite[Table~5]{LM15sporadic} have subsequently also been dealt with by
David Craven \cite{craven21} using our methods, and in fact he redoes all
sporadic groups this way.)

\subsection{Finite groups of Lie type}\label{ftgrpsoflietype}
The $p$--subgroup complex of a finite group of Lie type $G$ at the
characteristic is simply connected if the Lie rank is at least $3$,
since it is homotopy equivalent to the Tits building \cite[\S3]{quillen78}.
Theorem~\ref{sequence} hence implies that there are no exotic Sylow-trivial modules
in that case, a result originally found in \cite{CMN06}. (This is also
true in rank two, by a small direct computation, using
Theorem~\ref{main} and \cite[Thm.~12.6(b)]{steinberg68}.)

Away from the characteristic the $p$--subgroup complex is also
expected to be simply
connected, if $G$ is ``large enough'', but this is not known in general.
Stronger yet, the $p$--subgroup complex appears often to be
Cohen--Macaulay  \cite[Thm.~12.4]{quillen78} \cite{das98} \cite{das00} \cite[\S9.4]{smith11}.
Partial results in this direction imply by Theorem~\ref{sequence} that there are no exotic Sylow-trivial
modules in those cases. We give two examples of this, for $\GL_n(q)$ and $\Sp_{2n}(q)$. The
$\GL_n(q)$ case also follows from very recent work of
Carlson--Mazza--Nakano, while the $\Sp_n(q)$ case is new.

\begin{thrm}[\cite{CMN14,CMN16}]\label{GLn}
Let $G = \GL_n(q)$ with $q$ prime to the characteristic $p$ of
$k$. If 
G has an elementary abelian $p$--subgroup of rank $3$, then
$\TS \cong \Hom(G,k^\times)$.
\end{thrm}

\begin{proof}
It was proved by Quillen  \cite[Thm.~12.4]{quillen78}, that
$|\calS_p(G)|$ is simply connected (and in fact Cohen-Macaulay) under the stated assumptions if $p
| q-1$, and 
when $p \nmid q-1$ it is shown in \cite[Thm.~A]{das95}. Hence the
result follows from Theorem~\ref{sequence}.
\end{proof}

\symplectic*

\begin{proof}[\thmunderline{Proof of Theorem~\ref{Spn}}] The main theorem in \cite{das98} states that
  $|\calS_p(G)|$ is simply connected under the these assumptions, so the result
  follows from Theorem~\ref{sequence}.
\end{proof}

That the $p$--rank of $G$ is at least $3$ is not enough to ensure
simple connectivity in general. At the characteristic
$\SL_2(\F_{p^r})$ or $\SL_3(\F_{p^r})$ shows this, and even away from the characteristic it
is not true as $U_4(3) = O_6^-(3)$ has $2$--rank $4$, but the $2$--subgroup complex is not
simply connected (see  \cite[Ex.~9.3.11]{smith11}). In this case
$H_1(\bO_2^*(G)) \cong H_1(G)_{2'} =0$ though, just by virtue of the Sylow
$2$--subgroup being its own normalizer. See e.g.,  \cite{aschbacher93} and
\cite[9.3]{smith11} for more on simple connectivity.

As mentioned in the introduction, in joint work in progress with Carlson, Mazza,
and Nakano we classify all Sylow-trivial modules for finite
groups of Lie type using the more detailed theorems of this paper
together with the ``$\Phi_d$-local'' approach to finite groups of Lie type.
Here we will just
do one more example,
namely $G_2(5)$ at $p=3$, which is one of the borderline cases where
$H_1(\bO_3(G))=0$, but $\calA_p(G)$ is one-dimensional and
not simply connected. It is furthermore interesting since 
Carlson--Th\'evenaz observed via computer that $\rho^2(S) \neq
\rho^3(S) = \rho^\infty(S) = N_G(S)$ in \cite{CT15}. 
As required, the assumptions of Corollary~\ref{radicalsnormal} are not
satisfied, as all $p$--subgroups turn out to be $p$--radical, but one
subgroup of order $3$ is not normal in $S$. The group can
however be easily calculated using
either Theorem~\ref{limitformula} (with $\calC$ either $\calB_3(G)$ or
$\calA_3(G)$) or Theorem~\ref{centralizerthm} (using
Proposition~\ref{extraspecial27})---we choose the first,
slightly more cumbersome, method, to see
exactly how  $\rho^2(S) \neq \rho^3(S) =
\rho^\infty(S) = N_G(S)$.

\begin{prop} \label{G_25p3}For $k$ of characteristic $3$,
$T_k(G_2(5),S) = 0$.
\end{prop}
\begin{proof}
Looking at \cite{CCNPW85}, we see that $N_G(S) \cong 3^{1+2}_+:SD_{16}$ and
that there are $2$ conjugacy classes of
elements of order $3$, where $3A$ is a central
element in $3^{1+2}_+$ and $3B$ is a non-conjugate non-trivial
element. In fact $N_G(S)$ controls $3$--fusion in $G$, as otherwise $3A$ and
$3B$ would need to be conjugate (see e.g., \cite[Lem.~4.1]{RV04}). We hence have $3$
conjugacy classes of proper, non-trivial subgroups $\langle 3A\rangle$, $\langle 3B
\rangle$, and $V = \langle 3A, 3B\rangle$ and furthermore conjugacy classes of chains in
this case coincide with chains of conjugacy classes. (The subgroups all turn
out to be $3$-radical, but $\langle 3B\rangle$ is not normal
in $S$, so the assumptions of Corollary~\ref{radicalsnormal} are not satisfied.)

We want to see that any element in
$H_1(N_G(S))_{3'}$ is equivalent to zero in the colimit. We
do this
by considering $H_1(\cdot)_{3'}$ of the
following subdiagram:
$$\xymatrix@R-5pt{
 & N_G(3B<V<S) \ar[ld] \ar[d] \ar@{=}[rd]\\
 N_G(3B<V) \ar[d]\ar[rd]& N_G(V
  < S) \ar[d] \ar[rd] & N_G(3B <S) \ar[lld] \ar[d] &  N_G(3A <S) \ar[d] \ar@{=}[ld] \\
N_G(3B) & N_G(V) & N_G(S)& N_G(3A) }$$
Note that $N_G(3B < V) =N(3B) \cap
N(3A)$, by the description of fusion in $S$.
By looking at the order of centralizers of elements and the list of
maximal subgroups, as described by \cite{CCNPW85}, we see that  $N_G(V)$ and
$N_G(S)$ are contained in $N_G(3A) \cong 3 \cdot
U_3(5):2$, and that $N_G(3B)$ is contained in the maximal subgroup $N_G(2A) \cong
2 \cdot (\alt_5 \times \alt_5).2 = 2 \cdot A(\sym_5 \times \sym_5)$, the index $2$ subgroup of
$2\sym_5 \circ 2\sym_5$ of even permutations, for $2\sym_5$ the
central extension with Sylow $2$--subgroup $Q_{16}$.
Recall also that $SD_{16} = \langle \sigma, \tau | \sigma^8 = \tau^2=1,
\tau \sigma \tau = \sigma^3\rangle \leq \Out(3^{1+2}_+) \cong
\GL_2(\F_3)$, with $\sigma^4$ identifying with the non-trivial central
element in $\Out(3^{1+2}_+)$ which acts as $-1$ on $3^{1+2}_+/
Z(3^{1+2}_+)$ and trivially on $Z(3^{1+2}_+)$. In our model of
$3^{1+2}_+$, $\sigma^4$ will hence conjugate $3B$ to $-3B$ and commute
with $3A$, and we can furthermore choose the generator $\tau$ so it
commutes with $3B$ and conjugates $3A$ to $-3A$.
Inside  $2\cdot
A(\sym_5 \times \sym_5)$ we can hence represent $3B$ as $(123)$, $3A$ as
$(1'2'3')$,  $\sigma^4$ as $(12)(4'5')$ and $\tau$ as $(45)(1'2')$.
Hence $N_G(3B) = 2 \cdot A(\sym_3 \times \sym_2
\times \sym_5)$ and $N_G(3B < V) = 2 \cdot A(\sym_3 \times \sym_2
\times \sym_3 \times \sym_2)$.

Working inside these groups the diagram identifies as follows:
\begin{equation}\label{diagram-actualgroups}
\vcenter{\xymatrix@C-10pt@R-5pt{ 
 & V:(2 \times 2) \ar[ld] \ar[d] \ar@{=}[rd]\\
  2\cdot A(\sym_3 \times \sym_2 \times (\sym_3
\times \sym_2)) \ar[d]\ar[rd]& 3^{1+2}_+:(2 \times 2)
\ar[d] \ar[rd] & V:(2\times 2) \ar[lld] \ar[d] &
3^{1+2}_+:SD_{16} \ar[d] \ar@{=}[ld] \\
2 \cdot A(\sym_3 \times \sym_2 \times \sym_5) & N_G(V) &
3^{1+2}_+:SD_{16}&  3\cdot U_3(5):2}}
\end{equation}
(In fact $N_G(V) \cong 3 \cdot \alt_{3,4}:2$, as is seen working inside $N_G(3A)$, but we shall
not need this.)
Taking $H_1(\cdot)_{3'}$ we obtain the following diagram:
\begin{equation}\label{diagram-H1}
 \vcenter{\xymatrix@R-5pt{ &  {
{\operatorname{\Z/2}}} \times
 {
{\operatorname{\Z/2}}}
  \ar[ld] \ar[d]^\simeq \ar@{=}[rd]
\\
  {
{\operatorname{\Z/2}}}  \times \Z/2 \times \Z/2
  \ar@{->>}[d] \ar[rd]&  {\underset{ \sigma^4}{\operatorname{\Z/2}}} \times
 {\underset{\tau}{\operatorname{\Z/2}}} 
\ar@{->>}[d] \ar[rd] &  {\underset{ \sigma^4}{\operatorname{\Z/2}}} \times
 {\underset{\tau}{\operatorname{\Z/2}}} \ar[lld] \ar[d] &
 \Z/2 \times \Z/2 \ar@{->>}[d] \ar@{=}[ld] \\
{
{\operatorname{\Z/2}}}  \times \Z/2 &
 {\operatorname{\Z/2}} \times {\operatorname{\Z/2}} 
 &  {\underset{\bar \sigma}{\operatorname{\Z/2}}} \times
 {\underset{\bar \tau}{\operatorname{\Z/2}}}& {\underset{\bar \tau}{\operatorname{\Z/2}}}  }}
\end{equation}
The double headed arrows are surjections by
a Frattini argument, and we conclude from this that $H_1(N_G(V))\cong
\Z/2 \times \Z/2$, since $N_G(V)/C_G(V) \cong \Z/2 \times \Z/2$, by the $G$--fusion in $V$.  
The indicated generators follow since
we know how $\sigma^4$ and $\tau$ acts on $3A$ and $3B$ (and hence
$V$). The right-hand part of the diagram immediately reveals that
$\bar \sigma$ is zero in the colimit. The element $\bar \tau$ is also
zero in the colimit: Namely, consider the element $(1'2')(4'5') \in
N_G(3B<V)$. This maps to zero in $H_1(N_G(3B))$, by the above description. But in $H_1(N_G(V))$
it maps to the same element as $\tau$, since it acts the same way on
$V$. Hence $\bar \tau$ represents the zero element in the colimit as wanted.
\end{proof}

\subsection{Symmetric groups}\label{symmetricgroups}
The Sylow-trivial modules for the symmetric groups
are understood via representation theoretic methods by the work of Carlson--Hemmer--Mazza--Nakano
\cite{CMN09,CHM10} (see also \cite{LM15schur} for extensions). Let us
point out that simple connectivity of $\calS_p(G)$ and our work directly implies their
results, at least in the
generic case.
\begin{thrm} If $p$ is odd, and  $3p+2\leq n< p^2$ or $n \geq p^2+p$,
  then $$T_k(\sym_n,S) \cong \Hom(\sym_n,k^\times) \cong \Z/2.$$
\end{thrm}
\begin{proof}
In \cite[Thm.~0.1]{ksontini04} (building on
\cite{ksontini03}\cite{bouc92}) it is proved that
$\calA_p(\sym_n)$, and hence $\calS_p(\sym_n)$, is simply connected if and only if $n$ is in the above
range, $p$ odd. So the result follows from Theorem~\ref{sequence}.
\end{proof}
When $p=2$, $T_k(\sym_n,S) =0$ since $N_{\sym_n}(S) =S$ for all $n$ \cite[Cor.~2]{weisner25}.
It is an interesting exercise to fill in the left-out
cases, where $p$ is odd and $n$ is small relative to $p$, using the methods of this
paper. Let us just quickly do this calculation where $n=2p+b$, for $p$
odd and $0<b<p$,
where $T_k(\sym_n,S) = \Z/2 \times \Z/2$, as an illustration  (one can
in fact also do the general case directly this way,
without much more effort). The
case is of interest e.g., since $n=7=2\cdot3+1$ is the smallest case where
$\calS_p(G)$ is connected but not simply connected, and where the na\"ive guess  \cite[\S5]{carlson12} that groups without strongly
$p$--embedded subgroup should have no exotic Sylow-trivial modules
fails: Pick $S = \langle (12 \cdots p),((p+1) \cdots 2p)\rangle$;
there is just one $N_G(S)$--conjugacy class of non-trivial proper
$3$--radical subgroups, represented by $A = \langle (12 \cdots p)\rangle$.
By Proposition~\ref{homformula}, $H_1(\Oep(G))$ equals the colimit
of $H_1(-)_{p'}$ applied to the 
diagram $N_G(A)  \leftarrow N_G(A < S) \rightarrow N_G(S)$, or
$$C_p \semi C_{p-1}  \times \sym_{p+b} \leftarrow   (C_p \semi
                                                  C_{p-1})^2 \times
                                                  \sym_b \rightarrow   (C_p \semi
                                                  C_{p-1}) \wr
  C_2 \times \sym_b,$$
which is seen to be $\Z/2 \times \Z/2$.
Hence by Theorem~\ref{main},
$$T_k(\sym_{2p+b},S) = \Hom(H_1(\bO_p^*(\sym_{2p+b})),k^\times) \cong \Z/2 \times \Z/2$$ as
wanted. Alternatively one may use the centralizer
decomposition, using that $|\calF_p^*(G)|$ is contractible, since it
is one-dimensional with $\pi_1(\calF_p^*(G)) =1$ by
Corollary~\ref{vanishinglemma}.

To come full circle in the above example, we can also
use Proposition~\ref{remove-minimal} (or Theorem~\ref{grouptheory-pi1} directly)  to
determine $\pi_1(\bO_p^*(\sym_{2p+b}))$ by describing the effect of adding the subgroup
$A$---this reveals that
\begin{eqnarray}
  \pi_1(\bO_p^*(\sym_{2p+b}))
&\cong&  C_{p-1}\wr C_2 \times \sym_b  / \overline{( C_{p-1} \wr
 C_2 \times \sym_b)  \cap (1 \times \alt_{p+b})} \\ &\cong& \tuborg D_8 &\mbox{ if } b =1 \\ C_2
\times C_2 & \mbox{ if }  1<b<p \sluttuborg \nonumber
\end{eqnarray}
still with $p$ odd. For $b=0$, $G_0 = \sym_p \wr C_2$, and we likewise
get $\pi_1(\bO_p^*(\sym_{2p})) \cong D_8$.

By Corollary~\ref{pprimesubgroups}, the full $\pi_1$ calculation
for $\sym_n$ enables us to do the calculation for $\alt_n$ as well. Let us dwell on this for a moment as it illustrates some of the
preceding formulas, and lets us correct a small mistake in the
literature: Notice that, in the above notation, the generator of a $C_{p-1}$ and the wreathing $C_2$
correspond to odd elements inside $\sym_{2p+b}$. Hence their
product is even and defines an element in $\ker( \pi_1(\Oep(\sym_{2p+b})) \to
\Z/2)$, which is of order $4$ when $b=0,1$. So by
Corollary~\ref{pprimesubgroups}
\begin{equation}
  \pi_1(\bO_p^*(\alt_{2p+b})) \cong \tuborg C_4 &\mbox{ if } b =0,1 \\
  C_2  & \mbox{ if }  1<b<p \sluttuborg
\end{equation}
still with $p$ odd. In particular
\begin{equation}\label{correction}
    T_k(\alt_{2p+b},S) \cong \Hom(\pi_1(\bO_p^*(\alt_{2p+b})),k^\times)
\cong \tuborg \Z/4 &\mbox{ if } b =0,1 \\
  \Z/2  & \mbox{ if }  1<b<p \sluttuborg
\end{equation}
when $p$ is odd and $k$ has a $4$th root of unity. The formula
\eqref{correction} for
$p>3$ and $b=0,1$ corrects \cite[Thm.~1.2(c)]{CMN09}, which states
$\Z/2 \oplus \Z/2$ (the mistake was that the given modules
are not all self-dual).

Notice also that Corollary~\ref{onedimcor} gives a description of an
exotic generator for, say,  $T_k(\sym_{2p+b},S)$ as
$\Omega^{-1}H_1(|\calS_p(\sym_{2p+b})|;k_\phi)$, with
$$H_1(|\calS_p(\sym_{2p+b})|;k_\phi) \cong
\ker\left(k_\phi\!\uparrow_{N_G(V< S)}^G \twoheadrightarrow
  k_\phi\!\uparrow_{N_G(V)}^G \oplus
  k_\phi\!\uparrow_{N_G(S)}^G\right),$$
and $\phi \in \Hom(\pi_1(\bO_p^*(\sym_{2p+b})),k^\times)$ not coming
from a one-dimensional representation of $\sym_{2p+b}$.
E.g., taking $n=7$
and $p=3$, $H_1(|\calS_3(\sym_7)|;k_\phi)$ is
of dimension $35 = 140-(35+70)$ (cf.\ also \cite[\S5.6]{bouc92}). This agrees on the level of dimensions with the description
in \cite[Prop.~8.3]{CMN09} as the Young module $Y(4,3)$, which is of dimension $28$ with projective
cover of dimension $28+35=63$, as
explained to us by Anne Henke. For untwisted coefficients, the
representation given by the top homology group of the $p$--subgroup complex has been studied in some generality by Shareshian--Wachs
\cite{shareshian04,SW09}.

\subsection{$p$--solvable groups} \label{psolvable} When $G$ is a $p$--solvable group,
it is proved in \cite{NR12}, building on \cite{CMT11}, that $\TS
\cong \Hom(G_0,k^\times)$, at least when $k$ is algebraically closed,
and $G = G_0$ when the $p$--rank is two or more by
\cite[Thm.~2.2]{goldschmidt70}. 
The proof in  \cite{NR12,CMT11} reduces to the case $G=AH$, where $A$ is
an elementary abelian $p$--group, and $H$ is a normal $p'$--group,
and appeals to the classification of the finite simple groups in the
proof of the last statement, albeit in a mild way.
By Theorem~\ref{main} the statement is equivalent to
the isomorphism $H^1(\Oep(G);k^\times) \cong \Hom(G_0,k^\times)$. 
We will not reprove this isomorphism here, but would like to make two remarks:
First, we have the following result.
\begin{thrm} Suppose that $G=AH$ with $H$ a normal $p'$--group and $A$ a
  non-trivial elementary abelian
  $p$--group.
Then $\TS \cong \lim^0_{V \in \calF_{\calA_p(A)}}\!\Hom(C_H(V),k^\times)$.
\end{thrm}
\begin{proof}
Note that $\calF^*_p(A) \cong \calF^*_p(G)$, and in particular $|\calF^*_p(G)|$ is contractible.
so the result follows from Theorem~\ref{centralizerthm}, also using
that $(C_H(V)
)_{p'}\xrightcong (C_G(V))_{p'}$.
\end{proof}
It would obviously be interesting to have an identification of the
right-hand side with $\Hom(G,k^\times)$, when $A$ has rank at least
$2$, via a proof which did
not use
the classification of finite simple groups.

Second, again
$\TS = \Hom(G,k^\times)$ would follow from the simply
connectivity of $\calS_p(G)$, by Theorem~\ref{sequence}. When $G=AH$ as above, Quillen
\cite[Prob.~12.3]{quillen78} conjectures that $\calS_p(G)$ should in fact be
Cohen--Macaulay, and in particular simply connected when $A$ has
$p$--rank at least $3$, and proved this when $G$ is actually solvable
\cite[Thm.~11.2(i)]{quillen78}. Aschbacher also conjectured the simple
connectivity in \cite{aschbacher93}, and reduced the claim to where
$H=F^*(G)$ is a direct product of simple components being permuted 
transitively by $A$, see also \cite[\S9.3]{smith11}---Aschbacher uses the
commuting complex $K_p(G)$, but this is $G$--homotopy equivalent to
$\calS_p(G)$ (see
e.g., \cite[p.~431]{grodal02} or \cite[\S9.3]{smith11}). 
(The related
\cite[Thm.~11.2(ii)]{quillen78}, giving non-zero homology in the
top dimension, has
in fact been generalized to the $p$--solvable case, though with a
proof also using the classification of finite simple groups; see \cite[\S8.2]{smith11}.)
We remark though, that when $G$ is not of the form $G = AH$, the complex $\calS_p(G)$
need not be simply connected even for large $p$--rank (see
\cite[Ex.~5.1]{PW00}), so a direct proof in the $p$--solvable case would need one of our
more precise theorems. (The
homology of $\calS_p(G)$ when $G$ is solvable is described in \cite[Prop.~4.2]{PW00},
and when $G$ is $p$--solvable, the $p$--essential subgroups
are those $p$--radical subgroups $P$ 
where $N_G(P)/P$ has $p$--rank one by  \cite[Cor.~to~Prop.~II.4]{puig76}.)
As noted in \S\ref{computations-subsection}, one may hope to get vanishing results for $\pi_1(\Oep(G))$ for
any finite group $G$ by reducing to simple groups, by suitably generalizing the $p$--solvable case.

\appendix
\section{Varying the collection $\calC$ of subgroups
} \label{propagating-sec}
In this appendix we describe how the homotopy and homology
of $\bO_\calC(G)$ and the other standard
categories of this paper behave under
changing the collection $\calC$. The results are often extensions of
results from mod $p$
homology decompositions \cite{dwyer97,grodal02,GS06}, but now working
integrally or away from $p$, and also examining low-dimensional behavior. The results are referred to
throughout the paper when moving to collections smaller than $\calS_p(G)$.
We start in \S\ref{Gcategories} by recalling various $G$--categories
associated to $\calC$. In \S\ref{pruning-subsec} we explain how the
homotopy type of our categories change upon removing a
given subgroup (postponing parts of the proof to \S\ref{pruningproof-subsec}).
We then use it in \S\ref{propagating-subcat} to show that the 
homotopy types agree for various collections, and in
\S\ref{pruningfundgrp-subsec} to analyze the
low-dimensional homotopy type, when more subgroups are
removed. Finally \S\ref{components-subsec} describes the set of
components of $\calC$ and its relation to group theoretic notions of
strongly $p$--embedded subgroups.

\subsection{$G$--categories associated to collections of
  subgroups} \label{Gcategories}
Recall the categories $\bO_\calC(G)$ and $\calF_\calC(G)$
from \S\ref{cat-subsec}. We now introduce ``fattened up'' versions of
$\calC$ related to these,  also used in e.g., \cite{dwyer97,grodal02}---they will be preordered sets, meaning a category with at
most one map between any two objects. Preordered sets are equivalent
as categories to 
posets, but usually not equivariantly so.
Let $E\bO_\calC$ be the category of ``pointed
$G$--sets'', i.e., the $G$--category with objects $(G/P,x)$ for $P \in \calC$, and $x \in
G/P$, and morphisms maps of pointed $G$--sets. 
The group $G$ acts on
objects by $g \cdot (G/P,x) = (G/P,gx)$. Similarly $E\A_\calC$ is the
category with objects monomorphisms $i\co P \to G$ induced by
conjugation in $G$, with $P\in \calC$,
and morphisms $i \to i'$ given by the group homomorphisms $\phi\co P
\to P'$, induced by conjugation in $G$, such that $i = i'\phi$ (see also e.g.,
\cite[2.8]{grodal02}). (As $\A_\calC$ is an older name for
$\calF_\calC$, $E\A_\calC$ would most logically be called
$E\calF_\calC$ in the current notation, but we keep the traditional
name, to avoid confusion with other standard terminology.)
By inspection of simplices, as already observed in e.g.,
\cite[Prop.~2.10]{grodal02}, 
\begin{equation}\label{EOmodG}
|E\bO_\calC|/G = |\bO_\calC(G)| \mbox{
  and } |E\A_\calC|/G = |\calF_\calC(G)|
\end{equation}
The advantage of such a description is that a $G$--homotopy
equivalence induces a homotopy equivalence on orbit spaces, by
elementary algebraic topology, which we
will use to examine which subgroups can be removed from $\calC$
without changing the homotopy type of $ |\bO_\calC(G)|$ and $ |\calF_\calC(G)|$.

Note also the $G$--equivariant functors
\begin{equation}\label{EOCEA}
  E\bO_\calC \to \calC
    \leftarrow E\A_\calC
\end{equation}
given by $(G/P,x) \mapsto G_x$ and $(i\co Q \to G) \mapsto i(Q)$. These are
equivalences of categories but not in general
$G$--equivalences (non-equivariant functors the other
way are given by $P \mapsto (G/P,e)$ and $P \mapsto (P \to G)$).
We can use the maps to $\calC$ to
describe the fixed-points, as one checks that we have equivalences of categories
\begin{equation}\label{fixed-points}
  E\bO_\calC^H \to \calC_{\geq H}
  \mbox{\,\,\,\, and \,\,\,\,}
E\A_\calC^H \to \calC_{\leq C_G(H)}.
\end{equation}
(See also \cite[{($\dagger$)}]{GS06}.)
We record the following relationship:
\begin{lemma}\label{propagatinglemma}
Let $G$ be a finite group, and consider collections of $p$--subgroups $\calC' \leq \calC$.
\begin{enumerate}[(a)]
\item \label{propa-norm} If $|\calC'| \xrightsimeq |\calC|$, then
  $|\calT_{\calC'}(G)| \xrightsimeq |\calT_{\calC}(G)|$ and $\pi_1(\bO_{\calC'}(G))_{p'}\xrightcong \pi_1(\bO_{\calC}(G))_{p'}$.
\item \label{propa-sub} If $|E\bO_{\calC'}| \to |E\bO_{\calC}|$ is a $G$--homotopy
  equivalence then $|\bO_{\calC'}(G)| \xrightsimeq
  |\bO_{\calC}(G)|$.
\item \label{propa-cent} If $|E\A_{\calC'}| \to |E\A_{\calC}|$ is a $G$--homotopy
  equivalence then $|\calF_{\calC'}(G)| \xrightsimeq
  |\calF_{\calC}(G)|$.
\end{enumerate}
\end{lemma}
\begin{proof}

For \eqref{propa-norm} consider the diagram
$$\xymatrix{ |\calT_{\calC'}(G)| \ar[d] \ar@{=}[r] &  |\calC'_G| \ar[d] &
  |\calC'|_{hG} \ar[l]_\sim \ar[d]^\sim\\
|\calT_{\calC}(G)|  \ar@{=}[r] &  |\calC_G| &
  |\calC|_{hG} \ar[l]_\sim}$$
The horizontal equivalences are by Lemma~\ref{thomasonlemma} and
the vertical equivalences follow since $|\calC'| \to |\calC|$ is a
$G$-equivariant map, assumed to be a homotopy equivalence,
and hence induces a homotopy equivalence between Borel constructions.
Now the statement about fundamental groups follow from the first
claim, using Proposition~\ref{pprime}.
Point \eqref{propa-sub} and  \eqref{propa-cent} follow from
\eqref{EOmodG}, since a $G$--homotopy equivalence
induces a homotopy equivalence on orbits.
\end{proof}

\subsection{Removing a single conjugacy class}\label{pruning-subsec}
The next result describes the effect of removing a 
conjugacy class of subgroups, called ``pruning'' in
\cite[\S9]{dwyer98sharp} (see
also \cite[Lem.~2.5]{GS06}). The result is phrased in terms of certain
homotopy pushout squares \cite[Ch.~XII]{bk} (see e.g., \cite[I.4.18]{DH01}
for an introduction). These give rise to
Meyer-Vietoris sequences in homology, and a van Kampen theorem
description for the
fundamental groups. Some squares will furthermore be homotopy pushouts of
$G$--spaces, which hence induce homotopy pushouts on both fixed-points and orbit spaces.

Denote by $(E\bO_\calC)_{>P}$ the full subcategory of pairs $(G/Q,x)$ such that there
exists a
non-isomorphism $(G/Q,x) \to (G/P,e)$, i.e.,  $G_x > P$, and
similarly for  $(E\bO_\calC)_{<P}$ and $E\A_\calC$.

\begin{prop}[Pruning collections]\label{pruninglemma}
Let $\calC'$ be a collection of
$p$--subgroups obtained from a collection $\calC$ by removing
all $G$--conjugates of a $p$--subgroup
$P \in \calC$. We have the following five $G$--homotopy pushout
squares, with corresponding quotient homotopy pushout squares:

\begin{flushleft}
  $
  \arraycolsep=3pt
\begin{array}{rr}
  \xymatrix@C-3pt@-0.5em{
  G\times_N(|\calC_{<P}| \star |\calC_{>P}|)
  \ar[r]   \ar[d]& |\calC'| \ar[d]
 \ar@<-7ex>@{}[d]_{(1a)}
  \\  G/N \ar[r] & |{\calC}|}
&
\xymatrix@C-3pt@-0.5em{ (|\calC_{<P}|/P \star |\calC_{>P}|)/W
\ar[r]   \ar[d] &
                  |\calC'|/G \ar[d]
   \ar@<-7ex>@{}[d]_{(1b)}
 \\ 
  pt \ar[r] & |{\calC}|/G}
              \vspace{7pt}\\
   \xymatrix@C-3pt@-0.5em{ G \times_N(EN \times (|\calC_{<P}| \star
  |\calC_{>P}|))
\ar[r]   \ar[d] &
                  EG \times |{\calC'}| \ar[d]
 \ar@<-7ex>@{}[d]_{(2a)}
  \\ 
G\times_NEN \ar[r] & EG \times |{\calC}|} 
&
 \xymatrix@C-3pt@-0.5em{ (|\calC_{<P}| \star |\calC_{>P}|)_{hN}
\ar[r]   \ar[d] &
                  |\calT_{\calC'}(G)| \ar[d]
   \ar@<-7ex>@{}[d]_{(2b)}
\\ 
  BN \ar[r] & |{\calT_\calC(G)}|}
              \vspace{7pt}
  \\
\end{array}$
\end{flushleft}

\begin{flushleft} 
$\xymatrix@C-3pt@-0.5em{G \times_N (EW \times (|{(E\bO_\calC)}_{<P}| \star
  |\calC_{>P}|))
\ar[r]   \ar[d] &
                  |E\bO_{\calC'}| \ar[d]
   \ar@<-7ex>@{}[d]_{(3a)}
 \\ 
G \times_N EW \ar[r] & |E\bO_{\calC}|}
$
\end{flushleft}

\begin{flushright}
$\xymatrix@C-3pt@-0.5em{ (|{(E\bO_\calC)}_{<P}|/P \star |\calC_{>P}|)_{hW}
\ar[r]   \ar[d] &
  |\bO_{\calC'}(G)| \ar[d]  \ar@<-7ex>@{}[d]_{(3b)}
 \\ 
  BW \ar[r] & |\bO_{\calC}(G)|}
$
\end{flushright}

$
\xymatrix@C-3pt@-0.5em{ G \times_N(EW' \times (|\calC_{<P}| \star
  |(E\A_{\calC})_{>P}|))
\ar[r]   \ar[d] &
                  |E\A_{\calC'}| \ar[d]
                  \ar@<-7ex>@{}[d]_{(4a)}
                  \\  G\times_NEW' \ar[r] & |E\A_{\calC}|}
                                            $
\begin{flushright}                                       
$\xymatrix@C-3pt@-0.5em { (|\calC_{<P}| \star |(E\A_{\calC})_{>P}|/C)_{hW'}
\ar[r]   \ar[d] &
                  |\calF_{\calC'}(G)| \ar[d]
   \ar@<-7ex>@{}[d]_{(4b)}
 \\ 
 BW' \ar[r] & |\calF_{\calC}(G)|}$
\end{flushright}
where $C =C_G(P)$, $N=N_G(P)$, $W = N/P$, $W' =N/C$,
and $\star$ denotes join of spaces.
\end{prop}

The proof of Proposition~\ref{pruninglemma} is not hard, but is
postponed to \S\ref{pruningproof-subsec} to not interrupt the
flow. We move on to draw
consequences, starting with Symonds' theorem,  `Webb's
conjecture' \cite[Conj.~4.2]{webb87survey}, on the contractibility of the orbit
space of the $p$--subgroup complex, used in the proofs of
Theorems~\ref{limitformula} and \ref{carlsonthevenazconj}. To get
minimal assumptions on the collection $\calC$, rather than just
closed under passage to $p$--overgroups, we make a forward reference
Lemma~\ref{overgroupslemma}, allowing us to remove non-$p$--radical
subgroups.
\begin{prop}[\cite{symonds98}]\label{webbsconj} If $G$ is a finite group and  $\calC$ a non-empty collection of $p$--subgroups, closed
  under passage to $p$--radical overgroups, then $|\calC|/G$ is contractible.
\end{prop}
\begin{proof}
By  Lemma~\ref{overgroupslemma} $|\calC|$ is
$G$--homotopy equivalent to $|\bar \calC|$ where $\bar \calC$ is obtained by adding
all $p$--overgroups of subgroups in $\calC$. We may thus without restriction assume that $\calC$ is closed under passage
to all $p$--overgroups, as $G$--homotopy equivalences induce homotopy
equivalences on $G$--orbit spaces.

If $\calC$ consists of just all Sylow (i.e., maximal) $p$--subgroups, then the claim
amounts exactly to the part of Sylow's theorem saying that these
subgroups are all conjugate. We will prove the claim in general by
seeing that contractibility is preserved under adding conjugacy
classes of subgroups in order of decreasing size, using
Proposition~\ref{pruninglemma}(1b) by an induction on the size of the group:
Suppose the claim is true for all groups $G$ of strictly smaller
order, and that $\calC$ is obtained from $\calC'$ by adding $G$--conjugates of a subgroup $P$ this way.
We can also assume that $P$ is
non-trivial, since otherwise the claim is clear, as $\calC$ is
equivariantly contractible to the trivial subgroup in that case.
 
Now consider the pushout square of Proposition~\ref{pruninglemma}(1b):
The top left-hand corner becomes $|\calC_{>P}|/W$, for $W
=N_G(P)/P$. But  $|\calC_{>P}|$ has a
$W$--equivariant deformation retraction onto $|\calS_p(W)|$  via $Q
\mapsto N_Q(P)/P$, as observed by Quillen \cite[Prop.~6.1]{quillen78}
(see also Lemma~\ref{overgroupslemma}). But $|\calS_p(W)|/W$ is contractible by
our induction hypothesis, observing that $W$ is of smaller order since $P$ is
non-trivial and that $\calS_p(W)$ is non-empty as we have already added
the Sylow $p$--subgroups. Hence $|\calC'|/G \to |\calC|/G$ is a
homotopy equivalence, and we conclude that the claim is true for all 
collections $\calC$ as in the proposition, proving the claim for $G$.
\end{proof}

Let us spell out what
Proposition~\ref{pruninglemma}(3b) says when the subgroup $P$ we remove is
minimal.
\begin{prop}[The effect of adding or removing a minimal subgroup] \label{remove-minimal} Suppose that $\calC$ is a collection of $p$--subgroups in a finite
  group $G$, closed under
  passage to $p$--radical overgroups, and that $\calC'$
  is obtained from $\calC$ by removing the conjugacy class of a minimal $p$--subgroup
  $P \in \calC$. Set $W = N_G(P)/P$. Then $|\bO_\calC(G)|$ can be described via a homotopy pushout square
$$ \xymatrix{ |\Tep(W)| \ar[r] \ar[d]& |\bO_{\calC'}(G)| \ar[d]\\
BW \ar[r] & |\bO_\calC(G)|  }$$
Here the top horizontal map identifies with the nerve of the composite
$\Tep(W) \to \Oep(W) \to \bO_{\calC'}(G)$ where the first map is the
natural one from
\S\ref{cat-subsec} and the second sends a $W$--set $X$ to $G/P \times_W X$ (replacing $\calC$ by its
closure under taking $p$-overgroups if necessary, using
Lemma~\ref{overgroupslemma}\eqref{claim-radical}).
In particular 
\begin{enumerate}
\item\label{po1} $\pi_1(\bO_\calC(G)) \xleftcong
  \pi_1(\bO_{\calC'}(G))/\overline{\im(K)}$,
with $K =  \ker( \pi_1(\Oep(W)) \to W_{p'})$ and the overline denoting
normal closure.

\item\label{po2} $H_1(\bO_\calC(G))\mkern-1mu \xleftcong \mkern-1mu H_1(\bO_{\calC'}(G))/\mkern-1mu \im(K)$,
with $K =  \ker( H_1(\Oep(W)) \mkern-1mu\to \mkern-1mu H_1(W))$.
\end{enumerate}
Hence
\begin{enumerate}\addtocounter{enumi}{2}
\item\label{po3} If  $\pi_1(\Oep(W))\xrightcong
W_{p'}$ then $\pi_1(\bO_{\calC'}(G))\xrightcong \pi_1(\bO_{\calC}(G))$.
\item \label{po4} If  $H_1(\Oep(W))\xrightcong
H_1(W)_{p'}$ then $H_1(\bO_{\calC'}(G))\xrightcong H_1(\bO_{\calC}(G))$.
\end{enumerate}
\end{prop}
\begin{proof}
We claim that square in the position identifies with the homotopy
pushout square of
Proposition~\ref{pruninglemma}(3b), after making suitable
identifications. Namely the
left-hand corner in (3b) identifies with $|\Tep(W)|$ 
via the homotopy equivalences $|\calC_{>P}|_{hW} \simeq
|\calS_p(W)|_{hW} \simeq |\Tep(W)|$, using again
Lemma~\ref{overgroupslemma} and \eqref{grothendieck}, and the maps
identify with the stated ones via this equivalence.

Now \eqref{po1} is a consequence of van Kampen's
theorem \cite[\S1.2]{hatcher02}: By Proposition~\ref{boundslemma} both fundamental groups on
the right-hand side of the pushout square are finite $p'$--groups, so
van Kampen's theorem and \eqref{pprimequotient} produces a pushout of
groups
\begin{equation}
\vcenter{\xymatrix{ \pi_1(\Oep(W)) \ar[r] \ar@{->>}[d] & \pi_1(\bO_{\calC'}(G))
  \ar@{->>}[d] \\
W_{p'} \ar[r] & \pi_1(\bO_{\calC}(G))}}
\end{equation}
with vertical maps surjective, and \eqref{po1} follows. Point \eqref{po2} follows similarly, but using the
Mayer--Vietoris sequence instead. Points \eqref{po3} and \eqref{po4} are now
obvious. \end{proof}

For further applications of Proposition~\ref{pruninglemma}, we recall
the behavior of connectivity under joins. In general the
  join of an $n$--connected space with an $m$--connected
  space is $(n+m+2)$--connected (use cellular approximation \cite[\S4.1]{hatcher02}).  The cases we will need are
  summarized in the following lemma.
\begin{lemma}[Connectivity of joins]\label{connectivitylemma}
The join $X \star Y$ is connected for $X,Y$
  non-empty or $X$ connected. It is simply connected for $X$ connected and
  $Y$ non-empty, or for $X$ simply connected.
  
  The join $X \star Y$ is  ($G$--) contractible if $X$ or $Y$ is ($G$--) contractible.

\end{lemma}
\begin{proof}
    This follows from the definition of the join, e.g., as the homotopy pushout
of the diagram $X \leftarrow X \times Y \to Y$ (which identifies with the suspension of the smash  $\Sigma(X \wedge Y)$ if
picking a basepoint).\end{proof}

We finally give the postponed lemma about removing subgroups, which
will also be used several times in proofs of the subsequent theorems.

\begin{lemma} \label{overgroupslemma}
Let $\calC$ be a collection of $p$--subgroups in $G$, and let $\bar
\calC$ be the smallest collection containing $\calC$,
closed under passage to $p$--overgroups, and let $P$ be a
$p$--subgroup of $G$.

\begin{enumerate}
  \item \label{claim-radical}    If $\calC$ contains all $p$--radical groups in ${\bar \calC}_{>P}$,
    then $\calC_{>P} \to \bar \calC_{>P}$ is a $N_G(P)$--homotopy equivalence.
    Thus, if all $p$--radical overgroups 
    of $P$ are in $\calC$ then we have $N_G(P)$--homotopy equivalences $$|\calC_{>P}|
    \xrightsimeq |\calS_p(G)_{>P}| \xleftsimeq  |\calS_p(N_G(P)/P)| .$$
      \item \label{claim-essential}
    If $\calC$ contains all Sylow and $p$--essential overgroups in ${\bar \calC}_{>P}$,
    then $\pi_0(\calC_{>P}) \xrightcong \pi_0(\bar \calC_{>P})$.
    Thus, if all  Sylow and $p$--essential overgroups 
    of $P$ are in $\calC$ then we have $N_G(P)$--equivariant bijections $$\pi_0(\calC_{>P})
    \xrightcong \pi_0(\calS_p(G)_{>P}) \xleftcong  \pi_0(\calS_p(N_G(P)/P)) .$$
 \end{enumerate}
\end{lemma}

\begin{proof}

Start by recalling that $\calS_p(G)_{>P}$ is indeed $N_G(P)$--homotopy equivalent to
$\calS_p(N_G(P)/P)$ via the equivariant deformation retraction  $R
\mapsto N_R(P)/P$, as already observed by Quillen
\cite[Prop.~6.1]{quillen78}.

Now to see the first claim in \eqref{claim-radical}, we add $N_G(P)$--conjugacy classes of
$p$--subgroups to $\calC_{>P}$ in order of decreasing size to reach
$\bar \calC_{>P}$ (see also \cite[Pf.~of~Thm.~1.2]{grodal02}).
If 
$\tilde \calC$ is obtained from $\calC$ by adding a
 $N_G(P)$--conjugacy class of a $p$--subgroup $Q$, maximal in $\bar \calC \setminus
 \calC$, then we have a $N_G(Q)$--homotopy
pushout square analogous to Proposition~\ref{pruninglemma}(1a):
\begin{equation}\label{modified-pushout}
\vcenter{\xymatrix{
N_G(P) \times_{N_G(P <Q)} (|\calS_p(G)_{>Q}| \star |\calC_{>P,<Q}|)\ar[r] \ar[d] & |\calC_{>P}|
\ar[d] \\
N_G(P)/N_G(P<Q) \ar[r] & |\tilde \calC_{>P}|}}
\end{equation}
By assumption $Q$ is not $p$--radical, so that $\calS_p(G)_{>Q}$ is
$N_G(Q)$--contractible, via the standard contraction 
\begin{equation}\label{stdcontraction}
R \geq N_R(Q)
\leq N_R(Q)O_p(N_G(Q)) \geq O_p(N_G(Q))
\end{equation} of Quillen \cite{quillen78}
and Bouc \cite{bouc84}. Hence $|\calS_p(G)_{>Q}| \star |\calC_{>P,<Q}|$
is $N_G(P<Q)$--contractible by the initial observation and Lemma~\ref{connectivitylemma}.
Hence \eqref{modified-pushout} shows that $|\calC_{>P}| \to |\tilde
\calC_{>P}|$ is a $N_G(P)$--homotopy equivalence. Continuing this way
shows the first part of \eqref{claim-radical} by induction. The second
half is follows from the first, but applied to $\calC \cup [P]$ so
that $(\overline { \calC \cup [P]})_{>P} = \calS_p(G)_{>P}$.

The proof of \eqref{claim-essential} follows similarly. 
The assumption that $\calS_p(N_G(Q)/Q)$ is
connected still ensures that $|\calS_p(G)_{>Q}| \star
|\calC_{>P,<Q}|$ is connected by Lemma~\ref{connectivitylemma}.
Hence $|\calC_{>P}| \to |\tilde \calC_{>P}|$ is a bijection on $\pi_0$
by \eqref{modified-pushout}, and we again conclude
$\pi_0(\calC_{>P}) \xrightcong \pi_0(\bar \calC_{>P})$ by
induction. Again the second part follows by applying the first to  $\calC \cup [P]$.
\end{proof}

\begin{rem} \label{rem-higher-poset} Lemma~\ref{overgroupslemma} also has a
  generalization to higher homotopy groups: If $\calC$ is closed under
  passage to all overgroups $Q$ such that $|\calS_p(N_G(Q)/Q)|$ is not
  $i$--connected, then $\iota$ is an isomorphism on $\pi_j$ for $j
  \leq i$ and surjective on $\pi_{i+1}$. This follows by the same
  argument as above, but now in
  \eqref{modified-pushout} appealing to the 
Blakers-Massey's excision theorem (see
e.g., \cite[Prop.~6.4.2]{tomdieck08}) instead.
\end{rem}

\subsection{Varying the collection without changing homotopy types}\label{propagating-subcat}
We now see how we can vary our collection $\calC$ without changing the homotopy type of associated
categories. The omnibus Theorem~\ref{propagating} below is mainly a
translation of  ``classical'' homotopy equivalences
\cite[Thm.~1.1]{GS06}, using the elementary
Lemma~\ref{propagatinglemma} (see also
\cite{GS06} for historical references).

As usual let $\calS_p(G)$ and
$\calA_p(G)$ denote non-trivial $p$--groups and non-trivial
elementary abelian $p$-subgroups $V \cong (\Z/p)^r$ respectively.  The collection $\calB_p(G)$ is
the collection of non-trivial $p$--radical subgroups, i.e.,
non-trivial $p$--subgroups
$P$ such that $O_p(N_G(P)) = P$.

Here and elsewhere the superscript $e$ means that we do not
exclude the trivial subgroup.

\begin{thrm}
\label{propagating} Let $G$ be a
  finite group and $\calC$ a collection of
  $p$--subgroups.
\begin{enumerate}
\item\label{reduce-to-rad}
Suppose $\calC$ is closed under passage to $p$--radical overgroups. 
Then $|\calC \cap \calB_p^e(G)| \xrightsimeq |\calC|$ and $|E\bO_{\calC\cap
  \calB_p^e(G)}(G)| \xrightsimeq |E\bO_\calC|$ are $G$--homotopy equivalences
and consequently
$$ |\calT_{\calC \cap \calB_p^e(G)}(G)| \xrightsimeq
|\calT_{\calC}(G)|   \mbox{\,\,\,\, and \,\,\,\, }  |\bO_{\calC \cap \calB_p^e(G)}(G)| \xrightsimeq
|\bO_{\calC}(G)| $$

\item \label{reduce-to-eltab} Suppose $\calC$ is closed under passage to
non-trivial elementary abelian subgroups.
Then 
$$|\calC \cap \calA_p^e(G)| \xrightsimeq  |\calC| \mbox{\,\,\,\, and
  \,\,\,\, }
|E\A_{\calC \cap \calA_p^e(G)}| \xrightsimeq |E\A_{\calC}|
$$
are $G$--homotopy equivalences. 
Consequently
$$|\calT_{\calC \cap \calA_p^e(G)} (G)| \xrightsimeq |\calT_{\calC}(G)|
 \mbox{\,\,\,\, and \,\,\,\, }
|\calF_{\calC \cap \calA_p^e(G)}(G)|  \xrightsimeq |\calF_{\calC}(G)|.
$$
\end{enumerate}
\end{thrm}

\begin{proof}
Point \eqref{reduce-to-rad}:
  To see that $|\calC \cap \calB_p^e(G)| \xrightsimeq |\calC|$ is a
$G$--homotopy equivalence, 
note that for an arbitrary $P \in \calC \setminus
\calB_p^e(G)$, the space  $|\calC_{>P}| \star |\calC_{<P}|$ is $N_G(P)$--contractible
by Lemma~\ref{overgroupslemma}\eqref{claim-radical} and
Lemma~\ref{connectivitylemma}, as $|\calS_p(N_G(P)/P)|$ is since $P$
is not $p$--radical. Hence the
pushout square in
Proposition~\ref{pruninglemma}(1a) shows that we can remove the subgroups
in $\calC \setminus \calB_p^e(G)$ one $G$--conjugacy class
at a time (in some arbitrary order) without changing the $G$--homotopy
type.
That  $|E\bO_{\calC\cap
  \calB_p^e(G)}(G)| \xrightsimeq |E\bO_\calC|$  is a $G$--homotopy
equivalence follows from the analogous argument, but now using
 Proposition~\ref{pruninglemma}(3a) instead.
The homotopy equivalences on $\calT$ and $\bO$ follow from this and
Lemma~\ref{propagatinglemma}\eqref{propa-norm}\eqref{propa-sub}, finishing the proof
of \eqref{reduce-to-rad}. (For the last bit, one could also use
Proposition~\ref{pruninglemma}(2b)(3b) directly.)

Point \eqref{reduce-to-eltab}: 
We want to prove this by comparing $\calC$ and $\calC \cap
\calA_p^e(G)$ to the smallest collection $\bar \calC$ containing
$\calC$ and closed under passage to
non-trivial $p$--subgroups. We do this by adding $G$--conjugacy classes
of subgroups to $\calC$ in order of increasing size, and observe that this does
not change the $G$--homotopy types of $|\calC|$ and $E\A_\calC$, and
correspondingly for  $\calC \cap
\calA_p^e(G)$:

Let $\tilde \calC = \calC
\cup [P]$, where $P \in \bar \calC \setminus \calC$ is minimal. Then
Proposition~\ref{pruninglemma}(1a)(4a) shows that $|\calC| \to |\tilde \calC|$ and $|E\A_\calC| \to
|E\A_{\tilde \calC}|$ are $G$--homotopy equivalences, once we see
that $|\calC_{<P}|$ is $N_G(P)$-equivariantly contractible (also using
Lemma~\ref{connectivitylemma}).
If $\calC$
contains the trivial subgroup then  $\calC_{<P} = \calS^e_p(G)_{<P}$
and the claim is obvious, so
assume that this is not the case, i.e., $\calC_{<P} = \calS_p(G)_{<P}$. This
is still $N_G(P)$--equivariantly contractible, as is seen using the
standard contraction of
Quillen \begin{equation}\label{std-contraction} Q
  \leq Q \Phi(P) \geq \Phi(P),
\end{equation}
  where $\Phi(P)$ is the Frattini
subgroup, generated by commutators and $p$th powers and $Q\Phi(P) < P$
since $P$ is not elementary abelian.

By continuing to add subgroups this way we
show that $|\calC| \to |\bar \calC|$ and $|E\A_\calC| \to |E\A_{\bar
  \calC}|$ are $G$-homotopy equivalences. The proof that  
 $|\calC \cap
\calA_p^e(G)| \to |\bar \calC|$ and $|E\A_{\calC \cap
\calA_p^e(G)}| \to |E\A_{\bar
  \calC}|$ are $G$-homotopy equivalences is identical.
The statements about $\calT$ and $\calF$ now
follow from Lemma~\ref{propagatinglemma}\eqref{propa-norm}\eqref{propa-cent}.
\end{proof}

\begin{rem}
Theorem~\ref{propagating}\eqref{reduce-to-rad} does
not hold for $\calF$, since $|\calF_{\calB_p(S)}(S)| \cong B
(S/Z(S))$ (compare also Proposition~\ref{fusion-boundslemma}).
Theorem~\ref{propagating}\eqref{reduce-to-eltab} is not true for
for $\bO$, since $|\bO_{\calA_2(C_4)}(C_4)| \cong B\Z/2$.
See also \cite {JM12} for results with $\bar \calF$.

\end{rem}

\subsection{Varying the collection without changing the low dimensional
  homotopy type}\label{pruningfundgrp-subsec}
In this subsection we continue to apply the  formulas of \S\ref{pruning-subsec} to remove subgroups, but now focusing
on the first homology group and the fundamental group.

\subsubsection{Models for
$H_1(\bO_\calC(G))$, $\pi_1(\bO_\calC(G))$ and  $\pi_1(\calT_\calC(G))$}

\begin{thrm}[Propagating fundamental groups, $p$--overgroup-closed
  version]\label{remove-non-essential}
  Let $\calG_p(G)$ denote the collection of $p$--subgroups $P$ that are either Sylow,
  $p$--essential, or satisfy that  $\psi_P: H_1(\Oep(N_G(P)/P)) \to
  H_1(N_G(P)/P)_{p'}$  is not an isomorphism.
  For any collection $\calC$ of $p$--subgroups, closed under
  passage to $p$--overgroups in $\calG_p(G)$,   $H_1(\bO_{\calC
    \cap \calG_p(G)}(G)) \xrightcong  H_1(\bO_\calC(G))$.

Likewise 
 $\pi_1(\bO_{\calC\cap\calG'_p(G)}(G)) \xrightcong \pi_1(\bO_\calC(G))$ and
 $\pi_1(\calT_{\calC \cap \calG''_p(G)}(G)) \xrightcong
 \pi_1(\calT_\calC(G))$, where the collections $\calG'_p(G)$ and
 $\calG''_p(G)$ are defined
analogously, replacing 
 $\psi_P$ by $\psi'_P\co\! \pi_1(\Oep(N_G(P)/P))\mkern-1mu \to\mkern-1mu
 (N_G(P)/P)_{p'}$ and $\psi''_P\co\! \pi_1(\Tep(N_G(P)/P))\mkern-1mu \to\mkern-1mu
  N_G(P)/P$, respectively. By construction $\calG_p(G) \subseteq \calG'_p(G) \subseteq
  \calG_p''(G)$.
\end{thrm} 
We remark that $\pi_1(\Tep(W)) \xrightcong W$ if and only if
$\calS_p(W)$ is simply connected by \eqref{fundseq}. The proof of Theorem~\ref{remove-non-essential} needs a
lemma, which can also be used to remove additional subgroups in it.
\begin{lemma} \label{lem-add-over} Suppose that $\calC$ is a collection of $p$--subgroups closed under
  passage to $p$--essential and Sylow $p$--overgroups, and let $\bar
  \calC = \{ P \in \calS_p^e(G) | \mbox{ there exists } Q \leq P
  \mbox{ with } Q \in\calC\}$. Then $\pi_1(\calT_{\calC}(G)) \xrightcong
  \pi_1(\calT_{\bar \calC}(G))$ and $\pi_1(\bO_{\calC}(G)) \xrightcong
  \pi_1(\bO_{\bar \calC}(G))$.
      \end{lemma}
      \begin{proof}

Suppose $P \in \bar \calC \setminus \calC$. Then $\calC_{>P}$ is
connected by Lemma~\ref{overgroupslemma}\eqref{claim-essential} and
$\calC_{<P}$ is non-empty. Thus $|\calC_{<P}| \star |\calC_{>P}|$ and
$|{(E\bO_\calC)}_{<P}|/P \star |\calC_{>P}|$ are simply
connected by Lemma~\ref{connectivitylemma}.
Proposition~\ref{pruninglemma}(2b)(3b), together with van Kampen's
theorem, show that we can add the conjugacy class of $P$ to $\calC$
without changing $\pi_1(\calT_{\calC}(G))$ or
$\pi_1(\bO_{\calC}(G))$. Repeating this argument, now with $\calC \cup
[P]$ etc., allows us to add all the subgroups of  $\bar \calC
\setminus \calC$ as wanted.
      \end{proof}

      \begin{proof}[Proof of Theorem~\ref{remove-non-essential}]
        By Lemma~\ref{lem-add-over} we can just as well compare groups
        relative to the two collections $\overline{\calC
    \cap \calG_p(G)}  \subseteq  \overline{\calC}$. However all the
  subgroups in  $\overline{\calC} \setminus \overline{\calC
    \cap \calG_p(G)} $ can now be removed, in order of increasing size
  using Proposition~\ref{remove-minimal}\eqref{po4}.

  The statement for $\calG'_p(G)$ follows by the same argument, but
  now appealing to Proposition~\ref{remove-minimal}\eqref{po3}. The
  statement for  $\calG''_p(G)$ also follow from the same line of
  argument: As remarked after the statement of
  Theorem~\ref{remove-non-essential},  $P \not \in
  \calG''_p(G)$ means that $|\calS_p(G)_{>P}|$ is simply connected.
  Hence Proposition~\ref{pruninglemma}(2b), together with van Kampen's
  theorem, still allows us to remove subgroups in $\overline{\calC} \setminus \overline{\calC
    \cap \calG_p(G)} $ without changing the fundamental group of $\calT$.
\end{proof}

\begin{example} \label{G25-counterex}
 For finite groups of Lie type in characteristic $p$, $\calG_p(G) =
 \calG'_p(G) =  \{\mbox{Sylow}\} \cup \{p\mbox{--ess.}\}$ and are
 exactly the unipotent radicals of parabolic
 subgroups of rank at most
 one, whereas $\calG''_p(G)$ identify with unipotent radicals of parabolic
 subgroups of rank at most 2 (see \S\ref{ftgrpsoflietype}).

For $G=G_2(5)$ at
$p=3$, discussed in Proposition~\ref{G_25p3}, the subgroup
$\langle3A\rangle$ is in $\calG_p(G)$ but not  $p$--essential. And for
$\calC$ the collection
of non-trivial $p$--subgroups except $\langle 3A\rangle$,
$\pi_1(\bO_\calC(G)) \cong C_2$ whereas $\pi_1(\Oep(G)) =1$, so
$\langle3A\rangle$ is indeed necessary to control endotrivial modules.
\end{example}

\begin{rem}
The collections from
Theorem~\ref{remove-non-essential} fit in a hierarchy 
involving higher homotopy groups. 
Note that  $\pi_1(\Tep(N_G(P)/P)) \xrightncong N_G(P)/P$ if and
only if
$|\calS_p(N_G(P)/P)|$ is not simply connected. Say that $P$ is
``$\pi_i$--essential'' if $|\calS_p(N_G(P)/P)|$ is not
$i$--connected. Then $\pi_{-1}$--essential is Sylow,  $\pi_0$--essential is Sylow
or $p$--essential, $\pi_1$--essential means in $\calG''_p(G)$, and we
have inclusions
\begin{multline}\label{inclusions} \{\mbox{Sylow}\} \subseteq
  \{\mbox{Sylow}\} \cup \{p\mbox{--ess.} \} 
 \subseteq \calG_p(G) \subseteq  \calG'_p(G) \subseteq \calG''_p(G) = \{\mbox{$\pi_{1}$--ess.}\} \\
\subseteq
\{\mbox{$\pi_{2}$--ess.}\}
 \cdots \subseteq \calB_p^e(G) 
\end{multline}
It follows from Remark~\ref{rem-higher-poset} that the $\pi_i$-essential subgroups are the ones needed to
describe the $i$--truncation of the homotopy type of $|\calS_p(G)|$, and
hence the $i$--truncation of $|\Tep(G)|$ and
$|\Oep(G)|$. As an 
$n$--connected space of dimension $n$ is contractible (see e.g., \cite[Exc.~4.12]{hatcher02}) the filtration in
\eqref{inclusions} is finite, and furthermore
Quillen's famous conjecture
\cite[Conj.~2.9]{quillen78} predicts that  $p$--radical implies
$\pi_n$--essential, for $n$, say, the dimension of $\calS_p(G)$ (see
e.g., \cite{AS93} for known cases).
As mentioned in
Sections~\ref{ftgrpsoflietype} and \ref{psolvable},
several open questions
about simply connectivity of $\calS_p(G)$ may be shadows of stronger
statements about Cohen--Macaulayness, also justifying an
interest in $\pi_i$--essential subgroups. \end{rem}

\subsubsection{Models for $\pi_1(\calF_\calC(G))$}\label{modelsF}
A non--Sylow $p$--subgroup $P$
is called {\em $\calF$--essential} if
$W_0PC_G(P)/P$ is a proper subgroup of $W =N_G(P)/P$ (with $W_0$ as
in \eqref{pembeddeddef}). It is obviously
a subcollection of the $p$--essential subgroups and were introduced
in \cite{puig76} (as ``C-essential''). 
The $\calF$--essential
subgroups can also be described as the $p$--centric subgroups
such that $N_G(P)/PC_G(P)$ contains a strongly $p$--embedded subgroup
(see Lemma~\ref{lem:Fessentialdef}). (Beware that some fusion litterature
such as \cite[Def.~I.3.2]{AKO11}, but not \cite[\S5]{puig06},
take ``fully $\calF$--nomalized'' in $S$ as part of the definition, giving a
smaller, but non-conjugacy invariant set, as it refers to a fixed
Sylow $p$--subgroup $S$.)

\begin{prop} \label{remove-pione-F}
 Suppose that $\calC$ is closed under
  passage to $\calF$--essential and Sylow $p$--overgroups, and that $\calC'$
 is a collection obtained from $\calC$ by removing conjugacy classes of
  $p$--subgroups, which are neither Sylow $p$--subgroups, 
$\calF$--essential, or minimal in $\calC$. Then  $\pi_1(\calF_{\calC'}(G))\xrightcong \pi_1(\calF_{\calC}(G))$.
\end{prop}

\begin{proof}
It is enough to show that for both $\calF_\calC(G)$ and
$\calF_{\calC'}(G)$ we get the same fundamental group as $\calF_{\bar
  \calC}(G)$, for $\bar
\calC$ the minimal collection containing $\calC$ and closed under all
$p$--overgroups. And for that, it is enough to show that the
fundamental group of $\calF_{\bar \calC}(G)$ does not change when removing a $p$--subgroup $P$
from $\bar \calC$ which is neither minimal nor $\calF$--essential nor a Sylow
$p$--subgroup (by removing subgroups in order of increasing size).
This will be a consequence of Proposition~\ref{pruninglemma}(4b), if we see that
$|(E\A_{\bar \calC})_{>P}|/C_G(P)$ is connected. For this note, as in
\eqref{EOCEA}, that $|(E\A_{\bar \calC})_{>P}| \to |\bar \calC_{>P}|$
is an $N_G(P)$--equivariant map, which is a bijection on components, and
hence $|(E\A_{\bar \calC})_{>P}|/C_G(P) \to |{\bar \calC}_{>P}|/C_G(P)$ is also a
bijection on components.  By assumption $|{\bar \calC}_{>P}| =
|\calS_p(G)_{>P}|$ which is $N_G(P)$--homotopy equivalent to $|\calS_p(N_G(P)/P)|$ by
Lemma~\ref{overgroupslemma}. By \eqref{redtoG_0},
$\pi_0(|\calS_p(N_G(P)/P)|) \cong W/W_0$, in the notation from above
the proposition, and $W = W_0C_G(P)P/P$, as $P$ is non-$\calF$--essential. Putting this together we see that $|(E\A_{\bar \calC})_{>P}|/C_G(P)$ is connected as wanted.
\end{proof}

\subsubsection{Fundamental groups for subgroup-closed collections}
We end with the dual case:

\begin{thrm}[Propagating fundamental groups, subgroup-closed version]\label{remove-big-subgroups} Let $\calC$ be a collection
  of $p$--subgroups closed under passage to non-trivial
  elementary abelian
  subgroups.
Then $\pi_1(\calT_{\calC'}(G))\xrightcong
  \pi_1(\calT_{\calC}(G))$,  $\pi_1(\bO_{\calC'}(G))_{p'}\xrightcong
  \pi_1(\bO_{\calC}(G))_{p'}$, and 
$\pi_1(\calF_{\calC'}(G))\xrightcong
  \pi_1(\calF_{\calC}(G))$ for $\calC'$ the subgroups in $\calC$
  which are either maximal, or elementary abelian of rank at most two.
The same statements hold taking $\calC'$ the elementary abelian subgroups of $\calC$
of rank at most $3$. 
\end{thrm}

\begin{proof} First, the
  claim for $\bO$ follows from $\calT$ by Proposition~\ref{pprime}.
Second, the claims are obvious if
the trivial subgroup is in $\calC$, so we can assume that this is not the case.
  Now to prove the
  claim for $\calT$ and $\calF$, it is, as usual, enough to compare $\calC$ and $\calC'$
  to a collection $\bar \calC$ obtained from $\calC$ by adding non-elementary
  abelian $p$--groups, to make it closed under passage to all non-trivial
  $p$--subgroups. By removing subgroups from $\bar \calC$ in order
  of decreasing size, we just have to see that the fundamental groups
  do not change by removing a non-trivial subgroup $P$ which is not elementary
  abelian of rank $1$ or $2$, and if elementary abelian of rank $3$
  not maximal in $\calC$.  However, this all follows from
  Proposition~\ref{pruninglemma}(2b)(4b) and van Kampen's theorem: If $P$ is not elementary
  abelian then $\bar \calC_{<P} =
  \calS_p(G)_{<P}$ is contractible via the standard contraction
  \eqref{std-contraction}. If $P$ is
  elementary abelian of rank at least $4$, then $\calS_p(G)_{<P}$ is
  simply connected as it is homotopy equivalent to a simply connected
  Tits building. And if $P$ is elementary 
  abelian of rank $3$ and not maximal in $\calC$ then $\bar
  \calC_{<P}$ is connected and is joined with a non-empty space giving
  something simply connected (cf.\ Lemma~\ref{connectivitylemma}).
\end{proof}

\subsection{Essential and strongly $p$--embedded subgroups: the connected components of $\calC$} \label{components-subsec}
We now describe the set of connected components of $\calC$ in more
detail, generalizing the description for  $\calS_p(G)$ due to Quillen \cite[\S5]{quillen78} \cite[Prop.~5.8]{grodal02},
explained in \eqref{redtoG_0}---we need this in
Section~\ref{fundgrp-sec} to get the results in their optimal form, and
it also has group theoretic significance, see Remark~\ref{cores}.

Fix a Sylow $p$--subgroup $S$ and set $G_{0,\calC} = \langle
  N_G(Q) | Q \leq S, Q \in \calC\rangle$ and $\calC_0 = \{ Q \in \calC  |Q \leq
  G_{0,\calC}\}$ as in
\eqref{pembeddedC}. 

\begin{prop}[Connected components of $\calC$]\label{conn-comp-C} Let $\calC$ be a collection of $p$--subgroups in
  $G$, closed under passage to Sylow and $p$--essential overgroups. Then 
$G \times_{G_{0,\calC}} |\calC_0|\xrightcong |\calC|$ is a
  $G$--equivariant isomorphism of simplicial sets, via
$(g,P_0\leq \cdots \leq P_n) \mapsto ({}^gP_0 \leq \cdots \leq {}^gP_n)$, and
$|\calC_0|$ is connected. On the set of components 
$$G/G_{0,\calC}\xrightcong \pi_0(|\calC|) \xleftcong
\pi_0(|\calC'|) \xleftcong G/G_{0,\calC'}$$
 with $\calC'$ the Sylow or $p$--essential subgroups of $\calC$.

Furthermore for a subgroup $H \leq G$ the following conditions are equivalent:
\begin{enumerate} 
\item\label{subconj}$G_{0,\calC}$ is subconjugate to $H$.
\item \label{conjugates-not-in-C} $p \nmid |G:H|$ and if, for any $g \in G$, $H
  \cap {}^g H$ contains an element of $\calC$, then $g \in H$.
\end{enumerate}
In particular $G_{0,\calC}$ is characterized as a minimal subgroup
satisfying \eqref{conjugates-not-in-C}, containing $S$.
\end{prop}
\begin{proof} 
Let us first show that the two conditions are equivalent. Suppose
that \eqref{subconj} is satisfied. We can without loss of generality
assume that $G_{0,\calC} \leq H$, since the condition on $H$
is conjugation invariant. Suppose now that
$Q \leq H \cap {}^g H$, and $Q \in \calC$. We want to show that $g
\in H$.  By Sylow's theorem in $H$, using that $p \nmid |G:H|$ we can
upon changing $g$ by an element in $H$ assume that $Q \leq S$. By
assumption ${}^{x^{-1}}Q \leq H$, and hence we can, again by Sylow's
theorem find $h \in H$ so that ${}^{hx^{-1}}Q \leq S$. Alperin's
fusion theorem, in the version of Goldschmidt--Miyamoto--Puig
\cite[Cor.~1]{miyamoto77} (see also \cite[\S10]{grodal02}), now says that we can find
$p$--essential subgroups $P_1, \ldots, P_r$, and
elements $g_i \in N_G(P_i)$, and $n \in N_G(S)$, with
$Q \leq P_1$ and  ${}^{g_{i}\cdots g_1}Q\leq P_{i+1}
\leq S$ for $i \leq r-1$, such that $hx^{-1}
= n g_r \cdots g_1$. However since the right-hand side is in
$G_{0,\calC}$ by assumption, this shows that $x \in H$ as wanted, so
\eqref{conjugates-not-in-C} holds.

Now suppose that $H$ is a subgroup satisfying
\eqref{conjugates-not-in-C}. By Sylow's theorem, we can change $H$ up
to conjugation so that $S \leq H$. We want to show that
$G_{0,\calC} \leq H$. In fact we will prove the stronger statement
that $G_S$, the stabilizer of $[S] \in
  \pi_0(\calC)$, is contained in $H$. As it is clear that $G_{0,\calC}
  \leq G_S$, this will show $\eqref{conjugates-not-in-C} \Rightarrow
  \eqref{subconj}$, and furthermore establish that $G_S = G_{0,\calC}$
  as claimed in the first part of the theorem. Hence suppose that $g \in G_S$, so
  that ${}^gS$ and $S$ lies in the same component. By definition there
  exists a sequence of Sylow $p$-subgroups $S_0, \ldots, S_r$, so that
  $S = S_0$, $S_r = {}^gS$ such that $S_{i} \cap S_{i+1}$ contains an
  element of $\calC$.  Choose $g_i \in G$ such that ${}^{g_i}S_{i} =
  S_{i+1}$, so that ${}^{g_{r-1}\cdots g_0}S ={}^gS$. If $S_i \leq H$,
  then $S_{i+1} \leq H$ and $g_i \in H$ by our assumption, so by
  induction we conclude that $g_{r-1}\cdots g_0 \in H$, and hence also
  $g \in H$ since the two elements differ by an element of $N_G(S)
  \leq H$.

That $G \times_{G_{0,\calC}} |\calC_0|\xrightcong
|\calC|$ now follows: The map is surjective since for $P_0 \leq \cdots
\leq P_n \in |\calC|_n$ we can, by Sylow's theorem, find $g \in G$ so ${}^{g}P_n
\leq S$, and hence $({}^gP_0 \leq \cdots \leq {}^gP_n) \in |\calC_0|_n$.
Furthermore if $ ({}^gP_0 \leq \cdots \leq {}^gP_n) =  ({}^{g'}P'_0
\leq \cdots \leq {}^{g'}P'_n)$ then ${}^{g'^{-1}g}P_0 = P'_0 \leq
G_{0,\calC}$, and likewise $P_0 \leq G_{0,\calC}$ by definition, i.e.,
$P_0$ and $P'_0$ lie in the same component so
${g'^{-1}g} \in G_S = G_{0,\calC}$, by the first part. In other words $gG_{0,\calC}
= g'G_{0,\calC}$ as wanted.

The last claim we need to justify is that $\pi_0(|\calC'|)
\xrightcong \pi_0(|\calC|)$ which follows from
Lemma~\ref{overgroupslemma} with $P = e$, except the degenerate case
where $e \in \calC$. But here it is also true: this is clear if also $e \in \calC'$
since both spaces are contractible, and if $e \not \in \calC'$ then
Proposition~\ref{pruninglemma}(1a) still says that we can add $e$ to
$\calC'$ without changing the number of components, which hence has to
be one.
\end{proof}

We now check that the two definitions of
$\calF$--essential from \S\ref{modelsF} agree (see also
\cite[Cor.~III.2]{puig76}).

\begin{lemma}\label{lem:Fessentialdef}
$W_0PC_G(P)/P$ is a proper subgroup of $W =N_G(P)/P$ if and only if
$P$ is $p$--centric and $\calS_p(N_G(P)/PC_G(P))$ is disconnected.
\end{lemma}

\begin{proof}
  Let $H$ denote the preimage of $W_0PC_G(P)/P$ in $N_G(P)$.
As $P \leq PC_G(P) \leq H \cap H^g$ for all $g \in N_G(P)$, Proposition~\ref{conn-comp-C}
  (applied to $\calC = \calS_p(N_G(P)/P)$ in $N_G(P)/P$) shows that if $H \neq
  N_G(P)$, then $|PC_G(P):P|$ is prime to $p$, as $ |P : H \cap H^g|$ is, so $P$ is $p$--centric.

  If $P$ is $p$--centric then $C_G(P) \cong Z(P) \times R$, for $R$ a
  $p'$--group. Note that
$|\calS_p(N_G(P)/P)|/R \cong |\calS_p(N_G(P)/PR)|$ (by
\cite[Prop.~5.7]{grodal02}), a space with set of components $W/W_0R$. Hence
$W_0R$ is a proper subgroup of $W$ if and only if
$|\calS_p(N_G(P)/PR)|$ is disconnected, showing the lemma.
 \end{proof}

\begin{rem}[Groups with a strongly $p$--embedded
  subgroup]\label{stronglyembeddedclassification} 
To use the theorems in this paper, it is useful to know
when  $G_0 =
\langle N_G(Q) | 1< Q \leq S\rangle$ is proper in $G$,
i.e., in group theoretic language, when $G$ contains a strongly $p$--embedded subgroup.
The answer to this question forms an important chapter in
the classification of finite simple groups. The following is a theorem
of Bender when $p=2$ \cite{bender71},
and only known as a consequence of the classification when $p$ is odd:
Either $\rk_p(G)=1$ and $G_0 = N_G(\Omega_p(Z(S)))$, with $G_0<G$ exactly when $O_p(G)
=1$, or $\rk_p(G)\geq2$,
$O_{p'}(G) \leq G_0$, $\bar G = G/O_{p'}(G)$ has a unique minimal
normal non-abelian simple subgroup $\bar K = F^*(\bar G)$, and $\bar
G/ \bar K
\leq \Out(\bar K)$. The group $\bar K$  is either a finite group of Lie type of
rank 1 (possibly twisted) in defining characteristic, $A_{2p}$ ($p\geq 3$), $(L_3(4),3)$,
$(M_{11,}3)$, $(Fi_{22},5)$, $(Mc,5)$, $(F_4(2)',5)$, or $(J_4,11)$.
See \cite[7.6.1,
7.6.2]{GLS98}  and also
\cite[\S5]{quillen78} \cite[(6.2)]{aschbacher93} for more details.
\end{rem}

\begin{rem} \label{cores} As noticed, when taking $\calC = \calS_p(G)$,
  Proposition~\ref{conn-comp-C} is a strong version of Quillen's
  \cite[Prop.~5.2]{quillen78}. When taking $\calC$ to be the collection of
  $p$--subgroups of $p$--rank at least $k$, $G_{0,\calC}
  =\Gamma_{k,S}(G)$, the $k$--generated $p$--core from finite group
  theory, and one also makes geometric and extends a standard characterization
  of it
  \cite[(46.4)]{aschbacher00} (see also
 \cite[Sec.~22]{GLS96} 
 \cite[Sec.~B.4]{ALSS11}). Groups with proper $2$--generated
 $2$--core were famously classified by Aschbacher \cite{aschbacher74}.
\end{rem}

\begin{rem} \label{endoessential-remark}
By examining the list in Remark~\ref{stronglyembeddedclassification} one sees that very often when $G_0 <
  G$, $H_1(G_0)_{p'} \twoheadrightarrow H_1(G)_{p'}$ is not injective,
  providing exotic Sylow-trivial modules via
  \eqref{classicalbounds} (take for instance $G =
  \SL_2(\F_{p^r})$ with $p^r \neq 2$). But it may also be injective:
Let $p=2$ and consider $K =
  \SL_2(\F_{2^r})$, $r>1$ odd, and let
  $C_r$ act on $K$ via field automorphisms. Set $G = K \semi C_r$. Then
  $K_0$ consists of upper triangular matrices of determinant one, and
  $G_0 = H_0 \semi C_r$. Hence $H_1(G_0)_{2'} \cong H_1(G)_{2'}
  \cong C_r$. Other examples may be constructed along these lines,
  though perhaps limited to small primes. (We are grateful to Ron
  Solomon for consultations on these points.) 
\end{rem}

\subsection{Proof of
  Proposition~\ref{pruninglemma}} \label{pruningproof-subsec}
We now prove Proposition~\ref{pruninglemma}, via general observations about links in preordered sets, an abstraction of observations in
\cite{dwyer98sharp,GS06}.
Let $\calX$ be a preordered set,  i.e., a small category with at most one
morphism between any two objects. Note that our spaces $E\bO_\calC$,
etc., are all examples of such. For $x \in \calX$, let $\bar x$ denote
the full subcategory of $\calX$ on
objects isomorphic to $x$, let $\calX_{<\bar x}$ denote the full
subcategory of $\calX$ on elements smaller than and not isomorphic to
$x$, and define $\calX_{>\bar x}$ similarly. For preordered sets
$\calX$ and $\calY$ the join $\calX \star \calY$  is the preordered
set obtained from the disjoint union of $\calX$ and $\calY$ by
adding a unique morphism from each
object in $\calX$ to each object in $\calY$. This has the property
that $|\calX \star \calY| \cong |\calX| \star |\calY|$.
The star $\sta_{\calX}(\bar x)$ is
the full subcategory of $\calX$ on objects which admit a morphism to or from $x$,
and the link $\link_{\calX}(\bar x)$ is the full subcategory on objects that admit a non-isomorphism
to or from $x$. Note that
$$\sta_\calX(\bar x) = \calX_{<\bar x} \star \bar x \star \calX_{>\bar
  x} \mbox{ and
} \link_\calX(\bar x) = \calX_{<\bar x} \star \calX_{>\bar x}$$
If $\calX$ has
a $G$--action, these are all $G_{\bar x}$--subcategories, where $G_{\bar
  x}$ is the stabilizer of $\bar x$ as a set, and furthermore
$\sta_{\calX}(\bar x)$ is $G_x$--contractible to $x$ (but generally
not $G_{\bar x}$--contractible).

\begin{prop} \label{linklemma} Suppose
 $\calX$ is a preordered set equipped with a $G$--action such that
 isomorphic objects are $G$--conjugate, and let $\calX'$ denote the subcategory of $\calX$
obtained by removing all $G$--conjugates of an element $x$.

\begin{enumerate}
\item \label{linksquare} There is a pushout square of $G$--spaces, which is
also a homotopy pushout square of $G$--spaces:
$$\xymatrix{ G \times_{G_{\bar x}} |\link_{\calX}(\bar x)| \ar[d] \ar[r] &
  |\calX'| \ar[d]\\
G \times_{G_{\bar x}} |\sta_{\calX}(\bar x)|  \ar[r] &
  |\calX|}$$
where  $\bar x$ denotes the subcategory of elements isomorphic to $x$ and $G_{\bar x}$ its
stabilizer as a set.

\item \label{Gsquare}
Assume in addition that the stabilizer $G_x$ of any point $x \in \calX$ is a normal
subgroup in the stabilizer of its isomorphism class $\bar x$. Then the square
$$\xymatrix{ G \times_{G_{\bar x}} (E(G_{\bar x}/G_x) \times
  |\link_{\calX}(\bar x)|) \ar[d] \ar[r] &
  |\calX'| \ar[d]\\
G \times_{G_{\bar x}} (E(G_{\bar x}/G_x) \times |\sta_{\calX}(\bar x)|)  \ar[r] &
  |\calX|}
$$
obtained by collapsing
$E(G_{\bar x}/G_x)$ and continuing as in \kern-0.5pt \eqref{linksquare} is again a pushout and homotopy pushout square of $G$--spaces,
and remains a homotopy pushout of $G$--spaces after collapsing $|\sta_{\calX}(\bar x)|$.
\end{enumerate}

In particular, under these assumptions:
\begin{enumerate} \addtocounter{enumi}{2}
\item \label{XmodG} On $G$--orbits there is a homotopy pushout square
$$\xymatrix{((|\calX_{<\bar x}| \star |\calX_{>\bar x}|)/G_x)_{hG_{\bar x}/G_x} \ar[d] \ar[r] &
  |\calX'|/G \ar[d]\\
B(G_{\bar x}/G_x)  \ar[r] &
  |\calX|/G}$$
\end{enumerate}
\end{prop}
\begin{proof}
For \eqref{linksquare}, notice that it is a pushout of $G$--spaces by the fact that $|\calX'|$ and
the image of $G \times_{G_{\bar x}}|\sta_\calX({\bar x})|$ cover $|\calX|$,
and the part of $G \times_{G_{\bar x}}|\sta_\calX({\bar x})|$ that maps to
$\calX'$ is the subspace $G \times_{G_{\bar x}}|\link_\calX({\bar x})|$, just by the
definitions and that isomorphic elements are $G$--conjugate.
It is homotopy pushout of $G$--spaces since
$G \times_{G_{\bar x}}|\link_\calX({\bar x})| \to G \times_{G_{\bar
    x}}|\sta_\calX({\bar x})|$ is an injective map of $G$--spaces. 

For \eqref{Gsquare} it is enough to see that the projection square
$$\xymatrix{ G \times_{G_{\bar x}} E(G_{\bar x}/G_x) \times |\link_{\calX}({\bar x})| \ar[d] \ar[r] &
   G \times_{G_{\bar x}} \times |\link_{\calX}({\bar x})|  \ar[d]\\
G \times_{G_{\bar x}} E(G_{\bar x}/G_x) \times |\sta_{\calX}({\bar x})|  \ar[r] &
 G \times_{G_{\bar x}} \times |\sta_{\calX}({\bar x})|}$$
 is a
pushout and homotopy pushout of $G$--spaces, by transitivity of
pushouts and homotopy
pushouts. 
To see this it is again enough to see that 
$$\xymatrix{ E(G_{\bar x}/G_x) \times |\link_{\calX}({\bar x})| \ar[d] \ar[r] &
   |\link_{\calX}({\bar x})|  \ar[d]\\
E(G_{\bar x}/G_x) \times |\sta_{\calX}({\bar x})|  \ar[r] &
 |\sta_{\calX}({\bar x})|}$$
is a pushout and homotopy pushout of $G_{\bar x}$--spaces, which we do by checking
on all fixed-points for $H \leq
G_{\bar x}$. For $H \leq G_x$ $E(G_{\bar x}/G_x)^H$ is contractible and
the claim is clear. (Note, we use that $G_{\bar x}/G_x$ is a group,
not just a coset, so the isotropy does not depend on the chosen point $x
\in \bar x$.) For $H \not \leq G_x$, the spaces on the left are empty,
and $\link_\calX({\bar x})^H = \sta_{\calX}({\bar x})^H$, so it is
also a pushout and  homotopy
pushout square in that case.
For the purposes of having a {\em homotopy} pushout square of $G$--spaces,
we can replace $\sta_\calX(\bar x)$ by a point, since
$\sta_\calX(\bar x)$ is $G_x$--contractible, and hence  $E(G_{\bar
  x}/G_x) \times |\sta_{\calX}({\bar x})| \to E(G_{\bar x}/G_x)$ a
$G_{\bar x}$--homotopy equivalence.

For \eqref{XmodG}, note that since  $G \times_{G_{\bar x}} E(G_{\bar x}/G_x) \times
|\link_{\calX}({\bar x})|  \to G \times_{G_{\bar x}} E(G_{\bar x}/G_x) \times
|\sta_{\calX}({\bar x})| $ is a cofibration of $G$--spaces, passing to
$G$--orbits in \eqref{Gsquare} produces  the homotopy pushout square
in \eqref{XmodG}.
\end{proof}

\begin{rem}\label{replacement-rem}
In Proposition~\ref{linklemma}\eqref{Gsquare}\eqref{XmodG} we can replace $|\link_\calX(\bar x)| = |\calX_{<\bar x}|
\star  |\calX_{>\bar x}|$ by an $G_x$--equivalent space through an $G_{\bar
  x}$-equivariant map, without changing
the conclusion, as any group element in $G_{\bar x} \setminus G_x$
acts freely on $E(G_{\bar x}/G_x)$.
  \end{rem}

\begin{proof}[Proof of Proposition~\ref{pruninglemma}]
This will be applications of Proposition~\ref{linklemma}: (1a) is Proposition~\ref{linklemma}\eqref{Gsquare}
with $\calX = \calC$, noting that in this case $\calX$ is a poset
and $G_{\bar x} = G_x$. (1b) is
Proposition~\ref{linklemma}\eqref{XmodG}.
For (2a) cross the diagram from (1a) with $EG$, noting that $EG$ as an
$N$--space is equivalent to $EN$. For (2b) take $G$--orbits and note that $|\calC|_{hG}$ identifies with $|\calT_\calC|$ by
Lemma~\ref{thomasonlemma}. (Alternatively,
more directly take $\calX = E\calT_\calC = e \downarrow \iota$, the
undercategory for $\iota\co
\calT_{\calC} \to \calT_{\calC \cup \{e\}}$.)
For (3a) and (3b) take $\calX = E\bO_\calC$, so for $x =
(G/Q,y)$, $G_{x} = G_{y}$, and $G_{\bar x} = N_G(G_y)$. The diagrams
now follow from
Proposition~\ref{linklemma}\eqref{Gsquare}--\eqref{XmodG} using that
we have an $N_G(P)$-equivariant map
$|(E\bO_\calC)_{>P}| \to |\calC_{>P}|$ which is a $P$--homotopy
equivalence by \eqref{fixed-points}, so we can replace
$(E\bO_\calC)_{>P}$ by $\calC_{>P}$ by Remark~\ref{replacement-rem}.
Finally for (4a-b) 
take $\calX = E\A_\calC$
and again use
Proposition~\ref{linklemma}\eqref{Gsquare}--\eqref{XmodG} (where
$G_x = C_G(i(Q))$ and $G_{\bar x} = N_G(i(Q))$, for
$x = {(i\co Q \to G)}$), 
together with the simplification that $|(E\A_\calC)_{<P}| \to |\calC_{<P}|$ is an
$C_G(P)$--homotopy equivalence (which holds as $|(E\A_\calC)_{<P})|^H \xrightsimeq
|\calC_{<P}|$  by \eqref{fixed-points} when $H \leq C_G(P)$, or
equivalently $P \leq C_G(H)$). 
\end{proof}

\bibliographystyle{alpha}
\bibliography{../bib/master}

\end{document}